\documentclass[preprint,11pt,number,sort]{elsarticle}
\usepackage[a4paper,margin=2.5cm]{geometry}

\usepackage{amsmath,amssymb,amsthm,mathtools}
\usepackage{mathrsfs}
\usepackage{bm}
\usepackage{booktabs}
\usepackage{tabularx}
\usepackage{microtype}
\usepackage{xcolor}
\usepackage{graphicx}
\usepackage{float}
\usepackage{longtable}
\usepackage{pdflscape}
\usepackage[section]{placeins}
\usepackage{xurl}
\usepackage{subcaption}
\usepackage{comment}

\definecolor{citationblue}{RGB}{0,70,140}

\usepackage[
    colorlinks=true,
    citecolor=citationblue,
    linkcolor=black,
    urlcolor=citationblue,
    pdfauthor={Matteo Parsani, Rasha AlJahdali, Lisandro Dalcin},
    pdftitle={Isothermal walls in the compressible Navier--Stokes equations:
affine entropy, ballistic energy, and entropy-stable discrete formulations},
]{hyperref}

\usepackage{orcidlink}

\journal{Journal of Computational Physics}

\newcommand{\bQ}{\bm{Q}}
\newcommand{\bW}{\bm{W}}
\newcommand{\bWT}{\bm{W}_{T_w}}
\newcommand{\bu}{\bm{u}}
\newcommand{\bn}{\bm{n}}
\newcommand{\be}{\bm{e}}
\newcommand{\bqheat}{\bm{q}_{h}}
\newcommand{\btau}{\bm{\tau}}
\newcommand{\bFI}[1]{\bm{F}^{(I)}_{#1}}
\newcommand{\bFV}[1]{\bm{F}^{(V)}_{#1}}
\newcommand{\FI}[1]{F_{#1}}
\newcommand{\FIT}[1]{F_{T_w,#1}}
\newcommand{\R}{\mathbb{R}}
\newcommand{\dd}{\,\mathrm{d}}
\newcommand{\p}{\partial}
\newcommand{\Div}{\nabla\!\cdot}
\newcommand{\grad}{\nabla}
\newcommand{\transpose}{\mathsf{T}}
\newcommand{\qnormal}{q_n}
\newcommand{\wall}{\Gamma_w}
\newcommand{\Tw}{T_w}
\newcommand{\ST}{S_{T_w}}
\newcommand{\Bcal}{\mathscr{B}}
\newcommand{\Dcal}{\mathscr{D}}
\newcommand{\psol}{p_{\mathrm{s}}}
\newcommand{\Nsol}{N_{\mathrm{s}}}

\newtheorem{definition}{Definition}[section]
\newtheorem{lemma}[definition]{Lemma}
\newtheorem{theorem}[definition]{Theorem}
\newtheorem{corollary}[definition]{Corollary}
\newtheorem{remark}[definition]{Remark}

\begin{document}

\begin{frontmatter}

\title{Isothermal walls in the compressible Navier--Stokes equations:
affine entropy, ballistic energy, and entropy-stable discrete formulations}

\author[KAUSTCEMSE,KAUSTPSE]{Matteo Parsani\corref{cor1}\orcidlink{0000-0001-7300-1280}}
\ead{matteo.parsani@kaust.edu.sa}

\author[KAUSTCEMSE]{Rasha AlJahdali\orcidlink{0000-0002-1646-7196}}
\ead{rasha.aljahdali@kaust.edu.sa}

\author[KAUSTCEMSE]{Lisandro Dalcin\orcidlink{0000-0001-8086-0155}}
\ead{lisandro.dalcin@kaust.edu.sa}

\cortext[cor1]{Corresponding author.}

\address[KAUSTCEMSE]{Computer, Electrical and Mathematical Sciences and Engineering
Division, King Abdullah University of Science and Technology,
Thuwal 23955-6900, Saudi Arabia}

\address[KAUSTPSE]{Physical Science and Engineering Division, King Abdullah
University of Science and Technology, Thuwal 23955-6900, Saudi Arabia}

\begin{abstract}
We develop entropy-stable isothermal-wall conditions for the compressible
Navier--Stokes equations at continuous and semi-discrete levels. For a
stationary, impermeable, no-slip wall at constant positive temperature, adding
total-energy density divided by wall temperature to the canonical convex
entropy cancels the wall contribution without suppressing heat transfer. This
affine modification preserves the entropy Hessian, relative entropy, and
volume entropy production. It is proportional to total ballistic-energy
density and yields a conditional $L_2$-type bound based on the initial field
under the stated assumptions. Two wall formulations use a common discontinuous
Galerkin discretization with the summation-by-parts property and simultaneous
approximation terms. Both give zero adapted wall contribution without a
penalty and a nonpositive contribution with it. The complete-residual
formulation penalizes the full wall-data residual and balances fluid energy
through its complete numerical wall flux. The Fourier-energy-exact formulation
combines an entropy-neutral skew correction with a momentum-projected penalty
so that the complete numerical outward energy flux equals the one-sided
numerical Fourier heat flux at every wall node and spatial resolution.
Pointwise tests verify the local identities. Manufactured-solution studies with entropy-stable interior coupling
demonstrate convergence rates consistent with order $\psol+1$ in the
discrete $L_2$ norm, where $\psol$ is the solution polynomial degree.
Thermal-relaxation calculations verify the semi-discrete energy and
adapted-entropy balances. Three-dimensional calculations demonstrate
balance closure in a closed domain, convergence of curved-wall traces,
and applicability to a complex supersonic separated flow.
\end{abstract}

\begin{keyword}
Compressible Navier--Stokes equations \sep
Isothermal wall \sep
Affine entropy \sep
Ballistic energy \sep
Discontinuous Galerkin methods \sep
Summation-by-parts operators \sep
Simultaneous approximation terms
\end{keyword}

\end{frontmatter}

\section{Introduction}
\label{sec:introduction}

Convex entropy pairs underpin admissibility, symmetrization, and stability
estimates for nonlinear systems of conservation laws, from the foundations
of Lax, Friedrichs and Lax, Godunov, and Harten to the broader theory of
Dafermos, Smoller, and LeFloch
\cite{Lax1957,FriedrichsLax1971,Godunov1961,Harten1983,
Dafermos2016,Smoller1994,LeFloch2002}. For a sufficiently smooth, strictly
convex entropy, a global bound and a uniform positive Hessian lower bound
provide a priori $L_2$-type control of the conservative variables and a
standard nonlinear-stability mechanism
\cite{Dafermos1979,Dafermos2016,Smoller1994}. Tadmor's framework carries these 
entropy principles into numerical
discretization through the construction of entropy-conservative and
entropy-stable fluxes \cite{Tadmor1987,Tadmor2003}. For the compressible
Navier--Stokes equations considered here, both the continuous analysis and
its discrete counterpart start from the canonical convex entropy given by
the negative thermodynamic entropy density~\cite{HughesFrancaMallet1986}.

High-order computational fluid dynamics (CFD) offers greater accuracy per
degree of freedom with lower numerical dissipation and dispersion for
unsteady and wave- or vortex-dominated flows
\cite{WangEtAl2013HighOrderCFD}.
Nevertheless, production CFD relies heavily on mature lower-order methods
whose robustness, geometric flexibility, and industrial workflows remain
difficult to match. The National Aeronautics and Space Administration
(NASA) CFD Vision 2030 study and its 2024 assessment identify robust, efficient,
variable-order accurate, and scalable algorithms as enabling predictive unsteady and
scale-resolving simulations
\cite{SlotnickEtAl2014CFDVision2030,Malik2024CFDVision2030}.

Advances toward robust high-order methods include energy-stable
flux-reconstruction (FR) schemes~\cite{VincentCastonguayJameson2011};
entropy-stable summation-by-parts (SBP) finite-difference methods
\cite{Fisher2012PhD,FisherCarpenter2013}; the SBP interpretation and energy-stable split
forms of collocated discontinuous Galerkin (DG) spectral-element
methods~\cite{Gassner2013}; entropy-stable DG schemes
\cite{CarpenterEtAl2014}, viscous interface coupling
\cite{ParsaniInterface2015}, entropy-stable split-form DG operators
\cite{GassnerWintersKopriva2016}, and entropy-stable SBP discretizations
on general curved elements~\cite{CreanEtAl2018}. Further developments include
$hp$-refinement
\cite{FriedrichEtAl2018,DelReyFernandezEtAl2020HP}, scalable entropy-stable collocated DG
solvers~\cite{ParsaniEtAl2021,KraisEtAl2020SplitALE}, other high-order DG frameworks
\cite{KraisEtAl2021FLEXI,FerrerEtAl2023HORSES3D}, and nonlinearly
entropy-stable FR formulations
\cite{CicchinoNadarajahDelReyFernandez2022}.

Physical boundary closures remain essential: an incompatible closure can
destroy the global entropy estimate and destabilize the computation even
when volume and interior-interface operators are entropy-stable
\cite{ParsaniCarpenterNielsen2015,RojasEtAl2021}. At a conforming interior
face, paired contributions share a consistently oriented normal, so
conservative terms telescope and the remaining entropy contribution is zero
or a controlled nonpositive quadratic term
\cite{CarpenterEtAl2014,ParsaniInterface2015}. A physical boundary has no
neighboring contribution: only the interior state is available. A ghost
state or numerical boundary value must use the prescribed data and make the
one-sided contribution vanish or have the required sign.

Wall developments include the entropy-stable Euler conditions of Sv{\"a}rd
and {\"O}zcan~\cite{SvardOzcan2014}, the entropy-stable no-slip
Navier--Stokes conditions first reported by Parsani, Carpenter, and Nielsen
\cite{ParsaniCarpenterNielsen2014NASA,ParsaniCarpenterNielsen2015}, the
conservative and entropy-stable framework of Dalcin et al.
\cite{DalcinEtAl2019}, and further analysis of full no slip and
velocity-gradient control by Sv{\"a}rd, Carpenter, and Parsani~\cite{SvardCarpenterParsani2018}.
Chan, Lin, and Warburton later presented modal-DG treatments of adiabatic
data, closely related in structure to
Refs.~\cite{ParsaniCarpenterNielsen2015,DalcinEtAl2019}, and prescribed
isothermal data~\cite{ChanLinWarburton2022}. Their semi-discrete isothermal
boundary term reproduces the nonhomogeneous canonical heat-entropy exchange,
consistent with the heat-entropy-transfer analysis of
Ref.~\cite{ParsaniCarpenterNielsen2015}.

For the canonical mathematical entropy function
\begin{equation*}
  S=-\rho s,
\end{equation*}
where $\rho$ is the mass density and $s$ the specific thermodynamic entropy,
a prescribed-temperature wall contributes $q_n/T_w$ through Fourier
conduction \cite{ParsaniCarpenterNielsen2015,DalcinEtAl2019,ChanLinWarburton2022}. 
Here, $T_w>0$ is the prescribed wall temperature and $q_n$ is the
outward normal conductive heat flux. Unlike at an adiabatic wall, this
exchange has a solution-dependent sign, so $S$ does not generally satisfy a
homogeneous non-increase estimate. The isothermal construction of Chan, Lin, and Warburton
reproduces this exchange, not a homogeneous boundary-entropy inequality.
To the best of our knowledge, a
homogeneous wall estimate for a stationary no-slip wall maintained at a
prescribed temperature has not previously been formulated through an adapted
convex entropy together with a compatible semi-discrete wall treatment.

For heat-conducting compressible fluids, Dirichlet temperature data are
treated through a ballistic-energy functional associated with the prescribed
reservoir temperature
\cite{Feireisl2012,FeireislNovotny2017,ChaudhuriFeireisl2022},
\begin{equation*}
  \rho E-T_w\rho s,
\end{equation*}
where $E$ is the specific total energy. We connect this construction with
convex-entropy theory by identifying the affine representative selected by
the isothermal-wall balance,
\begin{equation*}
  S_{T_w}=S+\frac{\rho E}{T_w}.
\end{equation*}
We derive the continuous constant-temperature wall balance and prove that
$S_{T_w}$ preserves the canonical entropy Hessian, strict convexity, viscous
symmetrizer, and volume entropy production, while its boundary contribution
vanishes at a stationary no-slip wall with $T=T_w$. Therefore, its homogeneous estimate
falls within standard convex-entropy theory~\cite{Dafermos2016,Smoller1994}.

Within a common DG--SBP framework with simultaneous-approximation-term (SAT)
interface and boundary couplings, we introduce two isothermal-wall SAT
formulations. The complete-residual SAT (CR--SAT) extends the coupled
conservative/auxiliary-gradient architecture of
Refs.~\cite{ParsaniCarpenterNielsen2015,DalcinEtAl2019} through an admissible
wall state and an unprojected penalty on the complete wall-data residual.
The Fourier-energy-exact SAT (FEE--SAT) introduces an entropy-neutral skew
correction and a momentum-projected penalty so that the numerical outward
total-energy flux equals the one-sided numerical Fourier heat flux at finite
resolution. Both formulations are boundary-entropy-conservative or
boundary-entropy-stable with respect to $S_{T_w}$, depending on whether wall
dissipation is disabled or enabled. Pointwise tests verify the local wall identities.
In one dimension,
a matched manufactured-solution study compares the spatial accuracy of
both formulations under identical interior discretizations, while
thermal-relaxation tests verify the assembled total-energy and
adapted-entropy balances. Three-dimensional calculations examine
adapted-entropy balance closure in a closed domain, curved-wall trace
convergence in an annular pipe, and operation in supersonic separated
flow past an inclined cube.

Although presented in a collocated DG setting, the construction requires
only compatible entropy variables, numerical fluxes, viscous-gradient data,
and SBP volume and face operators. This operator-level structure facilitates
its adaptation to compatible finite-difference, finite-volume,
finite-element, modal DG, and FR discretizations, provided that their volume
and interface operators reproduce the same discrete entropy identities.

\textit{Note that the interior-interface couplings and boundary treatments are written
in SBP--SAT operator form, with weak interface coupling and boundary
enforcement supplied by SAT corrections. In an SBP-compatible weak DG
formulation, the same contributions appear as numerical surface-flux
corrections or lifting terms after discrete integration by parts~\cite{Gassner2013,GassnerEtAl2018BR1}.
Consequently, the proposed wall formulations can be incorporated directly
into SBP-compatible discontinuous Galerkin spectral-element methods (DGSEM),
even when those methods are derived and presented in weak DG form rather
than in explicit SAT terminology.}

The paper is organized as follows. Section~\ref{sec:equations}
introduces the governing equations and canonical entropy structure.
Section~\ref{sec:canonical-wall} derives the stationary
isothermal-wall entropy and energy exchanges.
Section~\ref{sec:adapted-entropy} develops the affine entropy, its a priori
consequence, and ballistic-energy interpretation.
Section~\ref{sec:semidiscrete-wall} presents the common discretization and
two wall SAT formulations, and Section~\ref{sec:results} reports their
algebraic, transient, and convergence verification. Supporting derivations
and detailed data appear in the appendices and supplementary material.

\section{Governing equations and canonical entropy structure}
\label{sec:equations}

Let $\Omega\subset\R^d$, with $d\in\{2,3\}$, be a fixed bounded domain
with sufficiently smooth boundary $\partial\Omega$. The position vector is
$\bm x=(x_1,\ldots,x_d)^{\transpose}$, $t\in[0,t_f]$ denotes time, and
$\bn=(n_1,\ldots,n_d)^{\transpose}$ is the outward unit normal. We set
$\p_t=\p/\p t$ and $\p_i=\p/\p x_i$, use the Einstein convention for
repeated spatial indices $i,j,k\in\{1,\ldots,d\}$, and let
$(\cdot)^{\transpose}$ and $\lvert\cdot\rvert$ denote transpose and the
Euclidean norm, respectively.

\subsection{Compressible Navier--Stokes equations}
\label{subsec:conservation-laws}

In the absence of body forces, volumetric heat sources, and radiative
effects, we consider the following initial-boundary-value problem for the
compressible Navier--Stokes--Fourier equations:
\begin{subequations}
\label{eq:cns-ibvp}
\begin{align}
  \p_t\bQ+\p_i\bFI{i}
  &=
  \p_i\bFV{i},
  &&
  (\bm x,t)\in\Omega\times(0,t_f],
  \label{eq:cns}
  \\
  \bQ(\bm x,0)
  &=
  \bQ_0(\bm x),
  &&
  \bm x\in\Omega,
  \label{eq:cns-initial-condition}
  \\
  \mathfrak B
  \bigl(
    \bQ,\grad\bQ;\bn,\bm x,t
  \bigr)
  &=
  \bm g_{\partial\Omega}(\bm x,t),
  &&
  (\bm x,t)\in\partial\Omega\times(0,t_f].
  \label{eq:cns-boundary-operator}
\end{align}
\end{subequations}
Here, $\bQ_0$ is the prescribed initial state, while $\mathfrak B$ denotes
the collection of boundary operators on the different portions of
$\partial\Omega$, with corresponding data $\bm g_{\partial\Omega}$. Its
possible dependence on $\grad\bQ$ accommodates viscous and thermal boundary
conditions. The wall operator required in this work is specified in
Section~\ref{sec:canonical-wall}, whereas the remaining boundary operators are
problem dependent and enter the global entropy balances through their
boundary contributions. 

The conservative state is
\begin{equation}
  \bQ
  =
  \begin{pmatrix}
    \rho\\
    \bm m\\
    \rho E
  \end{pmatrix}
  =
  \begin{pmatrix}
    \rho\\
    \rho\bu\\
    \rho E
  \end{pmatrix}
  \in\R^{d+2},
  \qquad
  E=e+\frac12\lvert\bu\rvert^2.
  \label{eq:conservative-state}
\end{equation}
Here, $\rho$ is the mass density, $\bu$ is the velocity,
$\bm m=\rho\bu$ is the momentum density, and $e$ and $E$ are the specific
internal and total energies, respectively. The physical inviscid and viscous 
fluxes in the $x_i$ direction,
denoted by $\bFI{i}$ and $\bFV{i}$, respectively, are given by
\begin{equation}
  \begin{aligned}
    \bFI{i}
    &=
    \begin{pmatrix}
      \rho u_i\\
      \rho u_i\bu+p\be_i\\
      (\rho E+p)u_i
    \end{pmatrix},
    &
    \bFV{i}
    &=
    \begin{pmatrix}
      0\\
      \btau\be_i\\
      \bu\cdot(\btau\be_i)-q_{h,i}
    \end{pmatrix}.
  \end{aligned}
  \label{eq:physical-fluxes}
\end{equation}
Here, $p$ is the thermodynamic pressure, $\be_i$ is the $i$th Cartesian unit
vector, $\btau=[\tau_{ij}]_{i,j=1}^d$ is the viscous stress tensor, and
$\bqheat=(q_{h,1},\ldots,q_{h,d})^{\transpose}$ is the conductive heat-flux
vector. The mass component of $\bFV{i}$ vanishes because the classical compressible Navier--Stokes model
contains no diffusive mass flux.

For a calorically perfect ideal gas,
\begin{equation}
  \begin{gathered}
    p=\rho RT=(\gamma-1)\rho e,
    \qquad
    e=c_vT,
    \qquad
    h=e+\frac{p}{\rho}=c_pT,
    \\
    c_p-c_v=R,
    \qquad
    \gamma=\frac{c_p}{c_v}>1.
  \end{gathered}
  \label{eq:ideal-gas-closure}
\end{equation}
Here, $T$ is the absolute temperature, $h$ is the specific enthalpy,
$R>0$ is the specific gas constant, and the constants $c_v>0$ and $c_p>0$
are the mass-specific heat capacities at constant volume and pressure,
respectively.

The specific thermodynamic entropy satisfies
\begin{equation}
  \begin{aligned}
    T\,\mathrm d s
    &=
    \mathrm d e
    +p\,\mathrm d\!\left(\frac1\rho\right),
    \\
    s(\rho,T)
    &=
    s_{\mathrm{ref}}
    +c_v\log\!\left(\frac{T}{T_{\mathrm{ref}}}\right)
    -R\log\!\left(\frac{\rho}{\rho_{\mathrm{ref}}}\right).
  \end{aligned}
  \label{eq:thermodynamic-entropy}
\end{equation}
Here, $(\rho_{\mathrm{ref}},T_{\mathrm{ref}})$ is any fixed positive
thermodynamic reference state, and
$s_{\mathrm{ref}}=s(\rho_{\mathrm{ref}},T_{\mathrm{ref}})$. A consistent
change of reference state leaves $s(\rho,T)$ unchanged. Shifting the entropy
zero by a constant $C_s\in\R$ replaces $s$ by $s+C_s$ and $S=-\rho s$
by $S-C_s\rho$. Thus, $\mathrm d s$, the Gibbs relation, the entropy
Hessian, relative entropy, and entropy production remain unchanged. The
reference-state and entropy-zero conventions are summarized
in~\ref{app:thermodynamics}.

The thermodynamically admissible state set is
\begin{equation}
  \mathcal A
  =
  \left\{
    \bQ
    =
    \begin{pmatrix}
      \rho,&\bm m^{\transpose},&\rho E
    \end{pmatrix}^{\transpose}
    \in\R^{d+2}:
    \rho>0,
    \quad
    \rho E-\frac{\lvert\bm m\rvert^2}{2\rho}>0
  \right\}.
  \label{eq:admissible-set}
\end{equation}
For a calorically perfect ideal gas, $\bQ\in\mathcal A$ is equivalent to
$\rho>0$ and $T>0$. We assume that
$\bQ_0(\bm x)\in\mathcal A$ and consider sufficiently smooth classical
solutions with values in $\mathcal A$, so that the chain rule, integration by
parts, and the required boundary values are well defined. Positivity
preservation is a separate challenge and is not addressed herein.
The equivalent positivity characterization can be found
in~\ref{app:thermodynamics}. See also
Refs.~\cite{LandauLifshitz1987,Feireisl2007,FeireislNovotny2017}.

Let $\bm I$ be the $d\times d$ identity tensor. For an isotropic Newtonian,
heat-conducting fluid,
\begin{equation}
  \begin{aligned}
    \bm D(\bu)
    &=
    \frac12\left(\grad\bu+(\grad\bu)^{\transpose}\right),
    &
    \bm D^0(\bu)
    &=
    \bm D(\bu)-\frac1d(\Div\bu)\bm I,
    \\
    \btau
    &=
    2\mu(T)\bm D^0(\bu)+\zeta(T)(\Div\bu)\bm I,
    &
    \bqheat
    &=
    -\kappa(T)\grad T.
  \end{aligned}
  \label{eq:constitutive-laws}
\end{equation}
Here, $\bm D(\bu)$ is the strain-rate tensor, and $\bm D^0(\bu)$ is its
trace-free part. The factor $1/d$ ensures
$\operatorname{tr}\bm D^0=0$, since
$\operatorname{tr}\bm D=\Div\bu$ and $\operatorname{tr}\bm I=d$.
Thus, $q_{h,i}=-\kappa(T)\p_iT$. The dynamic shear viscosity $\mu$, bulk
viscosity $\zeta$, and thermal conductivity $\kappa$ satisfy
\begin{equation}
  \mu(T)>0,
  \qquad
  \zeta(T)\geq0,
  \qquad
  \kappa(T)>0.
  \label{eq:transport-assumptions}
\end{equation}
These assumptions imply nonnegative viscous and thermal entropy
production. The conversion to the usual second coefficient of viscosity,
Stokes' hypothesis, and the resulting entropy-production identity are
reported in~\ref{app:constitutive-details}.

\begin{remark}\label{re:dimension}
All continuous results are formulated for $d\in\{2,3\}$. The
one-dimensional calculations reported later are spatial reductions of the
three-dimensional equations: the solution depends on one Cartesian
coordinate and has zero transverse velocity components. They are not obtained
by setting $d=1$ in the constitutive laws and retain the three-dimensional
normal viscous stress. In particular, this reduction gives
$\tau_{xx}=(4\mu/3+\zeta)\partial_xu$ and
$q_{h,x}=-\kappa\partial_xT$, as used in
Eq.~\eqref{eq:matched-mms-one-dimensional-constitutive-laws}.
\end{remark}

\subsection{Outward-normal flux convention}
\label{subsec:normal-convention}

On $\partial\Omega$, define
\begin{equation}
  u_n=\bu\cdot\bn,
  \qquad
  \p_nT=\bn\cdot\grad T,
  \qquad
  \qnormal
  =
  \bqheat\cdot\bn
  =
  -\kappa(T)\p_nT.
  \label{eq:normal-definitions}
\end{equation}
Because $\bn$ is outward, $\qnormal>0$ means that conductive heat is
transported out of the fluid domain. The outward normal physical fluxes are
\begin{align}
  \bm F_n^{(I)}
  =n_i\bFI{i}
  &=
  \begin{pmatrix}
    \rho u_n\\
    \rho u_n\bu+p\bn\\
    (\rho E+p)u_n
  \end{pmatrix},
  \label{eq:normal-inviscid-flux}
  \\
  \bm F_n^{(V)}
  =n_i\bFV{i}
  &=
  \begin{pmatrix}
    0\\
    \btau\bn\\
    \bu\cdot\btau\bn-\qnormal
  \end{pmatrix}.
  \label{eq:normal-viscous-flux}
\end{align}
Since the net conservative flux in Eq.~\eqref{eq:cns} is
$\bFI{i}-\bFV{i}$, its outward total-energy component is
\begin{equation}
  F_{E,n}
  =
  (\rho E+p)u_n
  -\bu\cdot\btau\bn
  +\qnormal.
  \label{eq:net-normal-energy-flux}
\end{equation}
At a stationary no-slip wall, $\bu=\bm0$ and $F_{E,n}=\qnormal$.
Consequently, $\qnormal>0$ decreases the total energy contained in the fluid
domain.

\subsection{Canonical mathematical entropy pair}
\label{subsec:canonical-entropy}

Following Ref.~\cite{HughesFrancaMallet1986}, we define the canonical
mathematical entropy density (also referred to as the canonical entropy
function) for the compressible Navier--Stokes equations by
\begin{equation}
  S(\bQ)=-\rho s,
  \label{eq:canonical-entropy}
\end{equation}
with entropy variables
\begin{equation}
  \bW
  =
  \nabla_{\bQ}S
  =
  \begin{pmatrix}
    \displaystyle
    \frac{h}{T}-s-\frac{\lvert\bu\rvert^2}{2T}
    \\[1.3ex]
    \displaystyle
    \frac{\bu}{T}
    \\[1.3ex]
    \displaystyle
    -\frac1T
  \end{pmatrix}.
  \label{eq:entropy-variables}
\end{equation}
For every $\bQ\in\mathcal A$, the Hessian
$\bm H_S=\nabla_{\bQ}^2S$ is symmetric positive definite.
Here and below, $\succ\bm0$ and $\succeq\bm0$ denote symmetric positive
definiteness and positive semidefiniteness, respectively. Consequently,
$S$ is strictly convex and the map $\bQ\mapsto\bW$ is one-to-one.
The entropy contraction and canonical balance are summarized
in~\ref{app:canonical-balance-derivation}. See also
Refs.~\cite{Harten1983,HughesFrancaMallet1986,Dafermos2016}.

The scalar inviscid entropy flux is
\begin{equation}
  \FI{i}(\bQ)=S(\bQ)u_i=-\rho s u_i,
  \qquad
  \left(\frac{\p\bFI{i}}{\p\bQ}\right)^{\transpose}\bW
  =\nabla_{\bQ}\FI{i},
  \qquad
  i=1,\ldots,d.
  \label{eq:entropy-pair-compatibility}
\end{equation}
This is the companion flux of the convex entropy pair in the sense of
Refs.~\cite{Lax1957,FriedrichsLax1971,Harten1983,Tadmor1987,Tadmor2003}.
The entropy potential and entropy-flux potential are
\begin{equation}
  \Phi
  =
  \bW^{\transpose}\bQ-S
  =
  \frac{p}{T}
  =
  \rho R,
  \qquad
  \Psi_i
  =
  \bW^{\transpose}\bFI{i}-\FI{i}
  =
  \frac{pu_i}{T}
  =
  \rho Ru_i.
  \label{eq:entropy-potentials}
\end{equation}
Viewed as functions of $\bW$, these potentials satisfy
$\nabla_{\bW}\Phi=\bQ$ and
$\nabla_{\bW}\Psi_i=\bFI{i}$. These standard identities follow from the
entropy-pair compatibility condition; see
Refs.~\cite{Godunov1961,Harten1983,Tadmor2003}. The potential $\Psi_i$
enters the Tadmor identity used below.

The viscous fluxes satisfy
\begin{equation}
  \begin{aligned}
    \bFV{i}
    &=
    \bm C_{ij}(\bQ)\p_j\bW,
    \\
    \bm C_{ij}
    &=
    \bm C_{ji}^{\transpose},
    &
    \sum_{i,j=1}^{d}
    \bm Z_i^{\transpose}\bm C_{ij}\bm Z_j
    &\geq0
  \end{aligned}
  \quad
  \text{for every }
  (\bm Z_1,\ldots,\bm Z_d)\in(\R^{d+2})^d.
  \label{eq:viscous-symmetrization}
\end{equation}
Hence, the block matrix $[\bm C_{ij}]_{i,j=1}^d$ is symmetric positive
semidefinite. This is the standard symmetric hyperbolic--parabolic entropy
structure of the compressible Navier--Stokes equations
\cite{HughesFrancaMallet1986}. The dimension-independent block
form used later is reported in~\ref{app:viscous-coefficient-matrices}.

\subsection{Canonical entropy balance}
\label{subsec:canonical-entropy-balance}

For a smooth solution, contraction of Eq.~\eqref{eq:cns} with
$\bW^{\transpose}$ gives
\begin{equation}
  \begin{aligned}
    \p_tS+\p_i\FI{i}
    &=
    \p_i\!\left(\frac{q_{h,i}}{T}\right)-\Dcal,
    \\
    \Dcal
    &=
    (\p_i\bW)^{\transpose}\bm C_{ij}\p_j\bW
    \\
    &=
    \frac{2\mu}{T}\bm D^0:\bm D^0
    +\frac{\zeta}{T}(\Div\bu)^2
    +\frac{\kappa}{T^2}\lvert\grad T\rvert^2
    \geq0.
  \end{aligned}
  \label{eq:local-canonical-entropy}
\end{equation}
Here, $\bm A:\bm B=A_{ij}B_{ij}$ denotes the Frobenius inner product.
Integration over $\Omega$ yields
\begin{equation}
  \frac{\dd}{\dd t}
  \int_\Omega S\,\dd\Omega
  =
  \int_{\partial\Omega}
  \left(
    \frac{\qnormal}{T}-Su_n
  \right)
  \dd\Gamma
  -
  \int_\Omega\Dcal\,\dd\Omega.
  \label{eq:global-canonical-entropy}
\end{equation}
Therefore, the pointwise canonical boundary contribution is
\begin{equation}
  \Bcal_S
  =
  \bW^{\transpose}\bm F_n^{(V)}-F_n
  =
  \frac{\qnormal}{T}-Su_n,
  \qquad
  F_n=n_i\FI{i}=Su_n.
  \label{eq:boundary-entropy-contribution}
\end{equation}
This is the continuous boundary quantity that the numerical wall treatment
must reproduce and bound. The entropy contraction and the sign of
the viscous quadratic are summarized in~\ref{app:canonical-balance-derivation}.

\section{Canonical entropy and total-energy balances at a stationary
isothermal wall}
\label{sec:canonical-wall}

Let $\wall\subseteq\partial\Omega$ denote a stationary solid wall held at 
constant temperature $\Tw>0$. On $\wall$, the wall
component of the boundary operator in
Eq.~\eqref{eq:cns-boundary-operator} is
\begin{equation}
  \mathfrak B_w(\bQ;\Tw)
  =
  \begin{pmatrix}
    \bu\\
    T-\Tw
  \end{pmatrix}
  =
  \bm0,
  \qquad
  (\bm x,t)\in\wall\times(0,t_f].
  \label{eq:isothermal-wall-data}
\end{equation}
The velocity component imposes no slip, and the temperature component imposes
the prescribed wall temperature. Because the wall is stationary, $\bu=\bm0$
also implies impermeability, $u_n=0$, on $\wall$. For a classical solution, 
the initial data satisfy
\begin{equation*}
  \mathfrak B_w(\bQ_0;\Tw)=\bm0
  \qquad
  \text{on }\wall.
\end{equation*}

The Dirichlet condition $T=\Tw$ fixes the boundary temperature but does
not prescribe $\p_nT$. The outward conductive heat flux
\begin{equation*}
  \qnormal=-\kappa(\Tw)\p_nT
\end{equation*}
is therefore determined by the solution rather than prescribed as wall
data.

\begin{theorem}[Canonical entropy contribution and total-energy flux]
\label{thm:canonical-isothermal-wall-fluxes}
For a sufficiently smooth solution of Eq.~\eqref{eq:cns} that satisfies
\eqref{eq:isothermal-wall-data}, the canonical mathematical-entropy boundary
contribution and the net outward total-energy flux satisfy, pointwise on
$\wall$,
\begin{equation}
  \left.\Bcal_S\right|_{\wall}
  =
  \frac{\qnormal}{\Tw},
  \qquad
  \left.F_{E,n}\right|_{\wall}
  =
  \qnormal.
  \label{eq:canonical-and-energy-wall-fluxes}
\end{equation}
\end{theorem}

\begin{proof}
Equation~\eqref{eq:boundary-entropy-contribution} gives
$\Bcal_S=\qnormal/T-Su_n$. Impermeability gives $u_n=0$, and the isothermal
condition gives $T=\Tw$, proving the first identity. For the total-energy
flux, Eq.~\eqref{eq:net-normal-energy-flux} shows that $u_n=0$ eliminates the
advective contribution, while the full no-slip condition $\bu=\bm0$
eliminates the viscous mechanical-work contribution. Hence,
$F_{E,n}=\qnormal$. These continuous wall identities are consistent with
Refs.~\cite{ParsaniCarpenterNielsen2015,DalcinEtAl2019}. The distinct roles
of no penetration and no slip are detailed in~\ref{app:section3-details}.
\end{proof}

Equation~\eqref{eq:canonical-and-energy-wall-fluxes} expresses the
open-system heat-entropy exchange associated with the canonical pair
$(S,\FI{i})$ \cite{ParsaniCarpenterNielsen2015}. Since $\Tw>0$ and $\qnormal$ may have either sign, the
isothermal wall data do not enforce a zero or nonpositive canonical
boundary contribution. Therefore, the canonical entropy balance alone
does not yield a homogeneous boundary-entropy estimate for these wall
data.

In the integral total-energy balance, an outward flux enters with a minus
sign. After multiplication by $1/\Tw$, the wall contribution to the rate of
change of $\int_\Omega(\rho E/\Tw)\,\dd\Omega$ has the pointwise integrand
$-F_{E,n}/\Tw=-\qnormal/\Tw$. Adding this term to the canonical entropy
boundary contribution gives
\begin{equation}
  \left.
  \left(
    \Bcal_S-\frac{F_{E,n}}{\Tw}
  \right)
  \right|_{\wall}
  =
  0.
  \label{eq:canonical-energy-wall-cancellation}
\end{equation}
Thus, the solution-dependent wall heat flux cancels exactly in this
combination. At this stage,
Eq.~\eqref{eq:canonical-energy-wall-cancellation} is only an identity obtained
by combining the canonical entropy and total-energy balances.
Section~\ref{sec:adapted-entropy} incorporates this combination into a
wall-temperature-adapted affine entropy representative
\cite{LeFloch2002,Dafermos2016,Feireisl2012}.

\section{Wall-temperature-adapted affine entropy and entropy stability}
\label{sec:adapted-entropy}

Equation~\eqref{eq:canonical-energy-wall-cancellation} identifies a linear
combination of the canonical entropy and total-energy balances for which
the isothermal-wall contribution vanishes. Because the total-energy
density $\rho E$ is a conserved component of $\bQ$ and $\Tw$ is constant,
this combination can be represented by an affine modification of the
canonical mathematical-entropy pair
\cite{FriedrichsLax1971,GiovangigliMatuszewski2013}. We first state the
general affine-invariance property and then specialize it to the wall
temperature. The resulting representative is not a new thermodynamic
entropy. Its inherited structure provides the basis for the adapted
entropy balance, exact wall cancellation, conditional entropy-based
estimate, and ballistic-energy interpretation.

\subsection{Affine modification by conserved variables}
\label{subsec:affine-entropy}

\begin{lemma}[Affine invariance of the entropy structure]
\label{lem:affine-entropy}
Let $(S,\FI{i})$ be the canonical entropy pair introduced in
Subsection~\ref{subsec:canonical-entropy}, and let
$\bm a\in\R^{d+2}$ be constant. Define
\begin{equation}
  S^{\bm a}(\bQ)
  =
  S(\bQ)+\bm a^{\transpose}\bQ,
  \qquad
  F_i^{\bm a}(\bQ)
  =
  \FI{i}(\bQ)+\bm a^{\transpose}\bFI{i}(\bQ),
  \qquad
  i=1,\ldots,d.
  \label{eq:affine-entropy-pair}
\end{equation}
Then $(S^{\bm a},F_i^{\bm a})$ is an entropy pair for the inviscid part of
\eqref{eq:cns}. Its entropy variables and Hessian satisfy
\begin{equation}
  \bm W^{\bm a}
  =
  \nabla_{\bQ}S^{\bm a}
  =
  \bW+\bm a,
  \qquad
  \nabla_{\bQ}^{2}S^{\bm a}
  =
  \nabla_{\bQ}^{2}S.
  \label{eq:affine-entropy-variables-hessian}
\end{equation}
Moreover, the entropy potential and entropy-flux potentials are unchanged
when evaluated at the same conservative state:
\begin{equation*}
  \Phi^{\bm a}
  =
  \Phi,
  \qquad
  \Psi_i^{\bm a}
  =
  \Psi_i.
\end{equation*}
Finally,
\begin{equation*}
  \p_i\bm W^{\bm a}
  =
  \p_i\bW.
\end{equation*}
Consequently, the viscous symmetrization and the volume
entropy-production density are identical to those associated with the
canonical mathematical entropy $S$.
\end{lemma}

The proof is given in~\ref{app:affine-invariance-proof}. The
unchanged Hessian preserves strict convexity on $\mathcal A$, so the
affine modification changes the entropy representative without changing
its underlying convex entropy structure. We now use this freedom to
select the representative associated with the wall-balance combination
in Eq.~\eqref{eq:canonical-energy-wall-cancellation}.

\subsection{Entropy pair associated with the wall temperature}
\label{subsec:temperature-adapted-pair}

Let
\begin{equation*}
  \be_E
  =
  (0,\ldots,0,1)^{\transpose}
  \in\R^{d+2},
  \qquad
  \be_E^{\transpose}\bQ
  =
  \rho E,
\end{equation*}
and assume throughout this section that $\Tw>0$ is spatially and
temporally constant. To realize the prescribed combination of entropy
and total-energy balances, choose
\begin{equation*}
  \bm a
  =
  \frac{\be_E}{\Tw}.
\end{equation*}
Lemma~\ref{lem:affine-entropy} then gives the wall-temperature-adapted
entropy representative
\begin{equation}
  \ST(\bQ)
  =
  S(\bQ)+\frac{\rho E}{\Tw}
  \label{eq:adapted-entropy}
\end{equation}
and its associated inviscid entropy flux
\begin{equation}
  \FIT{i}(\bQ)
  =
  \FI{i}(\bQ)
  +
  \frac{(\rho E+p)u_i}{\Tw},
  \qquad
  i=1,\ldots,d.
  \label{eq:adapted-entropy-flux}
\end{equation}
The entropy-variable shift in
Eq.~\eqref{eq:affine-entropy-variables-hessian} becomes
\begin{equation}
  \bWT
  =
  \nabla_{\bQ}\ST
  =
  \bW+\frac{\be_E}{\Tw}
  =
  \begin{pmatrix}
    \displaystyle
    \frac{h}{T}
    -
    s
    -
    \frac{\lvert\bu\rvert^2}{2T}
    \\[1.5ex]
    \displaystyle
    \frac{\bu}{T}
    \\[1.5ex]
    \displaystyle
    \frac{1}{\Tw}-\frac{1}{T}
  \end{pmatrix}.
  \label{eq:adapted-entropy-variables}
\end{equation}
In particular, the energy component of $\bWT$ vanishes whenever
$T=\Tw$.

The remaining structural conclusions of
Lemma~\ref{lem:affine-entropy} specialize to
\begin{equation}
  \nabla_{\bQ}^{2}\ST
  =
  \bm H_S,
  \qquad
  \p_i\bWT
  =
  \p_i\bW,
  \qquad
  \bFV{i}
  =
  \bm C_{ij}\p_j\bWT.
  \label{eq:adapted-entropy-structure}
\end{equation}
Thus $\ST$ is strictly convex on $\mathcal A$ and has the same entropy
Hessian and viscous coefficient matrices as $S$. The lemma also gives
the same potential functions $\Phi$ and $\Psi_i$ and the same volume
entropy-production density $\Dcal$.

These inherited properties determine the adapted local entropy balance.
For a sufficiently smooth solution, contracting Eq.~\eqref{eq:cns}
with $\bWT^{\transpose}$, using entropy-pair compatibility, and applying
the gradient identity in Eq.~\eqref{eq:adapted-entropy-structure} gives
\begin{equation}
  \p_t\ST
  +
  \p_i\FIT{i}
  =
  \p_i
  \left(
    \bWT^{\transpose}\bFV{i}
  \right)
  -
  \Dcal.
  \label{eq:local-adapted-entropy}
\end{equation}
Integrating over $\Omega$ gives
\begin{equation}
  \frac{\dd}{\dd t}
  \int_{\Omega}
  \ST\,\dd\Omega
  =
  \int_{\partial\Omega}
  \Bcal_{\ST}\,\dd\Gamma
  -
  \int_{\Omega}
  \Dcal\,\dd\Omega,
  \label{eq:global-adapted-entropy}
\end{equation}
where
\begin{equation*}
  F_{T_w,n}
  =
  n_i\FIT{i}
\end{equation*}
is the outward normal adapted inviscid entropy flux. Substituting the
affine entropy-variable shift and flux modification into the pointwise
boundary contribution yields
\begin{align}
  \Bcal_{\ST}
  &=
  \bWT^{\transpose}\bm F_n^{(V)}
  -
  F_{T_w,n}
  \notag\\
  &=
  \Bcal_S
  -
  \frac{F_{E,n}}{\Tw}
  \notag\\
  &=
  \frac{\qnormal}{T}
  +
  \frac{\bu\cdot\btau\bn-\qnormal}{\Tw}
  -
  \left(
    S+\frac{\rho E+p}{\Tw}
  \right)u_n.
  \label{eq:adapted-boundary-contribution}
\end{align}
The second equality identifies the boundary term of the adapted entropy
pair with the balance combination in
Eq.~\eqref{eq:canonical-energy-wall-cancellation}. Thus
Lemma~\ref{lem:affine-entropy} supplies the unchanged convexity and volume
entropy production, while the stationary isothermal-wall identities
established in Section~\ref{sec:canonical-wall} give the cancellation
of $\Bcal_{\ST}$ on $\wall$.

\subsection{Boundary entropy conservation at an isothermal wall}
\label{subsec:isothermal-wall-adapted}

For a specified entropy pair, a wall is called boundary-entropy-conservative
when its pointwise contribution to the corresponding continuous entropy
balance vanishes for the exact wall data. It is called boundary-entropy-stable
when that contribution is nonpositive. These terms refer only to the boundary
contribution. The viscous and heat-conduction mechanisms continue to produce
the nonnegative volume entropy dissipation $\Dcal$.

\begin{theorem}[Boundary entropy conservation at an isothermal wall]
\label{thm:isothermal-wall-adapted}
Let $\Tw>0$ be spatially and temporally constant. For a sufficiently smooth
solution of \eqref{eq:cns}, a stationary no-slip wall satisfying
\eqref{eq:isothermal-wall-data} is boundary-entropy-conservative with respect
to the adapted entropy pair $(\ST,\FIT{i})$. More precisely,
\begin{equation}
  \left.
  \Bcal_{\ST}
  \right|_{\wall}
  =
  0.
  \label{eq:zero-adapted-wall-contribution}
\end{equation}
\end{theorem}

\begin{proof}
At a stationary no-slip wall, $\bu=\bm{0}$ and $u_n=0$. Therefore,
\eqref{eq:adapted-boundary-contribution} reduces to
\begin{equation*}
  \left.
  \Bcal_{\ST}
  \right|_{\wall}
  =
  \qnormal
  \left(
    \frac{1}{T}
    -
    \frac{1}{\Tw}
  \right).
\end{equation*}
The isothermal condition $T=\Tw$ gives
\eqref{eq:zero-adapted-wall-contribution}. The same result follows
from \eqref{eq:canonical-energy-wall-cancellation} and the second equality
in \eqref{eq:adapted-boundary-contribution}.
\end{proof}

\begin{corollary}[Global adapted-entropy balance and estimate]
\label{cor:global-adapted-entropy}
Let the solution be sufficiently smooth, and let $\wall$ satisfy the
assumptions of Theorem~\ref{thm:isothermal-wall-adapted}. Then
\begin{equation}
  \frac{\dd}{\dd t}
  \int_{\Omega}
  \ST\,\dd\Omega
  +
  \int_{\Omega}
  \Dcal\,\dd\Omega
  =
  \int_{\partial\Omega\setminus\wall}
  \Bcal_{\ST}\,\dd\Gamma.
  \label{eq:global-adapted-wall-identity}
\end{equation}
If the remaining boundary portions satisfy
$\Bcal_{\ST}\leq0$, then
\begin{equation}
  \int_{\Omega}
  \ST\bigl(\bQ(\bm x,t)\bigr)\,\dd\Omega
  +
  \int_{0}^{t}
  \int_{\Omega}
  \Dcal\,\dd\Omega\,\dd\tau
  \leq
  \int_{\Omega}
  \ST\bigl(\bQ_0(\bm x)\bigr)\,\dd\Omega.
  \label{eq:global-adapted-dissipation-identity}
\end{equation}
If $\Bcal_{\ST}=0$ on $\partial\Omega\setminus\wall$, then the
inequality in
\eqref{eq:global-adapted-dissipation-identity} is an equality.
\end{corollary}

\begin{remark}[Adiabaticity is not required]
\label{rem:adapted-not-adiabatic}
Theorem~\ref{thm:isothermal-wall-adapted} does not require
$\qnormal=0$. The fluid may exchange canonical mathematical entropy with the
wall through the contribution $\qnormal/\Tw$. The adapted boundary
contribution vanishes because this exchange is cancelled exactly by the net
outward total-energy flux divided by $\Tw$. Hence, the wall is
boundary-entropy-conservative with respect to $\ST$, even when conductive heat
crosses
it.
\end{remark}

\subsection{Entropy stability with respect to
  \texorpdfstring{$S_{T_w}$}{S(Tw)} and a conditional
  \texorpdfstring{$L_2$}{L2}-type a priori bound}
\label{subsec:entropy-l2-bound}

Under the remaining-boundary condition of
Corollary~\ref{cor:global-adapted-entropy}, the global estimate
\eqref{eq:global-adapted-dissipation-identity} yields a conditional
bound about the prescribed initial field $\bQ_0$. For an admissible
solution, convexity of $\mathcal A$ ensures that the segment joining
$\bQ_0(\bm x)$ and $\bQ(\bm x,t)$ remains admissible. For each fixed
$t\in[0,t_f]$, define
\begin{equation}
  \underline{\lambda}_{T_w}(t)
  =
  \operatorname*{ess\,inf}_{\bm x\in\Omega}
  \inf_{\theta\in[0,1]}
  \lambda_{\min}
  \left[
    \nabla_{\bQ}^{2}\ST
    \left(
      (1-\theta)\bQ_0(\bm x)
      +
      \theta\bQ(\bm x,t)
    \right)
  \right].
  \label{eq:adapted-entropy-minimum-hessian-eigenvalue}
\end{equation}
Here, $\lambda_{\min}$ denotes the smallest eigenvalue of its symmetric
matrix argument.
Because
\begin{equation*}
  \nabla_{\bQ}^{2}\ST
  =
  \nabla_{\bQ}^{2}S,
\end{equation*}
the canonical and adapted entropies have the same segmentwise convexity
constant. Assume that $\underline{\lambda}_{T_w}(t)>0$ and that
$\bQ_0,\bWT(\bQ_0)\in L_2(\Omega;\R^{d+2})$. Taylor's theorem,
Cauchy--Schwarz, and the triangle inequality then give
\begin{equation}
  \left\|
    \bQ(\cdot,t)
  \right\|_{L_2(\Omega)}
  \leq
  \left\|
    \bQ_0
  \right\|_{L_2(\Omega)}
  +
  \frac{2}{\underline{\lambda}_{T_w}(t)}
  \left\|
    \bWT(\bQ_0)
  \right\|_{L_2(\Omega)}.
  \label{eq:continuous-entropy-l2-bound}
\end{equation}
The derivation is given in \ref{app:entropy-l2-derivation}. A time-uniform
bound follows if
$\inf_{0\leq t\leq t_f}\underline{\lambda}_{T_w}(t)>0$. 
The passage from entropy control to an $L_2$ estimate relies on uniform
convexity. See \cite[Chapter~V]{Dafermos2016} for the general framework
and \cite{SvardCarpenterParsani2018} for the compressible
Navier--Stokes wall-boundary setting. The estimate does not establish
positivity preservation, global well-posedness, or regularity.

Note that the relative entropies generated by $\ST$ and $S$ are identical, as
shown in \ref{app:relative-entropy-invariance}. However, this separate identity
is not needed for Eq.~\eqref{eq:continuous-entropy-l2-bound}.

\subsection{Ballistic free-energy interpretation}
\label{subsec:ballistic-free-energy}

For the fixed reservoir temperature $\Tw>0$, define the thermodynamic
ballistic free-energy density
\begin{equation}
  \mathcal F_{\Tw}^{\mathrm{bal}}(\rho,T)
   =
  \rho e(\rho,T)
  -
  \Tw\rho s(\rho,T).
  \label{eq:ballistic-free-energy-density}
\end{equation}
This potential is used in the mathematical analysis of compressible,
viscous, and heat-conducting fluids
\cite{Feireisl2012,ChaudhuriFeireisl2022}. The associated total
ballistic-energy density also includes kinetic energy:
\begin{align}
  \mathcal E_{\Tw}^{\mathrm{bal}}(\bQ)
  & =
  \frac12\rho\lvert\bu\rvert^2
  +
  \mathcal F_{\Tw}^{\mathrm{bal}}(\rho,T)
  \notag\\
  &=
  \rho E-\Tw\rho s
  \notag\\
  &=
  \Tw\ST(\bQ).
  \label{eq:total-ballistic-energy-density}
\end{align}
At $T=\Tw$, the thermodynamic potential reduces to
\begin{equation}
  \left.
  \mathcal F_{\Tw}^{\mathrm{bal}}(\rho,T)
  \right|_{T=\Tw}
  =
  \rho\bigl(e-Ts\bigr),
  \label{eq:ballistic-helmholtz-relation}
\end{equation}
which is the usual Helmholtz free-energy density \cite{Callen1985}. Note that away from
$T=\Tw$, the parameter $\Tw$ remains the prescribed reservoir
temperature, not the local fluid temperature.

Since $\mathcal E_{\Tw}^{\mathrm{bal}}=\Tw\ST$ with $\Tw>0$, the total
ballistic energy inherits strict convexity in $\bQ$ on $\mathcal A$.
Its balance is obtained by multiplying
Eq.~\eqref{eq:global-adapted-wall-identity} by $\Tw$. The Hessian and
balance identities are given in
\ref{app:ballistic-energy-identities}. At a stationary no-slip wall,
the energy and scaled canonical-entropy balances contribute
$-\qnormal$ and $\Tw\qnormal/T$, respectively, to the integral rates.
They cancel at $T=\Tw$, eliminating the solution-dependent conductive
heat flux from the wall contribution. This is the constant-reservoir
specialization of the ballistic-energy mechanism for Dirichlet
temperature data \cite{ChaudhuriFeireisl2022}.

Changing the entropy zero by $s\mapsto s+C_s$, with constant $C_s$,
adds the affine mass-density term $-\Tw C_s\rho$ to
$\mathcal E_{\Tw}^{\mathrm{bal}}$, and it leaves the Hessian, volume
entropy-production mechanism, and isothermal-wall cancellation unchanged.
The corresponding relative-entropy invariance is given in
\ref{app:relative-entropy-invariance}.

These continuous identities provide the entropy and total-energy
conditions that the two semi-discrete wall formulations in the next
section must reproduce (i.e., mimic) or bound.

\section{Semi-discrete DG--SBP formulation and two isothermal-wall SAT
formulations}
\label{sec:semidiscrete-wall}

In this section, we develop the CR--SAT and FEE--SAT isothermal-wall
formulations within a common semi-discrete nodal DG--SBP discretization
of the compressible Navier--Stokes equations. Entropy stability analysis
establishes that both formulations reproduce the adapted wall balance
derived in Section~\ref{sec:adapted-entropy}, with zero wall contribution
without wall dissipation and a nonpositive contribution when it is
enabled. These identities hold algebraically at the semi-discrete level and are
verified numerically up to floating-point roundoff in the numerical
results of Section~\ref{sec:results}.

Both wall formulations share the discretization developed in
Subsections~\ref{subsec:tensor-product-sbp}--%
\ref{subsec:interface-2015-coupling}. It combines tensor-product SBP
operators, the entropy-conservative inviscid flux differencing of
Fisher and Carpenter~\cite{FisherCarpenter2013,CarpenterEtAl2014},
and local discontinuous Galerkin (LDG) auxiliary-gradient equations
\cite{CockburnShu1998LDG,ParsaniInterface2015}. The interior-interface
coupling is entropy-conservative when the optional inviscid
dissipation and interior penalty (IP) are omitted.

Subsection~\ref{subsec:isothermal-wall-common-data} defines the shared
wall data and residual within the coupled conservative and
auxiliary-gradient SAT architecture of
Refs.~\cite{ParsaniCarpenterNielsen2015,DalcinEtAl2019}.
Subsection~\ref{subsec:cr-sat} constructs the CR--SAT correction acting
on the complete wall-data residual.
Subsection~\ref{subsec:fee-sat} constructs the FEE--SAT correction,
which additionally equates the complete numerical outward total-energy
flux to the one-sided numerical Fourier heat flux at finite spatial
resolution.
Finally, Subsection~\ref{subsec:two-sat-global-balances} assembles the
source-free global entropy and total-energy balances and compares the
two formulations.

\subsection{Semi-discrete DG--SBP setting}
\label{subsec:tensor-product-sbp}

This subsection introduces the tensor-product SBP operators, their
vector extensions, and the quadrature inner products used in the
semi-discrete DG--SBP formulation on conforming Cartesian affine
elements.

Let $\mathcal T_h$ be a conforming partition of $\Omega$ into Cartesian
affine tensor-product elements $K$, mapped coordinate-wise from
$\widehat K=[-1,1]^d$. In each reference direction, a polynomial of degree
$\psol\geq1$ is collocated at $\Nsol=\psol+1$
Legendre--Gauss--Lobatto (LGL) nodes. Let $\mathsf P_{\Nsol}$ be the
positive diagonal weight matrix, $\mathsf D_{\Nsol}$ the one-dimensional
differentiation matrix, and
$\mathsf Q_{\Nsol}=\mathsf P_{\Nsol}\mathsf D_{\Nsol}$. Then
\begin{equation}
  \mathsf Q_{\Nsol}+\mathsf Q_{\Nsol}^{\transpose}
  =\mathsf B_{\Nsol},
  \qquad
  \mathsf B_{\Nsol}=\operatorname{diag}(-1,0,\ldots,0,1),
  \qquad
  \mathsf D_{\Nsol}\bm1_{\Nsol}=\bm0,
  \label{eq:one-dimensional-sbp-property}
\end{equation}
where $\bm1_{\Nsol}\in\R^{\Nsol}$ is the vector of ones. Background on
SBP operators and their nodal DG correspondence is given in
Refs.~\cite{SvardNordstrom2014,DelReyFernandezHickenZingg2014,Gassner2013,CarpenterEtAl2014,CarpenterParsaniFisherNielsen2015}.

In a tensor-product element, there are $N_K=\Nsol^d$ volume nodes and $N_f=\Nsol^{d-1}$ nodes on each
face. The scalar physical norm and Cartesian derivative matrices are
$\mathsf P_K,\mathsf D_{i,K}\in\R^{N_K\times N_K}$, $i=1,\ldots,d$.
For $f\subset\partial K$, $\mathsf R_f\in\R^{N_f\times N_K}$ restricts
volume values to face nodes, and $\mathsf P_f\in\R^{N_f\times N_f}$ is
the physical face-quadrature matrix. Let $\bm n_{i,f}\in\R^{N_f}$ collect
the $i$th Cartesian components of the unit normal outward from $K$ and
set $\mathsf N_{i,f}=\operatorname{diag}(\bm n_{i,f})$.
The element index is suppressed on these face operators. They satisfy
\begin{equation}
  \mathsf P_K\mathsf D_{i,K}
  +\mathsf D_{i,K}^{\transpose}\mathsf P_K
  =
  \sum_{f\subset\partial K}
  \mathsf R_f^{\transpose}\mathsf P_f
  \mathsf N_{i,f}\mathsf R_f,
  \qquad
  \mathsf D_{i,K}\bm1_K=\bm0,
  \label{eq:multidimensional-sbp-property}
\end{equation}
where $\bm1_K\in\R^{N_K}$ is the element vector of ones.
\ref{app:sbp-operator-construction} gives the affine map, metric and
Jacobian factors, explicit tensor-product and face operators, and endpoint
norm scaling.

For $m=d+2$, let $\bm I_m$ be the $m\times m$ identity and $\otimes$ the
Kronecker product. The vector extensions are
\begin{equation*}
  \mathbb P_K=\mathsf P_K\otimes\bm I_m,
  \qquad
  \mathbb D_{i,K}=\mathsf D_{i,K}\otimes\bm I_m,
  \qquad
  \mathbb P_K,\mathbb D_{i,K}
  \in\R^{mN_K\times mN_K}.
\end{equation*}
For $a,b\in\R^{N_K}$, $a_f,b_f\in\R^{N_f}$,
$\bm a,\bm b\in\R^{mN_K}$, and
$\bm a_f,\bm b_f\in\R^{mN_f}$, define
\begin{align*}
  \langle a,b\rangle_{K,h}
  &=a^{\transpose}\mathsf P_Kb,
  &
  \langle a_f,b_f\rangle_{f,h}
  &=a_f^{\transpose}\mathsf P_fb_f,
  \\
  \langle\bm a,\bm b\rangle_{K,h}^{(m)}
  &=\bm a^{\transpose}\mathbb P_K\bm b,
  &
  \langle\bm a_f,\bm b_f\rangle_{f,h}^{(m)}
  &=\bm a_f^{\transpose}
    (\mathsf P_f\otimes\bm I_m)\bm b_f.
\end{align*}
In quadrature inner products of scalar fields, $1$ denotes the vector
of ones on the corresponding element or face.

Note that the analysis is restricted to conforming Cartesian affine elements.
Curvilinear extensions require entropy-compatible metric approximations
satisfying the discrete metric identities and geometric conservation law
\cite{CreanEtAl2018,
DelReyFernandezEtAl2020CurvilinearTM}. Entropy-stable $p$- and
$h/p$-nonconforming extensions require SBP-compatible, conservative,
entropy-stable interface operators
\cite{FriedrichEtAl2018,DelReyFernandezEtAl2020PNonconformingNS,
DelReyFernandezEtAl2020HP}. Under these conditions, the local CR--SAT and
FEE--SAT wall identities retain the same algebraic form presented herein.

\subsection{Inviscid entropy-conservative flux differencing}
\label{subsec:inviscid-flux-differencing}

This subsection presents the entropy-conservative SBP flux-differencing
construction of the inviscid volume operator and explains how the
two-point flux conditions and the SBP property reduce its
quadrature-weighted entropy contraction to element-face contributions.

The element vector $\bm q_K$ collects admissible conservative states
at the volume nodes:
\begin{equation*}
  \bm q_K
  =
  \left(
    \bQ_1^{\transpose},\ldots,\bQ_{N_K}^{\transpose}
  \right)^{\transpose}
  \in\R^{mN_K},
  \qquad
  \bQ_a
  =
  \left(
    \rho_a,(\rho u_1)_a,\ldots,(\rho u_d)_a,(\rho E)_a
  \right)^{\transpose}
  \in\mathcal A\subset\R^m,
\end{equation*}
where $a=1,\ldots,N_K$ and $\mathcal A$ is defined in
Section~\ref{sec:equations}.

For $i=1,\ldots,d$, let
$\bm F_i^{\mathrm{ec}}:\mathcal A\times\mathcal A\to\R^m$
be a two-point inviscid flux, where $\mathrm{ec}$ denotes entropy-conservative.
For admissible states $\bQ_L,\bQ_R\in\mathcal A$, the
labels $L,R$ identify flux arguments rather than physical interface
traces. We require symmetry and consistency:
\begin{equation*}
  \bm F_i^{\mathrm{ec}}(\bQ_L,\bQ_R)
  =
  \bm F_i^{\mathrm{ec}}(\bQ_R,\bQ_L),
  \qquad
  \bm F_i^{\mathrm{ec}}(\bQ,\bQ)
  =
  \bFI{i}(\bQ).
\end{equation*}
With $\bW_L=\bW(\bQ_L)$ and $\bW_R=\bW(\bQ_R)$, the flux must also
satisfy Tadmor's entropy-conservation identity
\cite{Tadmor1987,Tadmor2003}:
\begin{equation}
  (\bW_R-\bW_L)^{\transpose}
  \bm F_i^{\mathrm{ec}}(\bQ_L,\bQ_R)
  =
  \Psi_i(\bQ_R)-\Psi_i(\bQ_L),
  \label{eq:tadmor-identity-section5}
\end{equation}
where $\Psi_i$ is the entropy-flux potential defined in
Eq.~\eqref{eq:entropy-potentials}.

We apply the entropy-conservative SBP flux-differencing
construction of Fisher and Carpenter~\cite{FisherCarpenter2013} in each
Cartesian direction, using the tensor-product differentiation matrices
of Subsection~\ref{subsec:tensor-product-sbp}, as in the
spectral-collocation formulation of Ref.~\cite{CarpenterEtAl2014}.
The resulting directional and total inviscid volume operators are
\begin{equation}
  \left[
    \mathcal V_{i,K,h}^{(I),\mathrm{ec}}(\bm q_K)
  \right]_a
  =
  2\sum_{b=1}^{N_K}
  (\mathsf D_{i,K})_{ab}
  \bm F_i^{\mathrm{ec}}(\bQ_a,\bQ_b),
  \qquad
  \mathcal V_{K,h}^{(I)}(\bm q_K)
  =
  \sum_{i=1}^{d}
  \mathcal V_{i,K,h}^{(I),\mathrm{ec}}(\bm q_K).
  \label{eq:entropy-conservative-volume-flux-differencing}
\end{equation}
The $a$th block of the directional operator has the component ordering
of $\bQ_a$ and approximates $\partial_i\bFI{i}$ without summation over
$i$. The total operator belongs to $\R^{mN_K}$ and approximates the
inviscid flux divergence~\cite{DelReyFernandezEtAl2020HP}.
The indices $a,b$ enumerate volume nodes,
$(I)$ identifies the inviscid contribution, and $K,h$ identify the
element and spatial discretization. The tensor-product entries and
coordinate-line evaluation are given in~\ref{app:sbp-operator-construction}.

Contract each nodal block of
Eq.~\eqref{eq:entropy-conservative-volume-flux-differencing} with the
adapted entropy variables at that node and sum with the volume-quadrature
weights. Symmetry pairs nodal terms, and consistency recovers the physical
flux at equal states. By Lemma~\ref{lem:affine-entropy},
Eq.~\eqref{eq:tadmor-identity-section5} also holds for the adapted variables
and converts the paired contractions into entropy-potential differences.
The SBP identity then leaves only adapted inviscid entropy-flux terms on
element faces, with no inviscid volume entropy production
\cite{FisherCarpenter2013,CarpenterEtAl2014}.
This is the discrete telescoping property. Any two-point flux satisfying
these conditions may be used.

Equation~\eqref{eq:entropy-conservative-volume-flux-differencing} defines
only the inviscid volume contribution. The conservative and auxiliary-gradient
equations follow in Subsection~\ref{subsec:first-order-viscous-sbp}, and
Subsection~\ref{subsec:interface-2015-coupling} specifies the interior-face
SAT data.

\subsection{Viscous auxiliary-gradient discretization and semi-discrete
element equations}
\label{subsec:first-order-viscous-sbp}

This subsection introduces the LDG approximation of the entropy-variable
gradients, defines the nodal viscous fluxes, and formulates the coupled
semi-discrete conservative and auxiliary-gradient element equations
with face SAT terms.

The LDG auxiliary equations approximate the Cartesian derivatives
$\p_i\bW$, $i=1,\ldots,d$, following Ref.~\cite{CockburnShu1998LDG}.
Their DG--SBP--SAT formulation follows
Refs.~\cite{CarpenterEtAl2014,ParsaniInterface2015}. The canonical and
wall-temperature-adapted nodal entropy-variable vectors are
\begin{equation*}
  \bm w_K
  =
  \left(
    \bW(\bQ_1)^{\transpose},\ldots,
    \bW(\bQ_{N_K})^{\transpose}
  \right)^{\transpose}
  \in\R^{mN_K}
\end{equation*}
and
\begin{equation*}
  \bm w_{T_w,K}
  =\bm w_K+\bm1_K\otimes\frac{\be_E}{T_w}
  \in\R^{mN_K},
\end{equation*}
where $\be_E=(0,\ldots,0,1)^{\transpose}\in\R^m$ selects the
conservative total-energy component $\rho E$. The $a$th adapted block is
$\bW(\bQ_a)+\be_E/T_w$. Since $T_w$ is constant and
$\mathsf D_{i,K}\bm1_K=\bm0$,
\begin{equation*}
  \mathbb D_{i,K}\bm w_{T_w,K}
  =\mathbb D_{i,K}\bm w_K.
\end{equation*}
Thus, $\bm\theta_{i,K}\in\R^{mN_K}$ approximates the nodal values of
both $\p_i\bW$ and $\p_i\bWT$.

At each volume node $a$, we evaluate the single-state viscous coefficient
matrices $\bm C_{ij}(\bQ_a)\in\R^{m\times m}$ defined in
Section~\ref{sec:equations}. Their block-diagonal assembly and the
corresponding nodal viscous fluxes are given by
\begin{align}
  [\bm C_{ij}]_K
  &=
  \operatorname{diag}
  \left(
    \bm C_{ij}(\bQ_1),\ldots,\bm C_{ij}(\bQ_{N_K})
  \right)
  \in\R^{mN_K\times mN_K},
  \label{eq:block-diagonal-viscous-coefficients}
  \\
  \bm f_{i,K}^{(V)}
  &=[\bm C_{ij}]_K\bm\theta_{j,K}
  \in\R^{mN_K}.
  \label{eq:discrete-viscous-flux}
\end{align}
The dimension-independent form of $\bm C_{ij}$ is given in
\ref{app:viscous-coefficient-matrices}.

For each face $f\subset\partial K$, define the quadrature-weighted
injection
\begin{equation}
  \mathbb E_{f,K}
  =
  \left(
    \mathsf R_f^{\transpose}\mathsf P_f
  \right)
  \otimes\bm I_m
  \in\R^{mN_K\times mN_f}.
  \label{eq:face-lifting-matrix}
\end{equation}
For $\bm z_f\in\R^{mN_f}$, $\mathbb E_{f,K}\bm z_f$ injects the
physical face-quadrature-weighted data into the element vector. Its
mass-inverted counterpart $\mathbb P_K^{-1}\mathbb E_{f,K}$ is the
lifting into the volume contribution. Because $\mathsf R_f^{\transpose}$
selects face LGL nodes and $\mathbb P_K^{-1}$ is diagonal, the lifting
vanishes at all other volume nodes.

Let $\bm g_{f,K}^{q},\bm g_{i,f,K}^{\theta}\in\R^{mN_f}$ denote the
conservative and auxiliary-gradient face SAT data, respectively. With the
inviscid volume contribution in
Eq.~\eqref{eq:entropy-conservative-volume-flux-differencing}, the element
equations are
\begin{align}
  \frac{\dd\bm q_K}{\dd t}
  +\mathcal V_{K,h}^{(I)}(\bm q_K)
  -\mathbb D_{i,K}\bm f_{i,K}^{(V)}
  &=
  \mathbb P_K^{-1}
  \sum_{f\subset\partial K}
  \mathbb E_{f,K}\bm g_{f,K}^{q},
  \label{eq:semidiscrete-conservative-equation}
  \\
  \bm\theta_{i,K}-\mathbb D_{i,K}\bm w_K
  &=
  \mathbb P_K^{-1}
  \sum_{f\subset\partial K}
  \mathbb E_{f,K}\bm g_{i,f,K}^{\theta},
  \qquad i=1,\ldots,d.
  \label{eq:discrete-entropy-gradient}
\end{align}
To complete Eqs.~\eqref{eq:semidiscrete-conservative-equation}
and~\eqref{eq:discrete-entropy-gradient}, we must specify the face SAT
vectors $\bm g_{f,K}^{q}$ and $\bm g_{i,f,K}^{\theta}$.
Their definitions determine the coupling between neighboring elements
and the imposition of boundary conditions.
Subsection~\ref{subsec:interface-2015-coupling} specifies these vectors
on interior faces.
Subsections~\ref{subsec:isothermal-wall-common-data}--\ref{subsec:fee-sat}
introduce the common isothermal-wall data and define the corresponding
wall SAT vectors for the CR--SAT and FEE--SAT formulations.

\subsection{Conforming interior-interface SAT coupling}
\label{subsec:interface-2015-coupling}

To complete the element equations in
Subsection~\ref{subsec:first-order-viscous-sbp}, we now specify the SAT
vectors $\bm g_{f,K}^{q}$ and $\bm g_{i,f,K}^{\theta}$ on conforming
interior faces and derive their combined entropy contribution.
Following Refs.~\cite{CarpenterEtAl2014,ParsaniInterface2015}, we use a
common entropy-conservative inviscid flux with optional entropy-stable dissipation,
paired LDG traces, and an optional IP penalty. The paired LDG coupling
is entropy-conservative at the interface, while the IP penalty supplies
additional entropy dissipation.

For a face $f$ shared by elements $K^-$ and $K^+$, fix a unit normal
$\bn$ from $K^-$ to $K^+$. Superscripts $-$ and $+$ denote traces from
these elements at matching physical face nodes in a common ordering.
The face label $f$ is suppressed on one-sided traces. A subscript $n$
denotes contraction with the common normal $\bn$, and a hat denotes a
numerical trace or flux. The formulas are nodewise, with vectors in
$\R^m$ and matrices in $\R^{m\times m}$. In face-quadrature products,
these symbols denote stacked vectors in $\R^{mN_f}$ or block-diagonal
matrices in $\R^{mN_f\times mN_f}$. Define
\begin{equation*}
  \bm F_n^{(I),\pm}
  =n_i\bFI{i}(\bQ^{\pm}),
  \qquad
  \bm F_n^{(V),\pm}
  =n_i\bm C_{ij}(\bQ^{\pm})\bm\theta_j^{\pm},
  \qquad
  \bm d_f=\bW^- -\bW^+.
\end{equation*}
Here, $\bm\theta_j^{\pm}$ are traces of the auxiliary gradients.
Both physical flux traces use the same normal $\bn$.

\paragraph{Inviscid numerical flux}
Define the normal entropy-conservative flux and entropy-flux potential by~\cite{CarpenterEtAl2014}
\begin{equation*}
  \bm F_n^{\mathrm{ec}}(\bQ^-,\bQ^+)
  =n_i\bm F_i^{\mathrm{ec}}(\bQ^-,\bQ^+),
  \qquad
  \Psi_n(\bQ)=n_i\Psi_i(\bQ).
\end{equation*}
For $\bm D_f\in\R^{m\times m}$ satisfying $\bm D_f\succeq\bm0$, set
\begin{equation}
  \widehat{\bm F}_{n,f}^{(I)}
  =\bm F_n^{\mathrm{ec}}(\bQ^-,\bQ^+)
   -\frac{1}{2}\bm D_f(\bW^+-\bW^-).
  \label{eq:merriam-interface-flux}
\end{equation}
The choice $\bm D_f=\bm0$ recovers the entropy-conservative flux. A standard
entropy-scaled choice is
\begin{equation*}
  \bm D_f=\bm Y_f|\bm\Lambda_f|\bm Y_f^{\transpose}.
\end{equation*}
Here, $\bm\Lambda_f\in\R^{m\times m}$ is the diagonal eigenvalue
matrix of the normal inviscid flux Jacobian, and the columns of
$\bm Y_f\in\R^{m\times m}$ are its entropy-scaled right eigenvectors.
$|\bm\Lambda_f|$ contains the absolute characteristic speeds.
The eigensystem is evaluated at a consistent interface state, reducing to
$\bQ$ when $\bQ^-=\bQ^+=\bQ$. The normalization is given in
\ref{app:semidiscrete-entropy-identity}. This scaling follows
Merriam~\cite{Merriam1989} and its use in
Refs.~\cite{CarpenterEtAl2014,ParsaniInterface2015}. When Roe averaging is
used, the eigensystem is evaluated at the Roe-averaged state
\cite{Roe1981}.

\paragraph{Viscous numerical trace and flux}
Following Ref.~\cite{ParsaniInterface2015}, define the LDG trace with
bias parameter $\alpha\in[-1,1]$ by
\begin{equation}
  \widehat{\bW}_f
  =\frac{1}{2}
   \left[
     (1+\alpha)\bW^-+(1-\alpha)\bW^+
   \right].
  \label{eq:ldg-primal-trace}
\end{equation}
The nonnegative weights sum to one. The choices $\alpha=0$, $1$, and
$-1$ give the symmetric trace, $\bW^-$, and $\bW^+$, respectively.
Let $p_{n,f}^{\pm}>0$ be the physical normal-endpoint norm entries of
$K^\pm$ at $f$, as constructed in
Eq.~\eqref{eq:normal-endpoint-mass-entry}. The matching-normal-scaling
specialization assumes
\begin{equation*}
  p_{n,f}^-=p_{n,f}^+=p_{n,f},
\end{equation*}
and uses the normal viscous coefficient matrices and IP matrix
\begin{equation}
  \bm C_{nn}^{\pm}
  =n_i\bm C_{ij}(\bQ^{\pm})n_j,
  \qquad
  \bm L_f
  =-\frac{\beta_{\mathrm{int}}}{p_{n,f}}
   \frac{\bm C_{nn}^-+\bm C_{nn}^+}{2},
  \qquad
  \beta_{\mathrm{int}}\geq0.
  \label{eq:interface-ip-matrix}
\end{equation}
Since $\bm C_{nn}^{\pm}\succeq\bm0$, $\beta_{\mathrm{int}}\geq0$,
and $p_{n,f}>0$, the IP matrix satisfies $\bm L_f\preceq\bm0$.
The parameter $\beta_{\mathrm{int}}$ controls the penalty strength,
while $p_{n,f}^{-1}$ provides the inverse normal-mesh scaling
\cite{ParsaniInterface2015,DalcinEtAl2019}.
Setting $\beta_{\mathrm{int}}=0$ removes the IP term. For any
$\beta_{\mathrm{int}}\geq0$, the correction vanishes when the
entropy-variable traces coincide, $\bW^-=\bW^+$.
The numerical viscous flux uses complementary LDG weights:
\begin{equation}
  \widehat{\bm F}_{n,f}^{(V)}
  =\frac{1}{2}
   \left[
     (1-\alpha)\bm F_n^{(V),-}
     +(1+\alpha)\bm F_n^{(V),+}
   \right]
   +\frac{1}{2}\bm L_f\bm d_f.
  \label{eq:ldg-ip-viscous-trace}
\end{equation}
The LDG/IP
entropy contributions are given below and justified in
\ref{app:semidiscrete-entropy-identity}.

\paragraph{Element-local SAT data}
For $K\in\{K^-,K^+\}$, let
$\bn_K=(n_{1,K},\ldots,n_{d,K})^{\transpose}$ be its outward unit
normal, with $\bn_{K^-}=\bn$ and $\bn_{K^+}=-\bn$.
At each node of $f$, let $\bQ_{f,K}$, $\bW_{f,K}$, and
$\bm\theta_{j,f,K}\in\R^m$ denote the one-sided traces from $K$.
They are the nodewise blocks obtained by restricting the corresponding
element vectors with $\mathsf R_f\otimes\bm I_m$.
In particular, $\bW_{f,K^-}=\bW^-$ and $\bW_{f,K^+}=\bW^+$,
with analogous identifications for $\bQ$ and $\bm\theta_j$.
The physical outward fluxes on $f$ are
\begin{equation*}
  \bm F_{n,K}^{(I)}
  =n_{i,K}\bFI{i}(\bQ_{f,K}),
  \qquad
  \bm F_{n,K}^{(V)}
  =n_{i,K}\bm C_{ij}(\bQ_{f,K})\bm\theta_{j,f,K}.
\end{equation*}
For $\ell\in\{I,V\}$, the common-normal and element-outward fluxes
are related by
\begin{equation*}
  \begin{aligned}
    \bm F_{n,K^-}^{(\ell)}
    &=\bm F_n^{(\ell),-},
    &
    \bm F_{n,K^+}^{(\ell)}
    &=-\bm F_n^{(\ell),+},
    \\
    \widehat{\bm F}_{n,K^-}^{(\ell)}
    &=\widehat{\bm F}_{n,f}^{(\ell)},
    &
    \widehat{\bm F}_{n,K^+}^{(\ell)}
    &=-\widehat{\bm F}_{n,f}^{(\ell)}.
  \end{aligned}
\end{equation*}
The interior-face SAT data closing
Eqs.~\eqref{eq:semidiscrete-conservative-equation}
and~\eqref{eq:discrete-entropy-gradient} are
\begin{equation}
  \begin{aligned}
    \bm g_{i,f,K}^{\theta}
    &=n_{i,K}
      \left(
        \widehat{\bW}_f-\bW_{f,K}
      \right),
    \\
    \bm g_{f,K}^{q}
    &=\bm F_{n,K}^{(I)}
      -\widehat{\bm F}_{n,K}^{(I)}
      +\widehat{\bm F}_{n,K}^{(V)}
      -\bm F_{n,K}^{(V)}.
  \end{aligned}
  \label{eq:compact-interior-sat-data}
\end{equation}
The auxiliary SAT weakly couples $\bW_{f,K}$ to $\widehat{\bW}_f$.
The conservative SAT is the physical outward total flux
$\bm F_{n,K}^{(I)}-\bm F_{n,K}^{(V)}$ minus its numerical counterpart
$\widehat{\bm F}_{n,K}^{(I)}-\widehat{\bm F}_{n,K}^{(V)}$.
Together with the conservative SBP volume operator, it makes the element
balance use the common numerical flux. The numerical-flux contributions
cancel between neighboring elements because the face quadrature is
shared and the outward orientations are opposite.
Conservation of $\bQ$ alone does not imply entropy conservation, which
depends on the flux and trace choices as shown next.

\paragraph{Interface entropy contribution}
We combine the quadrature-weighted entropy contractions of the
conservative and auxiliary-gradient equations on $K^-$ and $K^+$
\cite{CarpenterEtAl2014,ParsaniInterface2015}, using the element identity
in Eq.~\eqref{app:eq:element-adapted-entropy-contraction}.
For the inviscid terms, the opposite outward normals produce trace and
potential differences. With $\bW^+-\bW^-=-\bm d_f$, Tadmor's identity
\cite{Tadmor2003} cancels the entropy-conservative part of
Eq.~\eqref{eq:merriam-interface-flux}, leaving
\begin{equation}
  (\bW^+-\bW^-)^{\transpose}
  \widehat{\bm F}_{n,f}^{(I)}
  -\left[
    \Psi_n(\bQ^+)-\Psi_n(\bQ^-)
   \right]
  =-\frac{1}{2}
    (\bW^+-\bW^-)^{\transpose}
    \bm D_f(\bW^+-\bW^-)
  \leq0.
  \label{eq:merriam-interface-entropy-identity}
\end{equation}
For the viscous terms, the auxiliary SAT contributions cancel the
complementary LDG weights in Eq.~\eqref{eq:ldg-ip-viscous-trace} for
every $\alpha\in[-1,1]$ \cite{ParsaniInterface2015}.
Therefore, the paired LDG coupling is entropy-conservative at the
interface, and only the IP contribution remains:
\begin{equation}
  \frac{1}{2}\bm d_f^{\transpose}\bm L_f\bm d_f
  \leq0.
  \label{eq:interface-ip-entropy-contribution}
\end{equation}
\ref{app:semidiscrete-entropy-identity} gives the trace-difference factors
and the sign change associated with the outward normal of $K^+$.
The sum of these nodewise contributions is nonpositive for
$\bm D_f\succeq\bm0$ and $\beta_{\mathrm{int}}\geq0$.
Positive face-quadrature weights preserve this sign after summation.
The contribution vanishes when $\bm D_f=\bm0$ and
$\beta_{\mathrm{int}}=0$, without eliminating the physical viscous
entropy production in the element volumes.

For constant $T_w$, Lemma~\ref{lem:affine-entropy} gives
\begin{equation}
  \bW_{T_w}^+-\bW_{T_w}^-
  =\bW^+-\bW^-,
  \qquad
  \p_i\bW_{T_w}=\p_i\bW,
  \qquad
  \Psi_{T_w,n}=\Psi_n.
  \label{eq:adapted-differences-and-potentials}
\end{equation}
Here, $\Psi_{T_w,n}
=\bW_{T_w}^{\transpose}\bm F_n^{(I)}-F_{T_w,n}$ is the adapted normal
entropy-flux potential. Since the LDG weights sum to one, the adapted
numerical entropy-variable trace satisfies
$\widehat{\bW}_{T_w,f}=\widehat{\bW}_f+\be_E/T_w$.
Thus, the auxiliary SAT differences are unchanged, so the same
interior SAT data and entropy dissipation apply to $S$ and $S_{T_w}$.
The following subsections specify the isothermal-wall closures.

\subsection{Common stationary isothermal-wall data and wall residual}
\label{subsec:isothermal-wall-common-data}

This subsection defines the data shared by the CR--SAT and FEE--SAT:
the inviscid mirror state and wall flux, entropy-variable wall value and
residual, auxiliary-gradient SAT, admissible wall state, and one-sided
viscous data. It then expresses the adapted wall-entropy contribution
and numerical outward total-energy flux in terms of
$\widehat{\bm F}_{n,w}^{(V)}$, whose two constructions are given in
Subsections~\ref{subsec:cr-sat} and~\ref{subsec:fee-sat}.

Consider a node on a stationary wall face
$f\subset\partial K\cap\wall$, with outward unit normal $\bn=\bn_K$,
wall velocity $\bu_w=\bm0$, and prescribed constant temperature
$\Tw>0$. The interior traces are
$\bQ^-=\bQ_{f,K}$, $\bW^-=\bW_{f,K}$, and
$\bm\theta_j^-=\bm\theta_{j,f,K}$.
The face and element indices are suppressed in the nodewise wall
notation, where $w$ identifies wall quantities.
A subscript $n$ on a flux, derivative, or entropy potential denotes
contraction with $\bn$, as in
Subsection~\ref{subsec:normal-convention}, and a hat denotes a numerical
boundary value or flux.
The superscript $+,(I)$ identifies a constructed inviscid mirror state,
not a trace from a neighboring element.

The impermeability condition (i.e., the no-penetration condition) 
is enforced inviscidly by reversing only the normal velocity
\cite{SvardOzcan2014,ParsaniCarpenterNielsen2015}:
\begin{equation}
  \rho^{+,(I)}=\rho^-,
  \qquad
  \bu^{+,(I)}=\bu^- -2(\bu^-\cdot\bn)\bn,
  \qquad
  T^{+,(I)}=T^-.
  \label{eq:inviscid-mirror-state}
\end{equation}
Let $\bQ^{+,(I)}$ denote the conservative representation of this
mirror state. The common numerical normal inviscid flux is
\begin{equation}
  \widehat{\bm F}_{n,w}^{(I)}
  =
  \begin{pmatrix}
    0\\
    p^-\bn\\
    0
  \end{pmatrix}.
  \label{eq:inviscid-entropy-conservative-wall-flux}
\end{equation}
Since $\Psi_n^-=p^-u_n^-/T^-$ and the affine shift changes only the
entropy-variable component associated with $\rho E$, we have
\begin{equation}
  (\bW^-)^{\transpose}\widehat{\bm F}_{n,w}^{(I)}-\Psi_n^-
  =
  (\bW_{T_w}^-)^{\transpose}\widehat{\bm F}_{n,w}^{(I)}-\Psi_n^-
  =
  0.
  \label{eq:inviscid-wall-entropy-cancellation}
\end{equation}
For the Ismail--Roe and Chandrashekar entropy-conservative fluxes
used in the numerical verification,
$\bm F_n^{\mathrm{ec}}(\bQ^-,\bQ^{+,(I)})$ reproduces
Eq.~\eqref{eq:inviscid-entropy-conservative-wall-flux}, as verified
in~\ref{app:closed-form-and-pointwise-details}. No characteristic dissipation
is added at the wall.

The no-slip and prescribed-temperature data are encoded by
\begin{equation}
  \widehat{\bW}_w
  =
  \begin{pmatrix}
    W_1^-\\
    \bm0\\
    -1/\Tw
  \end{pmatrix},
  \label{eq:canonical-entropy-boundary-value}
\end{equation}
with adapted counterpart
\begin{equation}
  \widehat{\bW}_{T_w,w}
  =
  \widehat{\bW}_w+\frac{\be_E}{\Tw}
  =
  \begin{pmatrix}
    W_1^-\\
    \bm0\\
    0
  \end{pmatrix}.
  \label{eq:adapted-entropy-boundary-value}
\end{equation}
The same constant affine shift is applied to the interior and numerical
wall entropy variables, so their differences coincide. We define the
common wall-data residual by
\begin{equation}
  \bm r_w
  =
  \bW^- -\widehat{\bW}_w
  =
  \bW_{T_w}^- -\widehat{\bW}_{T_w,w}
  =
  \begin{pmatrix}
    0\\[0.3ex]
    \bu^-/T^-\\[0.8ex]
    1/\Tw-1/T^-
  \end{pmatrix}.
  \label{eq:adapted-isothermal-wall-residual}
\end{equation}
The vector $\bm r_w\in\R^m$ measures the departure of the interior trace
from the prescribed no-slip and temperature data in entropy variables.
Its momentum components measure the velocity mismatch scaled by
$1/T^-$. Its energy component measures the temperature mismatch through
$1/\Tw-1/T^-=(T^--\Tw)/(\Tw T^-)$.
For an admissible interior state, $\bm r_w=\bm0$ if and only if
$\bu^-=\bm0$ and $T^-=\Tw$. The first component vanishes because the
first entropy variable is copied from the interior trace, rather than
specified independently as a boundary datum. The residual $\bm r_w$
is distinct from the conservative wall SAT.

The auxiliary-gradient wall SAT uses the negative of this residual,
weighted by the corresponding normal component:
\begin{equation}
  \bm g_{i,w}^{\theta}
  =
  n_i\left(\widehat{\bW}_w-\bW^-\right)
  =
  -n_i\bm r_w.
  \label{eq:wall-gradient-numerical-value}
\end{equation}
On a wall face $f\subset\partial K$, these nodewise vectors are stacked
into $\bm g_{i,f,K}^{\theta}\in\R^{mN_f}$. Their contribution to the
right-hand side of Eq.~\eqref{eq:discrete-entropy-gradient} is
\begin{equation*}
  \mathbb P_K^{-1}\mathbb E_{f,K}\bm g_{i,f,K}^{\theta}.
\end{equation*}
Here, $\mathbb E_{f,K}$ applies the physical face-quadrature weights
and injects the face data into the element vector. Multiplication by
$\mathbb P_K^{-1}$ applies the inverse element mass matrix.
This adds a weak boundary correction to the auxiliary-gradient
equation rather than imposing $\bm r_w=\bm0$ directly. It is the
prescribed-$\Tw$ counterpart of the auxiliary wall data in
Refs.~\cite{ParsaniCarpenterNielsen2015,DalcinEtAl2019}.

The same wall value determines the admissible conservative state used to
evaluate the viscous coefficient matrices:
\begin{equation}
  \widehat{\bQ}_w=\bQ(\widehat{\bW}_w),
  \label{eq:isothermal-wall-trace-state}
\end{equation}
with
\begin{equation}
  \widehat\rho_w>0,
  \qquad
  \widehat\bu_w=\bm0,
  \qquad
  \widehat T_w=\Tw.
  \label{eq:isothermal-wall-trace-primitive-data}
\end{equation}
The positive-density formula is derived in
\ref{app:viscous-boundary-state-admissibility}. No exterior
viscous-gradient state is introduced and all viscous data use the one-sided
LDG gradient.

The wall penalties introduced below share the same normal scaling and
coefficient matrices. For a wall face normal to Cartesian direction
$r$, let $h_{n,K}=h_{r,K}$ be the width of the wall-adjacent element
in that direction. The physical normal norm matrix is
$\mathsf P_{n,K}=(h_{n,K}/2)\mathsf P_{\Nsol}$, and $p_{n,w}>0$
is its wall endpoint entry, as in
Eq.~\eqref{eq:normal-endpoint-mass-entry}. Define
\begin{equation}
  \bm C_{nn,w}^{-}
  =n_i\bm C_{ij}(\bQ^-)n_j,
  \qquad
  \widehat{\bm C}_{nn,w}
  =n_i\bm C_{ij}(\widehat{\bQ}_w)n_j.
  \label{eq:wall-normal-viscous-matrices}
\end{equation}
For the admissible interior and wall states, these matrices satisfy
\begin{equation*}
  \bm C_{nn,w}^{-}\succeq\bm0,
  \qquad
  \widehat{\bm C}_{nn,w}\succeq\bm0.
\end{equation*}

Using the one-sided state and LDG gradient, let
$\bm t^-\in\R^d$ denote the viscous traction associated with the outward
wall normal, and let $q_n^-\in\R$ denote the outward Fourier heat flux:
\begin{equation}
  \bm t^-=\btau^-\bn,
  \qquad
  q_n^-=-\kappa(T^-)\,\p_nT^-.
  \label{eq:one-sided-traction-fourier-flux}
\end{equation}
Here, $\btau^-$ is the one-sided viscous stress tensor. Since $\btau^-$
and $\p_nT^-$ are assembled from the one-sided LDG gradient, $\bm t^-$
and $q_n^-$ are numerical approximations of the corresponding continuum
traction and Fourier heat flux. The resulting one-sided physical normal
viscous flux is
\begin{equation}
  \bm F_{n,w}^{(V),-}
  =
  \begin{pmatrix}
    0\\
    \bm t^-\\
    \bu^-\cdot\bm t^- -q_n^-
  \end{pmatrix}.
  \label{eq:one-sided-wall-viscous-flux}
\end{equation}

Equation~\eqref{eq:inviscid-wall-entropy-cancellation} establishes that
the common numerical inviscid wall flux contributes zero to the
pointwise $S_{T_w}$ balance, without requiring $\bm r_w=\bm0$.
The auxiliary-gradient wall SAT is already fixed by
Eq.~\eqref{eq:wall-gradient-numerical-value}. It remains to specify
the numerical viscous wall flux $\widehat{\bm F}_{n,w}^{(V)}$.
As for the physical viscous flux, its mass component must vanish.

Combining the conservative and auxiliary-gradient wall terms gives
the viscous contribution
\begin{align}
  \mathcal B_w^{(V)}
  &=
  (\bW_{T_w}^-)^{\transpose}
  \widehat{\bm F}_{n,w}^{(V)}
  -
  \bm r_w^{\transpose}\bm F_{n,w}^{(V),-}
  \notag\\
  &=
  \bm r_w^{\transpose}
  \left(
    \widehat{\bm F}_{n,w}^{(V)}
    -
    \bm F_{n,w}^{(V),-}
  \right).
  \label{eq:generic-discrete-wall-entropy-contribution}
\end{align}
The second term in the first line is the contribution of the
auxiliary-gradient SAT
$\bm g_{i,w}^{\theta}=-n_i\bm r_w$.
For the second equality, use
$\bW_{T_w}^-=\widehat{\bW}_{T_w,w}+\bm r_w$ and
\begin{equation*}
  \widehat{\bW}_{T_w,w}^{\transpose}
  \widehat{\bm F}_{n,w}^{(V)}
  =
  0.
\end{equation*}
This contraction vanishes because
$\widehat{\bW}_{T_w,w}=(W_1^-,\bm0,0)^{\transpose}$ has only its first
component potentially nonzero, whereas the numerical viscous flux has
zero mass component. The wall entropy contribution therefore depends
on the contraction of the wall-data residual with the correction to
the one-sided viscous flux.

The same numerical viscous flux also determines the energy exchange
through the wall. Since the inviscid wall flux in
Eq.~\eqref{eq:inviscid-entropy-conservative-wall-flux} has zero energy
component, the complete numerical outward total-energy flux is
\begin{equation}
  F_{E,n,h}
  =
  \be_E^{\transpose}
  \left(
    \widehat{\bm F}_{n,w}^{(I)}
    -
    \widehat{\bm F}_{n,w}^{(V)}
  \right)
  =
  -\be_E^{\transpose}\widehat{\bm F}_{n,w}^{(V)}.
  \label{eq:generic-numerical-wall-energy-flux}
\end{equation}
Equations~\eqref{eq:generic-discrete-wall-entropy-contribution}
and~\eqref{eq:generic-numerical-wall-energy-flux} separate the entropy
and energy effects of the numerical viscous wall flux. A vanishing
adapted wall entropy contribution does not require the numerical
energy flux to vanish. The following subsections specify
$\widehat{\bm F}_{n,w}^{(V)}$ for the CR--SAT and FEE--SAT formulations
and establish their respective entropy and energy-flux properties.

\subsection{Complete-residual SAT formulation}
\label{subsec:cr-sat}

Using the nodewise wall notation of
Subsection~\ref{subsec:isothermal-wall-common-data}, the CR--SAT chooses
a viscous-flux increment $-\bm\Lambda_w^{\mathrm{CR}}\bm r_w$, with
$\bm\Lambda_w^{\mathrm{CR}}\succeq\bm0$, to make the contribution in
Eq.~\eqref{eq:generic-discrete-wall-entropy-contribution} nonpositive.
The qualifier ``complete'' means that no projection is applied:
all nonzero momentum and temperature components of $\bm r_w$ enter
the correction.

To retain the constitutive and normal-mesh scaling of the viscous operator,
we construct the penalty matrix from the normal viscous coefficient matrices
evaluated at the interior and numerical wall states. For
$\beta_w\geq0$, define
\begin{equation}
  \bm\Lambda_w^{\mathrm{CR}}
  =
  \frac{\beta_w}{p_{n,w}}
  \left(
    \bm C_{nn,w}^{-}+\widehat{\bm C}_{nn,w}
  \right)
  \in\R^{m\times m},
  \qquad
  \bm\Lambda_w^{\mathrm{CR}}\succeq\bm0.
  \label{eq:cr-wall-penalty-matrix}
\end{equation}
Here, $\beta_w$ controls the penalty strength, while $p_{n,w}^{-1}$
provides the inverse wall-normal norm scaling.
The property $\bm\Lambda_w^{\mathrm{CR}}\succeq\bm0$ follows from
$\bm C_{nn,w}^{-}\succeq\bm0$,
$\widehat{\bm C}_{nn,w}\succeq\bm0$, and $\beta_w/p_{n,w}\geq0$.
With the one-sided normalization adopted here, no factor $1/2$ is
included, and $\beta_w$ is defined relative to this convention.

The corresponding entropy-stabilizing viscous-flux correction is
\begin{equation}
  \bm M_w^{\mathrm{CR}}
  =
  -\bm\Lambda_w^{\mathrm{CR}}\bm r_w
  \in\R^m.
  \label{eq:cr-wall-penalty}
\end{equation}
Its mass-flux component is zero because the first row of
$\bm\Lambda_w^{\mathrm{CR}}$ vanishes. The CR--SAT completes the numerical
viscous wall flux by setting
\begin{equation}
  \widehat{\bm F}_{n,w}^{(V),\mathrm{CR}}
  =
  \bm F_{n,w}^{(V),-}
  +\bm M_w^{\mathrm{CR}}.
  \label{eq:cr-wall-viscous-flux}
\end{equation}
The correction acts on the wall-data mismatch and vanishes when the
exact wall data are attained.

Substitution of Eq.~\eqref{eq:cr-wall-viscous-flux} into the generic
wall identity~\eqref{eq:generic-discrete-wall-entropy-contribution} gives
\begin{equation}
  \mathcal B_w^{(V),\mathrm{CR}}
  =
  -\bm r_w^{\transpose}
   \bm\Lambda_w^{\mathrm{CR}}
   \bm r_w
  \leq0.
  \label{eq:cr-wall-entropy-contribution}
\end{equation}
Hence, the CR--SAT is boundary-entropy-stable for every $\beta_w\geq0$
with respect to $S_{T_w}$, and its $\beta_w=0$ limit is
boundary-entropy-conservative.
Since $\bm\Lambda_w^{\mathrm{CR}}\succeq\bm0$, equality holds precisely
when $\bm\Lambda_w^{\mathrm{CR}}\bm r_w=\bm0$.

The numerical viscous flux also determines the nodewise SAT in the
conservative equation. Specializing the physical-minus-numerical
flux correction in Eq.~\eqref{eq:compact-interior-sat-data} to a wall
face gives
\begin{equation}
  \begin{aligned}
    \bm g_w^{q,\mathrm{CR}}
    &=
    \bm F_n^{(I),-}
    -\widehat{\bm F}_{n,w}^{(I)}
    +\widehat{\bm F}_{n,w}^{(V),\mathrm{CR}}
    -\bm F_{n,w}^{(V),-}
    \\
    &=
    \bm F_n^{(I),-}
    -\widehat{\bm F}_{n,w}^{(I)}
    +\bm M_w^{\mathrm{CR}}.
  \end{aligned}
  \label{eq:cr-complete-wall-sat}
\end{equation}
The first two terms weakly impose the common inviscid wall flux,
while $\bm M_w^{\mathrm{CR}}$ supplies the viscous wall correction.
On each wall face, the nodewise vectors
$\bm g_w^{q,\mathrm{CR}}\in\R^m$ are stacked into
$\bm g_{f,K}^{q}\in\R^{mN_f}$ and enter
Eq.~\eqref{eq:semidiscrete-conservative-equation} through
$\mathbb P_K^{-1}\mathbb E_{f,K}\bm g_{f,K}^{q}$.

The same choice determines the complete numerical outward total-energy
flux. Since
$\be_E^{\transpose}\widehat{\bm F}_{n,w}^{(I)}=0$, substitution of
Eq.~\eqref{eq:cr-wall-viscous-flux} into
Eq.~\eqref{eq:generic-numerical-wall-energy-flux} yields
\begin{equation}
  \begin{aligned}
    F_{E,n,h}^{\mathrm{CR}}
    &=
    \be_E^{\transpose}
    \left(
      \widehat{\bm F}_{n,w}^{(I)}
      -\widehat{\bm F}_{n,w}^{(V),\mathrm{CR}}
    \right)
    \\
    &=
    q_n^-
    -\bu^-\!\cdot\bm t^-
    -\be_E^{\transpose}\bm M_w^{\mathrm{CR}}.
  \end{aligned}
  \label{eq:cr-wall-numerical-energy-flux}
\end{equation}
At finite resolution, this flux generally contains both the one-sided
mechanical-work term and the total-energy component of the penalty. When
the exact wall data are attained, $\bu^-=\bm0$ and $\bm r_w=\bm0$, so these
terms vanish and
Eq.~\eqref{eq:cr-wall-numerical-energy-flux} reduces to $q_n^-$. With a
consistent auxiliary-gradient approximation, $q_n^-$ converges to the
physical Fourier heat flux.

\begin{theorem}[CR--SAT wall balance]
\label{thm:cr-sat-wall-balance}
Let $d\in\{2,3\}$, let $\bn$ be directed outward from the fluid domain, and
assume the constitutive conditions in Eq.~\eqref{eq:transport-assumptions}.
At a stationary wall, assume $\rho^->0$, $T^->0$, $T_w>0$,
$p_{n,w}>0$, and $\beta_w\geq0$, and use the common wall data of
Subsection~\ref{subsec:isothermal-wall-common-data}. Then the common wall
treatment completed by
$\widehat{\bm F}_{n,w}^{(V),\mathrm{CR}}$ is consistent with the stationary
no-slip isothermal boundary conditions. Its inviscid contribution to the
pointwise $S_{T_w}$ balance is zero, and its viscous contribution is
\begin{equation*}
  -\bm r_w^{\transpose}
   \bm\Lambda_w^{\mathrm{CR}}
   \bm r_w
  \leq0.
\end{equation*}
The contribution is identically zero when $\beta_w=0$. The associated
semi-discrete fluid-energy balance uses the numerical outward total-energy
flux in Eq.~\eqref{eq:cr-wall-numerical-energy-flux}.
\end{theorem}

\begin{proof}
When the exact wall data are attained, $\bu^-=\bm0$, $T^-=T_w$, and
$\bm r_w=\bm0$. Hence, the auxiliary-gradient wall SAT and
$\bm M_w^{\mathrm{CR}}$ vanish,
$\widehat{\bm F}_{n,w}^{(V),\mathrm{CR}}
=\bm F_{n,w}^{(V),-}$, and the one-sided physical inviscid flux equals
$\widehat{\bm F}_{n,w}^{(I)}=(0,p^-\bn,0)^{\transpose}$. Thus, both
wall SAT corrections vanish, proving consistency. The entropy statement
follows from
Eq.~\eqref{eq:generic-discrete-wall-entropy-contribution} and
\begin{equation*}
  \widehat{\bm F}_{n,w}^{(V),\mathrm{CR}}
  -\bm F_{n,w}^{(V),-}
  =
  -\bm\Lambda_w^{\mathrm{CR}}\bm r_w,
\end{equation*}
together with
$\bm\Lambda_w^{\mathrm{CR}}\succeq\bm0$. The energy statement follows by
taking the total-energy component of the complete numerical wall flux.
\end{proof}

\subsection{Fourier-energy-exact SAT formulation}
\label{subsec:fee-sat}

The FEE--SAT retains the common wall data and fluid-side trace convention
of Subsection~\ref{subsec:isothermal-wall-common-data}, with
$\bn=\bn_K$. In addition to the adapted wall entropy property, it
requires the complete numerical outward total-energy flux to satisfy
\begin{equation*}
  F_{E,n,h}^{\mathrm{FEE}}=q_n^-
\end{equation*}
at every wall node and spatial resolution. Since
$\be_E^{\transpose}\widehat{\bm F}_{n,w}^{(I)}=0$ and
\begin{equation*}
  \be_E^{\transpose}\bm F_{n,w}^{(V),-}
  =\bu^-\!\cdot\bm t^- -q_n^-,
\end{equation*}
the correction to the one-sided viscous flux must contribute
$-\bu^-\!\cdot\bm t^-$ to its total-energy component.
The FEE--SAT separates this adjustment from optional entropy dissipation:
an entropy-neutral skew correction cancels the mechanical-work term,
while a momentum-projected penalty adds dissipation without changing
the total-energy flux.

Using the density--momentum--total-energy block ordering, define
$\bm K_w\in\R^{m\times m}$ by
\begin{equation}
  \bm K_w
  =
  T^-
  \begin{pmatrix}
    0 & \bm0^{\transpose} & 0
    \\
    \bm0 & \bm0_{d\times d} & \bm t^-
    \\
    0 & -(\bm t^-)^{\transpose} & 0
  \end{pmatrix},
  \qquad
  \bm K_w^{\transpose}=-\bm K_w.
  \label{eq:fee-skew-matrix}
\end{equation}
The factor $T^-$ is chosen so that multiplication by the wall-data
residual gives
\begin{equation}
  \bm K_w\bm r_w
  =
  \begin{pmatrix}
    0
    \\
    (T^-/T_w-1)\bm t^-
    \\
    -\bu^-\!\cdot\bm t^-
  \end{pmatrix},
  \qquad
  \bm r_w^{\transpose}\bm K_w\bm r_w=0.
  \label{eq:fee-skew-correction}
\end{equation}
The energy component cancels the one-sided mechanical-work term.
The companion momentum component is required by
$\bm K_w^{\transpose}=-\bm K_w$ and vanishes when $T^-=T_w$.
Equation~\eqref{eq:generic-discrete-wall-entropy-contribution} shows
that this flux increment contributes
$\bm r_w^{\transpose}\bm K_w\bm r_w=0$ to the pointwise $S_{T_w}$
balance, so the energy-flux adjustment is entropy-neutral.

To add entropy dissipation without changing either the mass-flux or
total-energy component, introduce the orthogonal projector onto the momentum
block,
\begin{equation}
  \bm\Pi_u
  =
  \begin{pmatrix}
    0 & \bm0^{\transpose} & 0
    \\
    \bm0 & \bm I & \bm0
    \\
    0 & \bm0^{\transpose} & 0
  \end{pmatrix}
  \in\R^{m\times m}.
  \label{eq:fee-momentum-projector}
\end{equation}
Here, $\bm I$ is the $d\times d$ identity, and
$\bm\Pi_u^{\transpose}=\bm\Pi_u=\bm\Pi_u^2$.
Using the same wall-penalty parameter $\beta_w\geq0$, normal scale
$p_{n,w}>0$, and one-sided normalization as the CR--SAT, define
\begin{align}
  \bm\Lambda_{w,u}^{\mathrm{FEE}}
  &=
  \frac{\beta_w}{p_{n,w}}
  \bm\Pi_u
  \left(
    \bm C_{nn,w}^{-}+\widehat{\bm C}_{nn,w}
  \right)
  \bm\Pi_u
  \in\R^{m\times m},
  \qquad
  \bm\Lambda_{w,u}^{\mathrm{FEE}}\succeq\bm0,
  \label{eq:fee-projected-penalty-matrix}
  \\
  \bm M_{w,u}^{\mathrm{FEE}}
  &=
  -\bm\Lambda_{w,u}^{\mathrm{FEE}}\bm r_w
  \in\R^m.
  \label{eq:fee-projected-penalty}
\end{align}
Since $\bm C_{nn,w}^{-}+\widehat{\bm C}_{nn,w}\succeq\bm0$,
$\bm\Pi_u^{\transpose}=\bm\Pi_u$, and $\beta_w/p_{n,w}\geq0$,
projection on both sides gives
$\bm\Lambda_{w,u}^{\mathrm{FEE}}\succeq\bm0$.
The quadratic-form argument is given in
\ref{app:fee-wall-identities}.
The projected matrix has zero first and last rows and columns, so
$\bm M_{w,u}^{\mathrm{FEE}}$ has only momentum components and
$\be_E^{\transpose}\bm M_{w,u}^{\mathrm{FEE}}=0$.

The FEE--SAT numerical viscous wall flux is
\begin{equation}
  \widehat{\bm F}_{n,w}^{(V),\mathrm{FEE}}
  =
  \bm F_{n,w}^{(V),-}
  +\bm K_w\bm r_w
  +\bm M_{w,u}^{\mathrm{FEE}}.
  \label{eq:fee-wall-viscous-flux}
\end{equation}
All three terms have zero mass-flux component. Substitution into the generic
pointwise wall identity
\eqref{eq:generic-discrete-wall-entropy-contribution} gives
\begin{equation}
  \mathcal B_w^{(V),\mathrm{FEE}}
  =
  -\bm r_w^{\transpose}
   \bm\Lambda_{w,u}^{\mathrm{FEE}}
   \bm r_w
  \leq0,
  \label{eq:fee-wall-entropy-contribution}
\end{equation}
because $\bm r_w^{\transpose}\bm K_w\bm r_w=0$. Therefore,
the FEE--SAT is boundary-entropy-stable for every
$\beta_w\geq0$ with respect to $S_{T_w}$, and its $\beta_w=0$ limit
is boundary-entropy-conservative.
Since $\bm\Lambda_{w,u}^{\mathrm{FEE}}\succeq\bm0$, equality holds
precisely when $\bm\Lambda_{w,u}^{\mathrm{FEE}}\bm r_w=\bm0$.
In particular, when $\bu^-=\bm0$, the projected residual
$\bm\Pi_u\bm r_w$ vanishes even if $T^-\neq T_w$.
The penalty therefore does not directly dissipate a pure temperature
mismatch. The prescribed temperature remains weakly imposed through
the common entropy-variable wall value in
Eq.~\eqref{eq:canonical-entropy-boundary-value} and the auxiliary-gradient
SAT in Eq.~\eqref{eq:wall-gradient-numerical-value}.

The total-energy component of
Eq.~\eqref{eq:fee-wall-viscous-flux} satisfies
\begin{equation*}
  \be_E^{\transpose}
  \widehat{\bm F}_{n,w}^{(V),\mathrm{FEE}}
  =
  \bu^-\!\cdot\bm t^- -q_n^-
  -\bu^-\!\cdot\bm t^-
  =
  -q_n^-.
\end{equation*}
Therefore, the complete numerical outward total-energy flux is
\begin{equation}
  F_{E,n,h}^{\mathrm{FEE}}
  =
  \be_E^{\transpose}
  \left(
    \widehat{\bm F}_{n,w}^{(I)}
    -\widehat{\bm F}_{n,w}^{(V),\mathrm{FEE}}
  \right)
  =
  q_n^-.
  \label{eq:fee-fourier-energy-identity}
\end{equation}
This identity holds pointwise at every spatial resolution and for every
$\beta_w\geq0$, without requiring the interior trace to attain the exact
wall data. Here, $q_n^-$ is the one-sided numerical Fourier heat flux
computed from the LDG gradient. Fourier-energy exactness refers to
this numerical flux identity, not to an exact evaluation of the
continuum heat flux.

The same physical-minus-numerical flux correction used for the CR--SAT
now gives the nodewise conservative wall SAT
\begin{equation}
  \begin{aligned}
    \bm g_w^{q,\mathrm{FEE}}
    &=
    \bm F_n^{(I),-}
    -\widehat{\bm F}_{n,w}^{(I)}
    +\widehat{\bm F}_{n,w}^{(V),\mathrm{FEE}}
    -\bm F_{n,w}^{(V),-}
    \\
    &=
    \bm F_n^{(I),-}
    -\widehat{\bm F}_{n,w}^{(I)}
    +\bm K_w\bm r_w
    +\bm M_{w,u}^{\mathrm{FEE}}.
  \end{aligned}
  \label{eq:fee-complete-wall-sat}
\end{equation}
The first two terms weakly impose the common inviscid wall flux.
The skew term adjusts the viscous flux without contributing to the
adapted wall entropy balance, while the projected penalty supplies
entropy dissipation. On each wall face, the nodewise vectors
$\bm g_w^{q,\mathrm{FEE}}\in\R^m$ are stacked into
$\bm g_{f,K}^{q}\in\R^{mN_f}$ and enter
Eq.~\eqref{eq:semidiscrete-conservative-equation} through
$\mathbb P_K^{-1}\mathbb E_{f,K}\bm g_{f,K}^{q}$.
When the exact wall data are attained, $\bm r_w=\bm0$, so both
formulation-specific corrections vanish and the common wall treatment
is consistent. We report the supporting identities in
\ref{app:fee-wall-identities}.

\begin{theorem}[FEE--SAT wall balance]
\label{thm:fee-sat-wall-balance}
Under the assumptions of Theorem~\ref{thm:cr-sat-wall-balance}, the common
wall treatment completed by
$\widehat{\bm F}_{n,w}^{(V),\mathrm{FEE}}$ is consistent with the stationary
no-slip isothermal boundary conditions. Its inviscid contribution to the
pointwise $S_{T_w}$ balance is zero, and its viscous contribution is
\begin{equation*}
  -\bm r_w^{\transpose}
   \bm\Lambda_{w,u}^{\mathrm{FEE}}
   \bm r_w
  \leq0.
\end{equation*}
The contribution is identically zero when $\beta_w=0$. For every
$\beta_w\geq0$, the numerical outward total-energy flux satisfies
$F_{E,n,h}^{\mathrm{FEE}}=q_n^-$ pointwise.
\end{theorem}

\begin{proof}
At the exact wall data, $\bm r_w=\bm0$. Hence, the auxiliary-gradient wall
SAT, $\bm K_w\bm r_w$, and $\bm M_{w,u}^{\mathrm{FEE}}$ vanish, while
$\bu^-=\bm0$ makes the one-sided physical inviscid flux equal to
$\widehat{\bm F}_{n,w}^{(I)}$. Thus, both wall SAT corrections vanish,
proving consistency. The entropy identity follows from
Eq.~\eqref{eq:generic-discrete-wall-entropy-contribution},
$\bm K_w^{\transpose}=-\bm K_w$, and
$\bm\Lambda_{w,u}^{\mathrm{FEE}}\succeq\bm0$. The zero total-energy row of
$\bm\Lambda_{w,u}^{\mathrm{FEE}}$ and the last component in
Eq.~\eqref{eq:fee-skew-correction} give
Eq.~\eqref{eq:fee-fourier-energy-identity}.
\end{proof}

\subsection{Global source-free balances and structural comparison}
\label{subsec:two-sat-global-balances}

This subsection assembles the element and interface contributions and
the local CR--SAT and FEE--SAT wall identities into global source-free
adapted-entropy, total-energy, and canonical-entropy balances.
Table~\ref{tab:two-sat-structural-comparison} summarizes the shared
entropy structure and the distinct finite-resolution energy fluxes.

Let $\mathcal F_h^{\mathrm{int}}$ denote the set of conforming interior
interfaces, including identified periodic face pairs, and let
$\mathcal F_h^w$ denote the set of isothermal-wall faces.
Each interface or periodic pair is counted once, using the common
node ordering and face quadrature of
Subsection~\ref{subsec:interface-2015-coupling}.
Periodic pairs use the same coupling and satisfy the same interface
assumptions. In face-quadrature products, nodewise vectors are stacked
on the current face and nodewise matrices are assembled block diagonally.
The common viscous-volume, inviscid-interface, and IP dissipations are
\begin{equation}
  \begin{aligned}
  \Dcal_h^{(V)}
  &=
  \sum_{K\in\mathcal T_h}
  \bm\theta_{i,K}^{\transpose}
  \mathbb P_K[\bm C_{ij}]_K\bm\theta_{j,K},
  \\
  \Dcal_h^{(I,\mathrm{int})}
  &=
  \frac12
  \sum_{f\in\mathcal F_h^{\mathrm{int}}}
  \left\langle
    \bW^+-\bW^-,
    \bm D_f(\bW^+-\bW^-)
  \right\rangle_{f,h}^{(m)},
  \\
  \Dcal_h^{(\mathrm{IP,int})}
  &=
  -\frac12
  \sum_{f\in\mathcal F_h^{\mathrm{int}}}
  \left\langle
    \bm d_f,
    \bm L_f\bm d_f
  \right\rangle_{f,h}^{(m)}.
  \end{aligned}
\end{equation}
The formulation-specific wall dissipations are
\begin{equation}
  \begin{aligned}
  \Dcal_h^{(w,\mathrm{CR})}
  &=
  \sum_{f\in\mathcal F_h^w}
  \left\langle
    \bm r_w,
    \bm\Lambda_w^{\mathrm{CR}}\bm r_w
  \right\rangle_{f,h}^{(m)},
  \\
  \Dcal_h^{(w,\mathrm{FEE})}
  &=
  \sum_{f\in\mathcal F_h^w}
  \left\langle
    \bm r_w,
    \bm\Lambda_{w,u}^{\mathrm{FEE}}\bm r_w
  \right\rangle_{f,h}^{(m)}.
  \end{aligned}
\end{equation}
For $X\in\{\mathrm{CR},\mathrm{FEE}\}$, define
\begin{equation*}
  \Dcal_h^X
  =
  \Dcal_h^{(V)}
  +\Dcal_h^{(I,\mathrm{int})}
  +\Dcal_h^{(\mathrm{IP,int})}
  +\Dcal_h^{(w,X)}.
\end{equation*}
Let $S_{T_w,K},S_K,(\rho E)_K\in\R^{N_K}$ collect the nodal values of
$S_{T_w}(\bQ)$, $S(\bQ)$, and $\rho E$, respectively, on the element $K$.
Under the assumptions stated below, every term in $\Dcal_h^X$ is
nonnegative. Let $\Bcal_{T_w,h}^{\mathrm{other}}$ denote the total
adapted-entropy contribution from the unpaired portion of
$\partial\Omega\setminus\wall$, with the outward-normal convention of
Eq.~\eqref{eq:global-adapted-entropy}.
Let $\mathcal F_{E,h}^{\mathrm{other}}$ denote the integrated outward
numerical total-energy flux through that same portion.
The dissipations of identified periodic pairs are included in
$\Dcal_h^{(I,\mathrm{int})}$ and $\Dcal_h^{(\mathrm{IP,int})}$,
not in $\Bcal_{T_w,h}^{\mathrm{other}}$.

\begin{theorem}[Global source-free adapted-entropy balance]
\label{thm:global-source-free-two-sat-balances}
Assume admissible nodal states and the constitutive conditions in
Eq.~\eqref{eq:transport-assumptions}. Let one spatially and temporally
constant wall temperature $T_w>0$ be prescribed on $\wall$. Let the
conforming Cartesian affine tensor-product elements satisfy
Eq.~\eqref{eq:multidimensional-sbp-property}, and use a symmetric,
consistent two-point entropy-conservative volume flux satisfying
Eq.~\eqref{eq:tadmor-identity-section5}. On each interface in
$\mathcal F_h^{\mathrm{int}}$, including each identified periodic pair,
use the coupling of Subsection~\ref{subsec:interface-2015-coupling},
including
$p_{n,f}^-=p_{n,f}^+=p_{n,f}>0$, with $\bm D_f\succeq\bm0$ and
$\beta_{\mathrm{int}}\geq0$. On all wall faces, use the same formulation, 
either the CR--SAT or the FEE--SAT, with $\beta_w\geq0$. 
If the semi-discrete equations are source
free, then, for $X\in\{\mathrm{CR},\mathrm{FEE}\}$,
\begin{equation}
  \frac{\mathrm d}{\mathrm dt}
  \sum_{K\in\mathcal T_h}
  \langle1,S_{T_w,K}\rangle_{K,h}
  +\Dcal_h^X
  =\Bcal_{T_w,h}^{\mathrm{other}}.
  \label{eq:two-sat-adapted-entropy-balance}
\end{equation}
If $\Bcal_{T_w,h}^{\mathrm{other}}=0$, in particular when all non-wall
boundary faces are paired periodically, the discrete adapted entropy is
nonincreasing. Multiplication by $T_w$ gives the corresponding discrete
total-ballistic-energy balance.
\end{theorem}

\begin{proof}
Sum the element entropy identity in
Eq.~\eqref{app:eq:element-adapted-entropy-contraction} over
$K\in\mathcal T_h$. The entropy-conservative inviscid contributions cancel across each
interior interface or identified periodic pair, while
$\bm D_f\succeq\bm0$ yields $\Dcal_h^{(I,\mathrm{int})}$.
The viscous-volume contributions sum to $\Dcal_h^{(V)}$.
The complementary LDG terms cancel, and $\bm L_f\preceq\bm0$ gives
$\Dcal_h^{(\mathrm{IP,int})}$.
Theorems~\ref{thm:cr-sat-wall-balance}
and~\ref{thm:fee-sat-wall-balance} provide $\Dcal_h^{(w,X)}$.
The remaining unpaired non-wall boundary contributions are collected
in $\Bcal_{T_w,h}^{\mathrm{other}}$.
\end{proof}

Define the discrete fluid total energy by
\begin{equation}
  \mathcal E_h(t)
  =
  \sum_{K\in\mathcal T_h}
    \langle1,(\rho E)_K\rangle_{K,h}.
  \label{eq:discrete-fluid-total-energy}
\end{equation}
Contracting Eq.~\eqref{eq:semidiscrete-conservative-equation} with
$(\bm1_K\otimes\be_E)^{\transpose}\mathbb P_K$ and summing over
elements cancels the common numerical fluxes across interior
interfaces and identified periodic pairs. Using the wall fluxes in
Eqs.~\eqref{eq:cr-wall-numerical-energy-flux}
and~\eqref{eq:fee-fourier-energy-identity}, the source-free energy
balances are
\begin{align}
  \frac{\mathrm d\mathcal E_h}{\mathrm dt}
  +\sum_{f\in\mathcal F_h^w}
   \left\langle F_{E,n,h}^{\mathrm{CR}},1\right\rangle_{f,h}
  +\mathcal F_{E,h}^{\mathrm{other}}
  &=0,
  \label{eq:cr-global-energy-balance}
  \\
  \frac{\mathrm d\mathcal E_h}{\mathrm dt}
  +\sum_{f\in\mathcal F_h^w}
   \left\langle q_n^-,1\right\rangle_{f,h}
  +\mathcal F_{E,h}^{\mathrm{other}}
  &=0.
  \label{eq:fee-global-energy-balance}
\end{align}
Thus, the CR--SAT uses its complete numerical outward total-energy flux, whereas
the FEE--SAT makes that flux equal to the one-sided numerical Fourier heat flux
at every wall node. The integral of the fluid total energy does not have to be 
constant because an
isothermal wall is an open energy boundary.

Since $S=S_{T_w}-\rho E/T_w$, define
\begin{equation}
  \Bcal_{S,h}^{\mathrm{other}}
  =
  \Bcal_{T_w,h}^{\mathrm{other}}
  +\frac{\mathcal F_{E,h}^{\mathrm{other}}}{T_w}.
  \label{eq:other-canonical-boundary-contribution}
\end{equation}
The canonical-entropy balances are
\begin{align}
  \frac{\mathrm d}{\mathrm dt}
  \sum_{K\in\mathcal T_h}\langle1,S_K\rangle_{K,h}
  +\Dcal_h^{\mathrm{CR}}
  &=
  \frac1{T_w}
  \sum_{f\in\mathcal F_h^w}
  \left\langle F_{E,n,h}^{\mathrm{CR}},1\right\rangle_{f,h}
  +\Bcal_{S,h}^{\mathrm{other}},
  \label{eq:cr-canonical-entropy-balance}
  \\
  \frac{\mathrm d}{\mathrm dt}
  \sum_{K\in\mathcal T_h}\langle1,S_K\rangle_{K,h}
  +\Dcal_h^{\mathrm{FEE}}
  &=
  \frac1{T_w}
  \sum_{f\in\mathcal F_h^w}
  \left\langle q_n^-,1\right\rangle_{f,h}
  +\Bcal_{S,h}^{\mathrm{other}}.
  \label{eq:fee-canonical-entropy-balance}
\end{align}
Neither relation is a homogeneous non-increase statement for the canonical
fluid entropy alone.

\begin{remark}[Conservative volume sources]
\label{rem:conservative-volume-sources}
If the nodal conservative source on element $K$ is
$\bm G_K\in\R^{mN_K}$, the following terms are added to the right-hand
sides of the adapted-entropy, total-energy, and canonical-entropy
balances, respectively:
\begin{equation*}
  \sum_{K\in\mathcal T_h}
  \langle\bm w_{T_w,K},\bm G_K\rangle_{K,h}^{(m)},
  \qquad
  \sum_{K\in\mathcal T_h}
  \langle\bm1_K\otimes\be_E,\bm G_K\rangle_{K,h}^{(m)},
  \qquad
  \sum_{K\in\mathcal T_h}
  \langle\bm w_K,\bm G_K\rangle_{K,h}^{(m)}.
\end{equation*}
These terms are not part of the source-free theorem.
\end{remark}

\begin{table}[t]
  \centering
  \caption{Structural comparison of the two isothermal-wall SAT
  formulations. Both use the common data of
  Subsection~\ref{subsec:isothermal-wall-common-data}.}
  \label{tab:two-sat-structural-comparison}
  \begin{tabularx}{\linewidth}{
    >{\raggedright\arraybackslash}p{0.25\linewidth}
    >{\raggedright\arraybackslash}X
    >{\raggedright\arraybackslash}X}
    \toprule
    Property & CR--SAT & FEE--SAT \\
    \midrule
    Viscous-flux completion
    & $\bm F_{n,w}^{(V),-}+\bm M_w^{\mathrm{CR}}$
    & $\bm F_{n,w}^{(V),-}+\bm K_w\bm r_w
      +\bm M_{w,u}^{\mathrm{FEE}}$ \\
    Penalty action
    & Complete wall-data residual
    & Momentum projection of the wall-data residual \\
    Numerical viscous energy component
    & $\bu^-\!\cdot\bm t^- -q_n^-
      +\be_E^{\transpose}\bm M_w^{\mathrm{CR}}$
    & $-q_n^-$ \\
    Numerical outward energy flux
    & $F_{E,n,h}^{\mathrm{CR}}$
    & $q_n^-$ \\
    Pointwise $S_{T_w}$ wall contribution
    & $-\bm r_w^{\transpose}\bm\Lambda_w^{\mathrm{CR}}\bm r_w
      \leq0$
    & $-\bm r_w^{\transpose}
      \bm\Lambda_{w,u}^{\mathrm{FEE}}\bm r_w\leq0$ \\
    Additional wall operations
    & One complete viscous-matrix residual product
    & One skew correction and one projected matrix product \\
    \bottomrule
  \end{tabularx}
\end{table}

\paragraph{Implementation}
The nodewise wall assembly is summarized in
\ref{app:wall-implementation-sequences}.

\section{Numerical verification}
\label{sec:results}

In this section, we numerically study the wall formulations from their local
algebraic identities to their behavior in assembled flow calculations,
complementing the proofs in Sections~\ref{sec:canonical-wall}--\ref{sec:semidiscrete-wall}.
Unless stated otherwise, calculations are performed in double-precision 
floating-point format.

We perform randomized wall tests to check the local energy and entropy
identities and the sign of the wall entropy contribution independently
of mesh assembly and time integration.
\ref{app:closed-form-and-pointwise-details} contains the test
definitions, random-input specifications, closed-form flux checks, and
complete pointwise tables. In both two and three dimensions, the matched
randomized wall tests satisfy the scaled equality tolerance of
$10^{-11}$, with a nonpositive wall contribution in every
entropy-stable test.

Building on these local checks, a forced one-dimensional
manufactured-solution study compares the spatial accuracy of the CR--SAT and
the FEE--SAT under identical volume and interior-interface operators, while
thermal relaxation tests the assembled total-energy and adapted-entropy
balances for all four combinations of wall formulation and wall-penalty
setting. The assessment then extends to three-dimensional calculations
that examine adapted-entropy balance closure, convergence of weak wall
traces on a curved geometry, and operation in a supersonic separated flow.

The balanced wall campaign and all assembled flow calculations use the
Chandrashekar entropy-conservative two-point inviscid flux~\cite{Chandrashekar2013}.
The one-dimensional simulations use an in-house Fortran code, whereas
the three-dimensional simulations use the in-house solver named
SSDC~\cite{ParsaniEtAl2021}, which is written in Fortran, C, and C++.
Its input files use YAML Ain't Markup Language (YAML).

\subsection{Common semi-discrete balance diagnostics}
\label{subsec:balance-diagnostics}

This subsection defines the instantaneous balance diagnostics used in
the thermal-relaxation test in
Subsection~\ref{subsec:thermal-relaxation-comparison}. The adapted-entropy
diagnostics are also used for the three-dimensional closed-domain test in
Subsection~\ref{subsubsec:three-dimensional-adapted-entropy}.
Both tests have a fixed mesh, no volumetric sources, and stationary
isothermal walls at one constant temperature on the entire boundary.

Let $X\in\{\mathrm{CR},\mathrm{FEE}\}$ identify the wall formulation.
Using $\mathcal E_h$ from
Eq.~\eqref{eq:discrete-fluid-total-energy}, define the discrete total
adapted entropy by
\begin{equation}
  \mathcal S_{T_w,h}(t)
  =
  \sum_{K\in\mathcal T_h}
  \langle1,S_{T_w,K}(t)\rangle_{K,h}
  =
  \sum_{K\in\mathcal T_h}
  \bm1_K^{\transpose}\mathsf P_K S_{T_w,K}(t).
  \label{eq:balance-discrete-adapted-entropy}
\end{equation}
The scalar quadrature matrix $\mathsf P_K$ contains the physical volume
weights, including the nodal Jacobian factors on a curvilinear mesh. Its vector 
extension is 
$\mathbb P_K$. Since these weights and $T_w$ are
constant in time, the conservative-variable contractions give
\begin{align}
  \frac{\mathrm d\mathcal E_h}{\mathrm dt}
  &=
  \sum_{K\in\mathcal T_h}
  (\bm1_K\otimes\be_E)^{\transpose}
  \mathbb P_K\frac{\mathrm d\bm q_K}{\mathrm dt},
  \label{eq:balance-energy-rate}
  \\
  \frac{\mathrm d\mathcal S_{T_w,h}}{\mathrm dt}
  &=
  \sum_{K\in\mathcal T_h}
  \bm w_{T_w,K}^{\transpose}
  \mathbb P_K\frac{\mathrm d\bm q_K}{\mathrm dt}.
  \label{eq:balance-adapted-entropy-rate}
\end{align}
Here, $\be_E$ selects the total-energy component, and
$\bm w_{T_w,K}$ contains the adapted, not the canonical, entropy
variables defined in Subsection~\ref{subsec:first-order-viscous-sbp}.
At each control state, $\mathrm d\bm q_K/\mathrm dt$ is obtained
from the complete assembled conservative right-hand side, including
the volume, interior-face, and wall contributions. 

Separately, the dissipation $\Dcal_h^X$ is evaluated from the quadratic
forms given in Subsection~\ref{subsec:two-sat-global-balances}. In particular,
its viscous-volume contribution is
\begin{equation}
  \Dcal_h^{(V)}(t)
  =
  \sum_{K\in\mathcal T_h}\sum_{i,j=1}^{d}
  \bm\theta_{i,K}^{\transpose}
  \mathbb P_K[\bm C_{ij}]_K\bm\theta_{j,K}
  \geq0.
  \label{eq:balance-viscous-volume-dissipation}
\end{equation}
This term represents viscosity and Fourier heat conduction.
The dissipation uses the auxiliary gradients and coefficient matrices of
the viscous operator, including the interface and wall liftings in the
gradients. The total
$\Dcal_h^X$ also includes the active inviscid-interface, IP, and wall
quadratic forms, $\Dcal_h^{(I,\mathrm{int})}$,
$\Dcal_h^{(\mathrm{IP,int})}$, and $\Dcal_h^{(w,X)}$, respectively.
All contributions are nonnegative under the stated admissibility,
constitutive, and penalty-matrix conditions. The paired LDG terms cancel
in the assembled entropy identity and do not provide a separate
dissipation.

The one-sided outward Fourier heat flux and the complete numerical
outward energy flux are integrated independently of the energy rate:
\begin{equation}
  \begin{aligned}
    \mathcal Q_{F,h}(t)
    &=
    \sum_{f\in\mathcal F_h^w}
    \langle q_n^-,1\rangle_{f,h},
    \qquad
    \mathcal F_{E,h}^{X}(t)
    &=
    \sum_{f\in\mathcal F_h^w}
    \langle F_{E,n,h}^{X},1\rangle_{f,h}.
  \end{aligned}
  \label{eq:thermal-relaxation-energy-quantities}
\end{equation}
The rates, dissipation, and wall fluxes use the same instantaneous state
and the gradients used by the spatial operator. Independence means that
the rates are obtained from the conservative contractions, not assigned
from the negatives of the dissipation or boundary-flux sums. Hence, no time
difference of successive energy or entropy values is used.

The complete numerical-energy residual, Fourier-only mismatch, and
adapted-entropy residual are
\begin{equation}
  \begin{aligned}
  \mathcal R_{E,h}^{X}
  &=
  \frac{\mathrm d\mathcal E_h}{\mathrm dt}
  +\mathcal F_{E,h}^{X},
  \qquad
  \mathcal M_{F,h}^{X}
  &=
  \frac{\mathrm d\mathcal E_h}{\mathrm dt}
  +\mathcal Q_{F,h},
  \qquad
  \mathcal R_{S_{T_w},h}^{X}
  &=
  \frac{\mathrm d\mathcal S_{T_w,h}}{\mathrm dt}
  +\Dcal_h^{X}.
  \end{aligned}
  \label{eq:thermal-relaxation-adapted-entropy-residual}
\end{equation}
The source-free balances in
Subsection~\ref{subsec:two-sat-global-balances} predict
$\mathcal R_{E,h}^{X}=0$ and $\mathcal R_{S_{T_w},h}^{X}=0$ in exact
arithmetic, under the provided spatial-operator assumptions. Thus, their
numerical cancellation tests the assembled semi-discrete identities. 
The adapted entropy decreases
when $\Dcal_h^X>0$.

By contrast, $\mathcal M_{F,h}^{X}$ is a balance residual only for
the FEE--SAT, for which $\mathcal F_{E,h}^{\mathrm{FEE}}=\mathcal Q_{F,h}$.
Indeed,
\begin{equation*}
  \mathcal M_{F,h}^{X}
  =
  \mathcal R_{E,h}^{X}
  +\mathcal Q_{F,h}-\mathcal F_{E,h}^{X}.
\end{equation*}
Therefore, for the CR--SAT, a nonzero Fourier-only mismatch measures the
intended finite-resolution flux difference, not a failure of the
complete energy balance.

We also define the regularized scaled residual
\begin{equation}
  \widehat{\mathcal R}_{S_{T_w},h}^{X}(t)
  =
  \frac{
    |\mathcal R_{S_{T_w},h}^{X}(t)|
  }{
    1+
    |\mathrm d\mathcal S_{T_w,h}/\mathrm dt|
    +\Dcal_h^X(t)
  }.
  \label{eq:balance-scaled-adapted-entropy-residual}
\end{equation}
The denominator measures the magnitudes of the two balance terms, with
a unit floor to avoid division by zero. These instantaneous values do
not establish a fully discrete entropy identity for the time integrator.
Fully discrete stability with respect to the adapted entropy can, for
example, be obtained using a suitable relaxation Runge--Kutta method
with $\mathcal S_{T_w,h}$ as the entropy functional, under the conditions
of Ref.~\cite{RanochaEtAl2020}.

\subsection{One-dimensional manufactured-solution accuracy study}
\label{subsec:matched-manufactured-solution}

The three spatial configurations used in this study are summarized below:
\begin{equation}
  \begin{array}{l|cccc}
    \text{configuration}
    & \bm D_f
    & \alpha
    & \beta_{\mathrm{int}}
    & \beta_w
    \\
    \hline
    \text{complete entropy-conservative}
    & \bm0 & 0 & 0 & 0
    \\
    \begin{gathered}\text{Entropy-conservative wall}\\[-0.2ex]
    \text{with entropy-stable interior}\end{gathered}
    & \bm Y_f|\bm\Lambda_f|\bm Y_f^{\transpose}
    & 0 & 0.25 & 0
    \\
    \text{complete entropy-stable}
    & \bm Y_f|\bm\Lambda_f|\bm Y_f^{\transpose}
    & 0 & 0.25 & 0.25
  \end{array}
  \label{eq:numerical-configuration-definitions}
\end{equation}
The middle configuration isolates the effect of interior stabilization while
retaining a boundary-entropy-conservative physical wall. For every row of
Eq.~\eqref{eq:numerical-configuration-definitions}, the CR--SAT and the FEE--SAT use
identical volume and interior-interface operators. 
The configuration names distinguish the optional numerical dissipation
in the interface and wall treatments. In the complete entropy-conservative
configuration, these numerical dissipation terms are disabled. This does
not imply entropy conservation for the viscous flow: physical viscous
and thermal dissipation remain active in every configuration.

\subsubsection{Test problem and error measures}
\label{subsubsec:matched-mms-setup}

We apply the method of manufactured solutions (MMS) to a forced
one-dimensional reduction of the three-dimensional compressible
Navier--Stokes model on $\Omega=[0,1]$,
\begin{equation}
  \partial_t\bQ+\partial_x\bm F_x^{(I)}
  =\partial_x\bm F_x^{(V)}+\bm G_{\mathrm{MMS}},
  \label{eq:matched-mms-governing-equation}
\end{equation}
with
\begin{equation}
  \tau_{xx}=\left(\frac43\mu+\zeta\right)\partial_xu,
  \qquad
  q_{h,x}=-\kappa\partial_xT.
  \label{eq:matched-mms-one-dimensional-constitutive-laws}
\end{equation}
The imposed exact primitive variables are
\begin{equation}
  \begin{aligned}
  \rho_{\mathrm{ex}}(x,t)
  &=1+0.08\exp(-t)\sin^2(2\pi x),
  \\
  u_{\mathrm{ex}}(x,t)
  &=0.08\exp(-t)\sin^2(2\pi x),
  \\
  T_{\mathrm{ex}}(x,t)
  &=1+0.08\exp(-t)\sin(4\pi x).
  \end{aligned}
  \label{eq:matched-mms-exact-solution}
\end{equation}
The initial condition is the exact state at $t=0$. At both endpoints,
$u_{\mathrm{ex}}=0$, $\partial_xu_{\mathrm{ex}}=0$,
$T_{\mathrm{ex}}=T_w=1$, and
$\partial_xT_{\mathrm{ex}}=0.32\pi\exp(-t)\neq0$, satisfying the
no-slip and isothermal data with nonzero conductive wall heat transfer.
We set $\gamma=1.4$, $R=1$, $\mu=10^{-3}$, $\zeta=0$, and the Prandtl
number $\mathrm{Pr}=0.72$, giving $c_v=2.5$, $c_p=3.5$, and
$\kappa=\mu c_p/\mathrm{Pr}=4.861111\times10^{-3}$.

With the exact fields $\rho=\rho_{\mathrm{ex}}$,
$u=u_{\mathrm{ex}}$, and $T=T_{\mathrm{ex}}$, set
$\vartheta=T-1$, $\chi=0.08\exp(-t)\cos(4\pi x)$, and
$\eta=4\mu/3+\zeta$. Analytic substitution into
Eq.~\eqref{eq:matched-mms-governing-equation} gives
\begingroup
\setlength{\jot}{0pt}
\begin{equation}
  \begin{aligned}
    \bm G_{\mathrm{MMS}}
    &=\bigl(G_\rho,G_{\rho u},0,0,G_{\rho E}\bigr)^{\transpose},
      \qquad G_\rho=-u+2\pi\vartheta(\rho+u),
    \\
    G_{\rho u}
    &=-u(\rho+u)+2\pi\vartheta\bigl(u^2+2\rho u+RT\bigr)
      +4\pi R\rho\chi-8\pi^2\eta\chi,
    \\
    G_{\rho E}
    &=-uE-\rho\bigl(c_v\vartheta+u^2\bigr)
      +2\pi\vartheta
      \bigl[c_pT(\rho+u)+\tfrac12u^2(3\rho+u)\bigr]
    \\
    &\quad+4\pi c_p\rho u\chi
      -4\pi^2\eta\bigl(\vartheta^2+2u\chi\bigr)
      +16\pi^2\kappa\vartheta.
  \end{aligned}
  \label{eq:matched-mms-source}
\end{equation}
\endgroup
Here, $E=c_vT+u^2/2$ and $c_p=c_v+R$. Both transverse momentum
sources vanish.

For $\psol\in\{1,2,3,4,5\}$, we use uniform meshes with
$N_{\mathrm{cell}}=2^\ell$ cells of size $h=2^{-\ell}$ at refinement
levels $\ell=0,\ldots,6$. For each mesh and degree, the fixed
Courant--Friedrichs--Lewy (CFL)-type time step is
\begin{equation}
  \Delta t
  =
  \min\left\{
    \frac{C_{\mathrm{adv}}h}{(2\psol+1)a_{\max}},
    \frac{C_{\mathrm{vis}}h^2}{(2\psol+1)^2\nu_{\max}}
  \right\}.
  \label{eq:matched-mms-time-step}
\end{equation}
Here, $C_{\mathrm{adv}}=0.25$ and $C_{\mathrm{vis}}=0.05$ are the
advective and viscous step factors, $a_{\max}=1.4$ is the prescribed
wave-speed bound, and the effective diffusive bound
\begin{equation*}
  \nu_{\max}
  =
  \max\left\{
    \frac{4\mu/3+\zeta}{\rho_{\min}^{\mathrm{CFL}}},
    \frac{\kappa}{\rho_{\min}^{\mathrm{CFL}}c_v}
  \right\}
\end{equation*}
uses the density lower bound $\rho_{\min}^{\mathrm{CFL}}=0.9$.
This conservative choice follows the explicit nodal DG
advective and diffusive scalings of Ref.~\cite{HesthavenWarburton2008}.
The solution is advanced to $t_f=0.5$ with the classical four-stage,
fourth-order Runge--Kutta (RK4) method~\cite{Butcher1996}, shortening
the final step to reach $t_f$ exactly. 

For both wall formulations in the complete entropy-stable configuration, halving
$\Delta t$ at $\psol=3$ and $\ell=6$ ($N_{\mathrm{cell}}=64$) changed
the temperature solution at $t_f=0.5$ by at most $5.18\times10^{-8}$
times the nominal-step temperature spatial error against the exact
manufactured temperature. Both quantities use the normalized discrete
SBP $L_2$ norm. This supports a negligible temporal contribution to
the temperature error in these two calculations.

At $t_f=0.5$, we evaluate the normalized SBP volume errors
$\lVert e_\varphi\rVert_{1,h}$ and $\lVert e_\varphi\rVert_{2,h}$,
and the nodal error $\lVert e_\varphi\rVert_{\infty,h}$, for
$\varphi\in\{\rho,u,T\}$, as defined in
\ref{app:constant-Tw-error-norms}. For any error $E_\ell$ on
refinement level $\ell$, the rate assigned to the finer level is
\begin{equation}
  r_\ell
  =
  \log_2\left(\frac{E_{\ell-1}}{E_\ell}\right),
  \qquad
  \ell=1,\ldots,6.
  \label{eq:matched-mms-observed-rate}
\end{equation}
No rate is reported at $\ell=0$. The complete paired errors and rates
are provided in Supplementary Material~S1.

\subsubsection{Convergence comparison}
\label{subsubsec:matched-mms-results}

Figure~\ref{fig:balanced-temperature-convergence} shows the discrete
$L_2$ temperature errors under mesh refinement for the complete entropy-conservative and
complete entropy-stable configurations. The density and velocity plots appear in
Section~1 of Supplementary Material~S1 (Figures~S1 and~S2, respectively).
Each matched CR--SAT/FEE--SAT comparison uses identical meshes, exact
solution, physical parameters, time integrator, and interior operators.

\begin{figure}[t]
  \centering
  \includegraphics[width=0.92\linewidth]
  {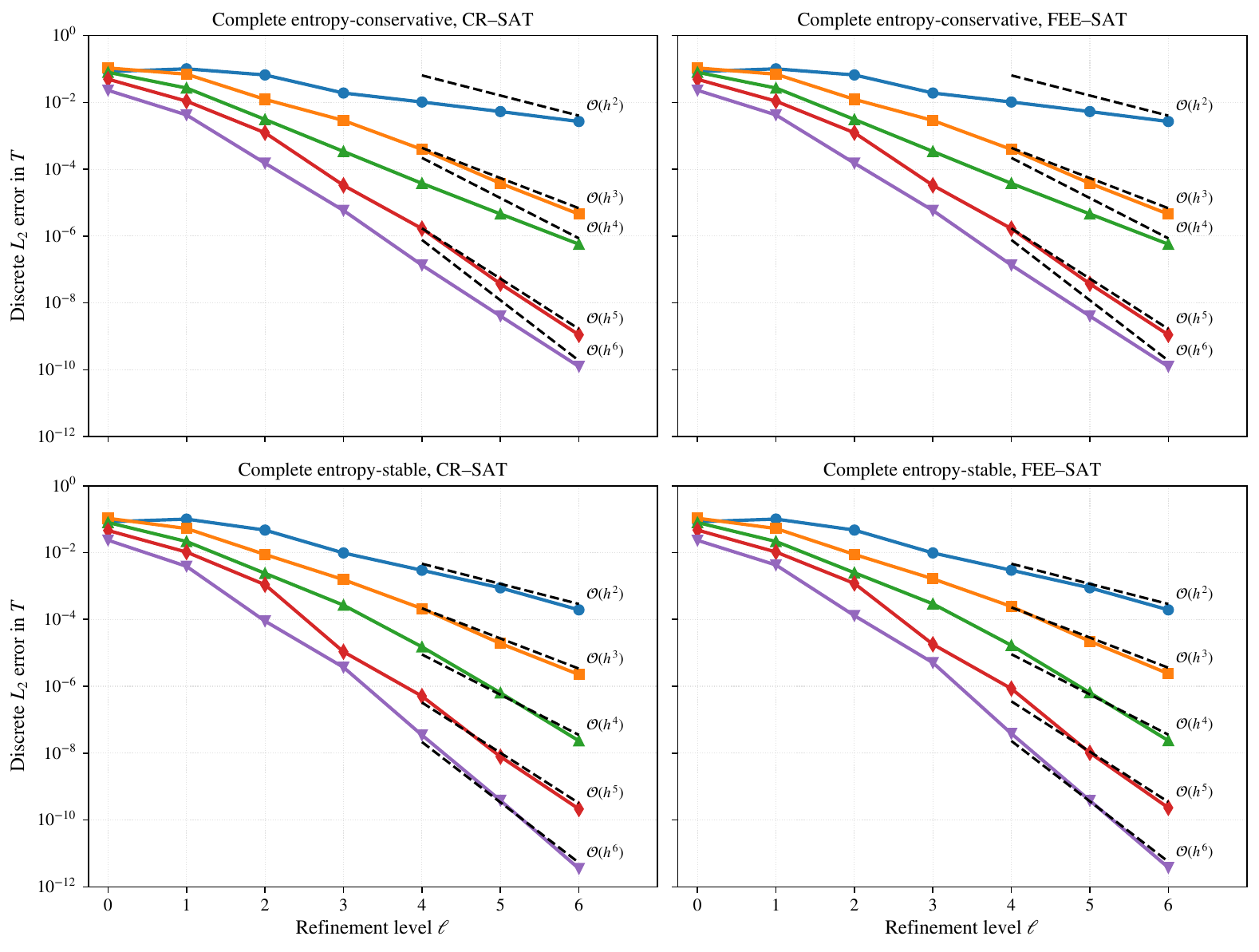}
  \vspace{0.2em}
  \includegraphics[width=0.70\linewidth]
  {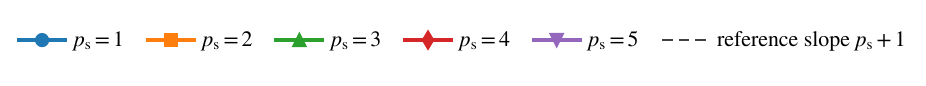}
  \caption{Discrete SBP $L_2$ temperature errors for the one-dimensional
  constant-$T_w$ manufactured solution. The panels compare the complete
  entropy-conservative and complete entropy-stable configurations using the CR--SAT and the FEE--SAT.
  The curves correspond to $\psol=1,\ldots,5$. Black dashed segments
  over the final three refinement levels show the
  $\mathcal O(h^{\psol+1})$ reference slopes.}
  \label{fig:balanced-temperature-convergence}
\end{figure}

Table~\ref{tab:compact-finest-rate-summary} complements these plots by
comparing the finest-pair $L_2$ rates for all three variables with the
nominal rate $\psol+1$. It also includes the entropy-conservative wall with entropy-stable interior
configuration, which is not shown in the plots. The finest pair
comprises $N_{\mathrm{cell}}=32$ and $64$ and each variable entry is ordered
$(\text{CR--SAT},\text{FEE--SAT})$.

\begin{table}[t]
\centering
\small
\caption{Finest-pair discrete SBP $L_2$ convergence rates for density,
velocity, and temperature under the three matched spatial configurations for the one-dimensional
constant-$T_w$ manufactured solution.
Each entry in the last three columns is the ordered pair
$(\text{CR--SAT},\text{FEE--SAT})$. The nominal rate is
$\psol+1$, and the finest pair corresponds to
$N_{\mathrm{cell}}=32$ and $64$, or $\ell=5$ and $6$.}
\label{tab:compact-finest-rate-summary}
\setlength{\tabcolsep}{4.3pt}
\renewcommand{\arraystretch}{1.05}
\begin{tabular}{ccccc}
\toprule
$\psol$ & Nominal rate
& \multicolumn{3}{c}{Finest-pair rate $(\mathrm{CR},\mathrm{FEE})$} \\
\cmidrule(lr){3-5}
& $\psol+1$ & $\rho$ & $u$ & $T$ \\
\midrule
\multicolumn{5}{l}{\emph{Complete entropy-conservative configuration}} \\
1 & 2 & $(0.988,0.988)$ & $(1.017,1.017)$ & $(0.989,0.989)$ \\
2 & 3 & $(3.028,3.028)$ & $(3.028,3.028)$ & $(3.062,3.062)$ \\
3 & 4 & $(3.040,3.040)$ & $(3.006,3.006)$ & $(3.006,3.006)$ \\
4 & 5 & $(5.075,5.075)$ & $(4.967,4.967)$ & $(5.072,5.072)$ \\
5 & 6 & $(5.017,5.017)$ & $(5.025,5.025)$ & $(5.016,5.016)$ \\
\addlinespace[0.35em]
\multicolumn{5}{l}{\emph{Entropy-conservative wall with entropy-stable interior}} \\
1 & 2 & $(1.958,1.958)$ & $(1.965,1.965)$ & $(2.192,2.192)$ \\
2 & 3 & $(3.167,3.167)$ & $(3.281,3.281)$ & $(3.218,3.218)$ \\
3 & 4 & $(4.133,4.133)$ & $(4.111,4.111)$ & $(4.757,4.757)$ \\
4 & 5 & $(5.097,5.097)$ & $(5.384,5.384)$ & $(5.462,5.462)$ \\
5 & 6 & $(6.054,6.054)$ & $(6.213,6.213)$ & $(6.665,6.665)$ \\
\addlinespace[0.35em]
\multicolumn{5}{l}{\emph{Complete entropy-stable configuration}} \\
1 & 2 & $(1.948,1.958)$ & $(1.964,1.965)$ & $(2.186,2.192)$ \\
2 & 3 & $(2.949,3.159)$ & $(3.272,3.276)$ & $(3.087,3.220)$ \\
3 & 4 & $(4.124,4.129)$ & $(4.113,4.113)$ & $(4.767,4.756)$ \\
4 & 5 & $(4.836,5.073)$ & $(5.363,5.364)$ & $(5.203,5.463)$ \\
5 & 6 & $(6.093,6.061)$ & $(6.213,6.212)$ & $(6.777,6.665)$ \\
\bottomrule
\end{tabular}
\end{table}

In the complete entropy-conservative configuration, the formulations give similar
$L_2$ error curves and identical finest-pair $L_2$ rates to the reported
precision in Table~\ref{tab:compact-finest-rate-summary}.
For the tested degrees, these rates are approximately $\psol$ for odd
$\psol$ and $\psol+1$ for even $\psol$. The same pattern persists when
one or two wall-adjacent cells are excluded from the error calculation.
Extending the $\psol=1,3$ CR--SAT runs to $N_{\mathrm{cell}}=128$ gives
velocity rates of $1.004$ and $3.003$, respectively.

The block for the entropy-conservative wall with entropy-stable interior
in Table~\ref{tab:compact-finest-rate-summary} shows recovery of the
$\psol+1$ trend for both formulations while retaining $\beta_w=0$.
Their finest-level $L_2$ errors, reported in Supplementary Material~S1,
differ by less than $5\times10^{-7}$ relative to the corresponding
CR--SAT error. For this manufactured solution, the comparison associates
the parity-sensitive reduction with the symmetric-LDG configuration
without numerical entropy dissipation, rather than with either wall
formulation.

For the complete entropy-stable configuration, the plots and tabulated rates again
show the $\psol+1$ trend, although the finite-grid error constants differ.
The finest-level $L_2$ errors in Supplementary Material~S1 give
FEE--SAT/CR--SAT ratios differing from unity by at most about $8.6\%$
for density and temperature and $0.34\%$ for velocity. Neither
formulation is consistently more accurate across variables and degrees
in this test.

\subsection{One-dimensional thermal-relaxation energy and adapted-entropy balances}
\label{subsec:thermal-relaxation-comparison}

This calculation tests the two finite-resolution total-energy structures
and the common adapted-entropy balance along a nontrivial solution trajectory.
We solve the source-free one-dimensional compressible Navier--Stokes
equations on $[0,1]$, with stationary no-slip isothermal walls at both
endpoints and $T_w=1$. There is no volumetric forcing, but heat exchange
with the walls is permitted. The pressure-equilibrated initial state is
\begin{equation*}
  T(x,0)=1+0.2x(1-x),
  \qquad
  u(x,0)=0,
  \qquad
  \rho(x,0)=\frac{1}{R T(x,0)},
  \qquad
  p(x,0)=1.
\end{equation*}
The nonuniform temperature satisfies the wall value and gives a positive
initial outward Fourier heat flux at both endpoints. The calculation uses
$N_{\mathrm{cell}}=8$ and $\psol=3$. The common interior coupling includes
Merriam dissipation, the symmetric LDG trace $\alpha=0$, and the IP
strength $\beta_{\mathrm{int}}=0.25$. Both the CR--SAT and the FEE--SAT are tested
with $\beta_w=0$ and $\beta_w=0.25$. Note that a boundary-entropy-conservative
wall setting removes the wall dissipation, not the interior numerical
dissipation.

The physical parameters are
\begin{equation*}
  \gamma=1.4,
  \qquad
  R=1,
  \qquad
  \mu=10^{-2},
  \qquad
  \mathrm{Pr}=0.72,
  \qquad
  \zeta=0.
\end{equation*}
The thermal conductivity is determined by
\begin{equation*}
  c_p=\frac{\gamma R}{\gamma-1},
  \qquad
  \kappa=\frac{\mu c_p}{\mathrm{Pr}}.
\end{equation*}

The RK4 time integration scheme advances the solution to $t_f=0.05$.
A safety factor of $0.65$ is applied to the implementation's
advective--diffusive time-step estimate. 
The following time-step statistics are representative of all four
thermal-relaxation calculations: $1{,}232$ accepted steps, with
$1.14535\times10^{-5}\leq\Delta t\leq4.06081\times10^{-5}$.
We evaluate the instantaneous balances using the common diagnostics in
Subsection~\ref{subsec:balance-diagnostics}. In the one-dimensional
reduction (see Remark \ref{re:dimension}), only the axial derivative contributes to
Eq.~\eqref{eq:balance-viscous-volume-dissipation}, while the
three-dimensional normal stress in
Eq.~\eqref{eq:matched-mms-one-dimensional-constitutive-laws} is retained.
The thermal-relaxation results below report unscaled residuals. We
report a separate stage-integrated energy check in
\ref{app:thermal-relaxation-details}.

Figure~\ref{fig:thermal-relaxation-compact} shows the two calculations
with active wall penalties. Over the recorded time window of all
four calculations, the complete energy residual is at most
$6.94\times10^{-18}$ and the adapted-entropy residual is at most
$7.98\times10^{-18}$ in absolute value. For the boundary-entropy-stable
CR--SAT, the Fourier-only mismatch reaches $1.78\times10^{-4}$ because
its complete wall flux includes finite-resolution mechanical-work and
complete-residual penalty contributions. For the FEE--SAT, the complete
energy residual and Fourier-only mismatch coincide and remain at
roundoff scale. We report the four-case summary and full time histories
in~\ref{app:thermal-relaxation-details}.

\begin{figure}[t]
  \centering
  \includegraphics[width=0.92\linewidth]
    {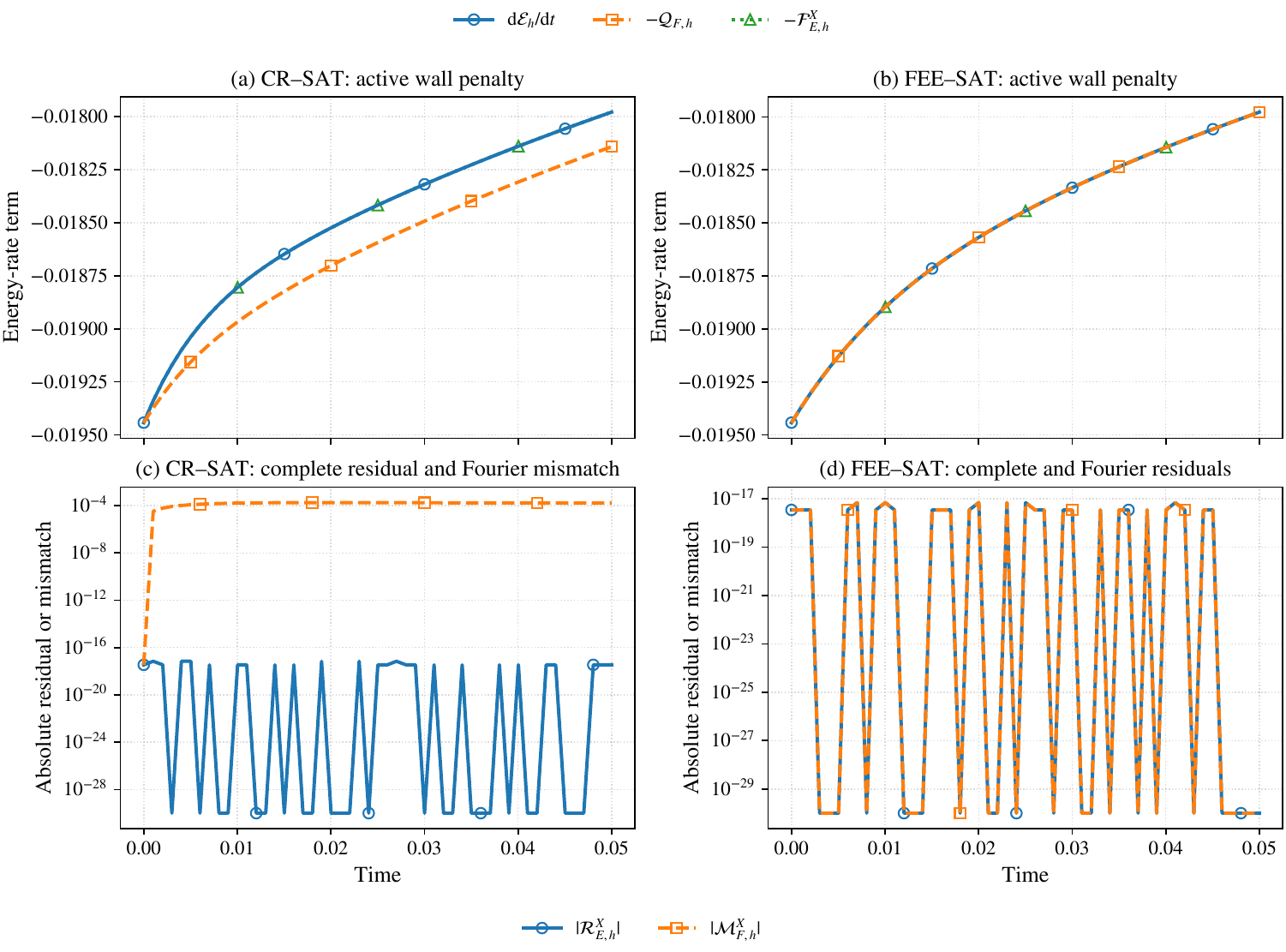}
  \caption{One-dimensional thermal-relaxation instantaneous energy test for
  the CR--SAT (left) and the FEE--SAT (right), with $\beta_w=0.25$ in both
  calculations. Panels (a) and (b) compare
  $\mathrm d\mathcal E_h/\mathrm dt$, $-\mathcal Q_{F,h}$, and
  $-\mathcal F_{E,h}^{X}$. The rate is evaluated from the conservative
  right-hand side using Eq.~\eqref{eq:balance-energy-rate}. The outward
  wall-flux sums are evaluated separately using
  Eq.~\eqref{eq:thermal-relaxation-energy-quantities}. Their minus signs
  permit direct comparison with the energy rate. Panels (c) and (d)
  show the unscaled absolute complete energy residual
  $|\mathcal R_{E,h}^{X}|$ and Fourier-only mismatch
  $|\mathcal M_{F,h}^{X}|$ on logarithmic axes. Both formulations close
  the complete balance at roundoff scale. The Fourier-only mismatch
  is required to vanish only for the FEE--SAT.}
  \label{fig:thermal-relaxation-compact}
\end{figure}

\subsection{Three-dimensional test cases}
\label{subsec:multidimensional-evidence}

The preceding tests isolate the local wall algebra, one-dimensional spatial
accuracy, and assembled one-dimensional energy and adapted-entropy balances.
We now increase the geometric and flow complexity. The three calculations
below test, respectively, closure of the assembled three-dimensional
adapted-entropy balance, convergence of the numerical wall traces to the
prescribed no-slip and isothermal data on smooth curved walls, and
application of the wall treatment to a supersonic separated flow.
Reference scales are specified in
\ref{app:multidimensional-reference-scales}.
Mesh-generation procedures are specified with the
individual test configurations. ParaView is used for flow visualization
\cite{Ayachit2015:ParaView}.

Theorem~\ref{thm:global-source-free-two-sat-balances} is stated for
conforming Cartesian affine elements. However, the CR--SAT and the FEE--SAT
wall identities are pointwise at each boundary quadrature node, so their
local algebraic form is unchanged on curvilinear and $hp$-nonconforming
meshes. The assembled global balance extends provided that the metric terms
satisfy the entropy-compatible discrete metric identities and geometric
conservation law \cite{CreanEtAl2018,ReynaNolascoEtAl2020Metrics}, 
and that the nonconforming interface operators satisfy
the required SBP-compatibility, conservation, and entropy-stability
conditions
\cite{DelReyFernandezEtAl2020CurvilinearTM,
DelReyFernandezEtAl2020PNonconformingNS,
DelReyFernandezEtAl2020HP}. Thus, no modification of the pointwise wall
proof is required. The cited geometric and interface constructions provide
the curvilinear and variable-degree extension used for the three-dimensional 
test cases reported next.

\subsubsection{Closed-domain adapted-entropy balance}
\label{subsubsec:three-dimensional-adapted-entropy}

This test assesses the assembled adapted-entropy balance for both
the CR--SAT and the FEE--SAT on a three-dimensional curvilinear mesh.
The fluid occupies the region between an outer cube and a spherical
cavity, both centered at the origin. The dimensional sphere diameter
is the reference length $L_0$, and the full dimensional cube edge is
$2L_0$. After normalizing lengths by $L_0$, the sphere diameter is $1$
and the cube edge is $L_{\mathrm{box}}=2$. The nondimensional fluid
domain is
\begin{equation*}
  \Omega_{\mathrm{cd}}
  =
  (-1,1)^3
  \setminus
  \overline{B_{1/2}(\bm0)},
\end{equation*}
where $B_{1/2}(\bm0)$ is the open ball of nondimensional radius $1/2$
centered at the origin. The cube faces lie at $x=\pm1$, $y=\pm1$, and
$z=\pm1$, and the spherical cavity is strictly contained within the cube.

Both wall boundary components are stationary, impermeable, no-slip
isothermal walls with $T_w=1$. Herein, ``closed domain'' means
impermeability (i.e., no penetration), not thermal insulation.

We nondimensionalize length, density, velocity, and temperature by
$L_0$, $\rho_0$, $U_0$, and $T_0$, respectively, as specified in
\ref{app:multidimensional-reference-scales}. Here, $\rho_0$ is the
uniform initial dimensional density, and $T_0=T_w^{\mathrm{dim}}$ is
the dimensional wall temperature.
The reference Mach, Reynolds, and Prandtl numbers are
\begin{equation*}
  \mathrm{Ma}=\frac{U_0}{a_0}=0.5,
  \qquad
  \mathrm{Re}=\frac{\rho_0U_0L_0}{\mu_0}=100,
  \qquad
  \mathrm{Pr}=\frac{\mu_0c_p}{\kappa_0}=0.72.
\end{equation*}
Here, $a_0=\sqrt{\gamma R T_0}$ is the reference sound speed,
$\mu_0$ and $\kappa_0$ are the dimensional reference viscosity and
thermal conductivity, and $R$ and $c_p$ are the dimensional specific
gas constant and specific heat at constant pressure. Thus,
$U_0=0.5a_0>0$ is the velocity normalization scale for this nonuniform
initial-value problem, not the speed of a uniform incoming flow.
These fixed reference parameters are not evaluated from the evolving
local velocity and temperature fields. In addition, no volumetric source
term is applied.

For $\bm x=(x,y,z)^{\transpose}\in\Omega_{\mathrm{cd}}$, the initial
primitive variables are
\begin{equation*}
  \begin{aligned}
    \rho(\bm x,0)
    &=1,
    \\
    u_1(\bm x,0)
    &=\frac{1}{16}+\sin(x)\cos(y)\cos(z),
    \\
    u_2(\bm x,0)
    &=\frac{1}{16}-\cos(x)\sin(y)\cos(z),
    \\
    u_3(\bm x,0)
    &=\frac{1}{16},
    \\
    T(\bm x,0)
    &=1+
    \frac{1}{16}
    \bigl[\cos(2x)+\cos(2y)\bigr]
    \bigl[\cos(z)+2\bigr].
  \end{aligned}
\end{equation*}
Here, $u_1$, $u_2$, and $u_3$ are the Cartesian velocity components.
The analytic initial velocity is divergence free, and the temperature
satisfies $T(\bm x,0)\geq5/8>0$.

The nonuniform initial velocity and temperature intentionally violate
the wall data and initiate a transient. The butterfly-type mesh is generated with Gmsh~\cite{Geuzaine2009:Gmsh}
and contains
$384$ curvilinear hexahedra, obtained by subdividing each geometric line
uniformly into four segments. The geometry representation is cubic,
and the solution polynomial degree is $\psol=3$.
Figure~\ref{fig:closed-domain-initial} shows the same mesh and LGL
solution nodes colored by the two initial fields.

\begin{figure}[!htb]
  \centering
  \begin{subfigure}[t]{0.35\linewidth}
    \centering
    \includegraphics[width=\linewidth]
      {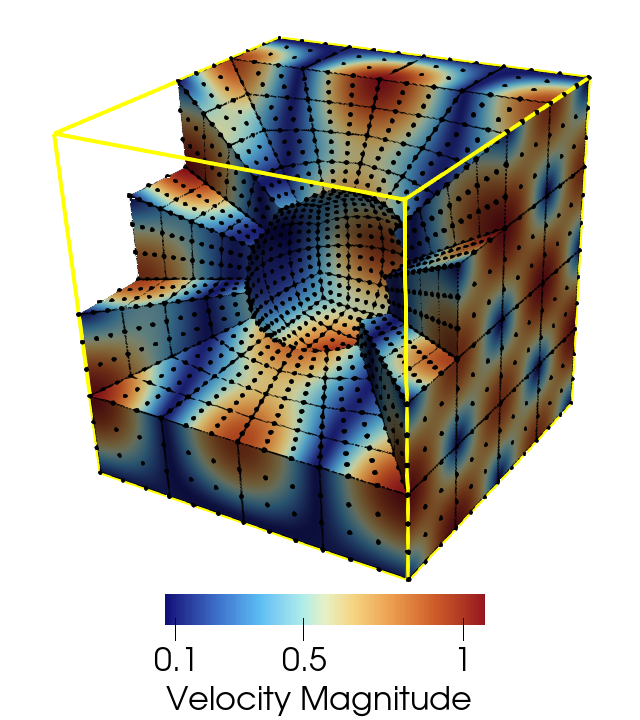}
    \caption{Initial velocity magnitude, $|\bu|$.}
  \end{subfigure}
  \hspace{0.04\linewidth}
  \begin{subfigure}[t]{0.35\linewidth}
    \centering
    \includegraphics[width=\linewidth]
      {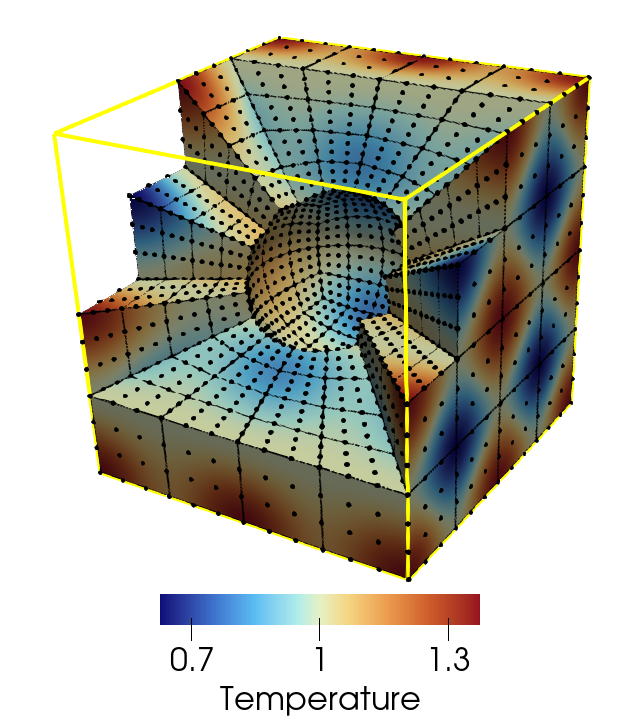}
    \caption{Initial temperature, $T$.}
  \end{subfigure}
  \caption{Three-dimensional closed-domain geometry, mesh, and initial fields at $t=0$.
  Black curves mark element edges, black points mark LGL solution nodes,
  and yellow curves outline the outer cube. The cut surfaces are
  visualization planes, not additional physical boundaries.}
  \label{fig:closed-domain-initial}
\end{figure}

We use both wall formulations in their boundary-entropy-conservative
settings. Unlike the thermal-relaxation test, both calculations disable
all optional numerical entropy dissipation:
\begin{equation*}
  \bm D_f=\bm0,
  \qquad
  \beta_{\mathrm{int}}=0,
  \qquad
  \beta_w=0.
\end{equation*}
The LDG coupling remains active, including the complementary
viscous-flux trace and the auxiliary-gradient wall SATs. The LDG
trace-bias parameter used in each run is $\alpha = 0$. Physical viscosity and Fourier heat conduction
remain active. Thus, in each calculation the total dissipation reduces
to $\Dcal_h^X=\Dcal_h^{(V)}$, where
$X\in\{\mathrm{CR},\mathrm{FEE}\}$ identifies the wall formulation.

Time integration uses the explicit Bogacki--Shampine embedded
Runge--Kutta pair of orders three and two (BS3)
\cite{BogackiShampine1989}. It is implemented as \texttt{TSRK3BS} in the
Portable, Extensible Toolkit for Scientific Computation (PETSc) \cite{BalayEtAl2026PETSc}.

We use the common diagnostic procedure in
Subsection~\ref{subsec:balance-diagnostics}, with $d=3$.
The adapted-entropy rate and viscous-volume dissipation are evaluated
separately using Eqs.~\eqref{eq:balance-adapted-entropy-rate}
and~\eqref{eq:balance-viscous-volume-dissipation}, respectively.
The superscript $\mathrm{cd}$ labels the closed-domain test and not a
different entropy. Thus, $\mathcal S_{T_w,h}^{\mathrm{cd}}$ is the functional
in Eq.~\eqref{eq:balance-discrete-adapted-entropy} evaluated for this
test. The dependence of the rate, dissipation, and residuals on $X$
is left implicit and we identify the wall formulation in each panel
and table row. Each balance uses quantities from the same
calculation and their values do not need to coincide between formulations.

For either formulation, the signed residual is
\begin{equation}
  \mathcal R_{S_{T_w},h}^{\mathrm{cd}}(t)
  =
  \frac{\mathrm d\mathcal S_{T_w,h}^{\mathrm{cd}}}{\mathrm dt}
  +\Dcal_h^{(V)}(t).
  \label{eq:three-dimensional-adapted-entropy-defect}
\end{equation}
The curvilinear extension of
Eq.~\eqref{eq:two-sat-adapted-entropy-balance} is subject to the
requirements stated in
Subsection~\ref{subsec:multidimensional-evidence}: entropy-compatible
metric terms satisfying the discrete metric identities and geometric
conservation law, together with the stated interface-compatibility
requirements wherever nonconforming interfaces are used. Under these
requirements and the source-free wall and dissipation settings
specified here, the balance gives
$\mathcal R_{S_{T_w},h}^{\mathrm{cd}}=0$ in exact arithmetic.
The regularized scaled residual
$\widehat{\mathcal R}_{S_{T_w},h}^{\mathrm{cd}}$ is obtained from
Eq.~\eqref{eq:balance-scaled-adapted-entropy-residual} using the same
case label and $\Dcal_h^{(V)}$.

Figure~\ref{fig:three-dimensional-adapted-entropy} shows the balance
contributions and signed residuals for the CR--SAT and the FEE--SAT over
$0\leq t\leq1$.

\begin{figure}[!htb]
  \centering
  \begin{subfigure}[t]{0.455\linewidth}
    \centering
    \includegraphics[width=\linewidth]
      {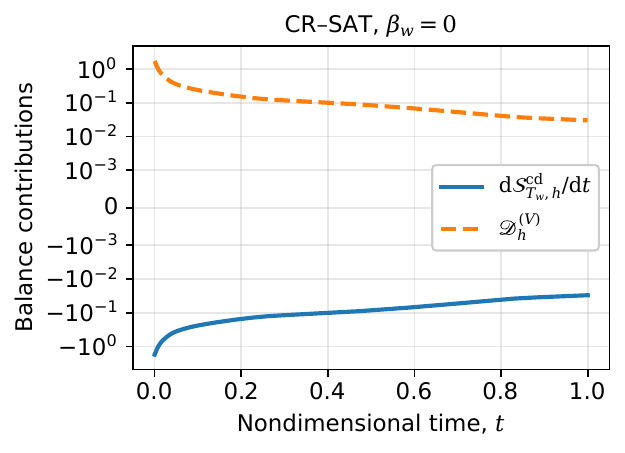}
    \caption{CR--SAT: balance contributions.}
    \label{fig:closed-domain-cr-terms}
  \end{subfigure}
  \hfill
  \begin{subfigure}[t]{0.455\linewidth}
    \centering
    \includegraphics[width=\linewidth]
      {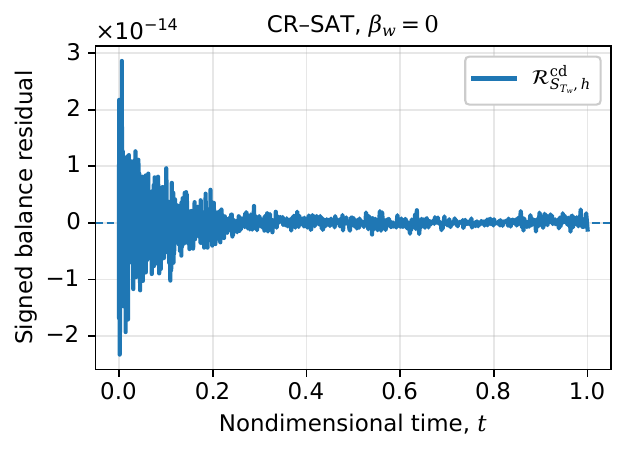}
    \caption{CR--SAT: signed residual.}
    \label{fig:closed-domain-cr-residual}
  \end{subfigure}

  \medskip

  \begin{subfigure}[t]{0.455\linewidth}
    \centering
    \includegraphics[width=\linewidth]
      {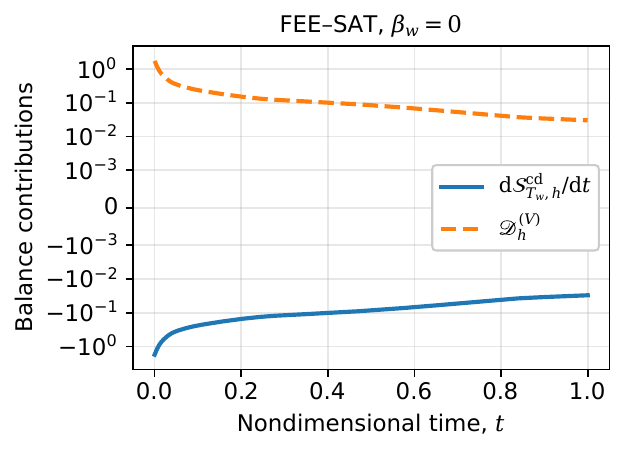}
    \caption{FEE--SAT: balance contributions.}
    \label{fig:closed-domain-fee-terms}
  \end{subfigure}
  \hfill
  \begin{subfigure}[t]{0.455\linewidth}
    \centering
    \includegraphics[width=\linewidth]
      {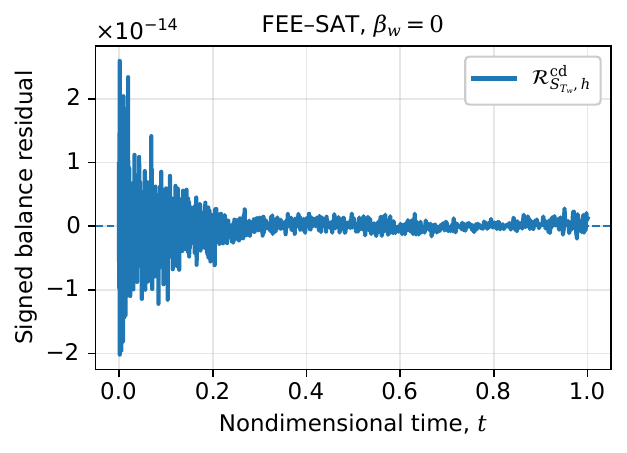}
    \caption{FEE--SAT: signed residual.}
    \label{fig:closed-domain-fee-residual}
  \end{subfigure}
  \caption{Three-dimensional closed-domain adapted-entropy balance for
  the CR--SAT (upper row) and the FEE--SAT (lower row), with $\beta_w=0$,
  active LDG coupling, and all optional numerical entropy dissipation
  disabled. Panels (a) and (c) show the adapted-entropy rate
  $\mathrm d\mathcal S_{T_w,h}^{\mathrm{cd}}/\mathrm dt$,
  which is negative in these calculations, and the positive
  viscous-volume dissipation $\Dcal_h^{(V)}$. 
  Their vertical axes are
  symmetric logarithmic, with a linear region for ordinate magnitudes
  below $10^{-3}$. Panels (b) and (d) show their signed sum,
  $\mathcal R_{S_{T_w},h}^{\mathrm{cd}}$, on linear axes, with dashed
  zero reference lines. Each panel uses quantities from its stated wall
  formulation.}
  \label{fig:three-dimensional-adapted-entropy}
\end{figure}

At every recorded state, the adapted-entropy
rate is negative and the viscous-volume dissipation is positive. In
both calculations, the dissipation decreases along the recorded
history from approximately $1.7463$ at $t=0$ to
$3.0276\times10^{-2}$ for the CR--SAT and $3.0329\times10^{-2}$ for
the FEE--SAT at $t=1$. Therefore, the two contributions remain nonzero at the final time,
while their signed sum fluctuates around zero at a level consistent
with accumulated roundoff in double-precision arithmetic.

Table~\ref{tab:closed-domain-adapted-entropy-balance} reports the
maximum absolute and scaled residuals. The maxima are evaluated over
all recorded states, separately for each calculation.
These results provide numerical evidence for closure of the assembled
semi-discrete adapted-entropy balance for both wall formulations, without
added numerical entropy dissipation.

\begin{table}[!htb]
  \centering
  \caption{Three-dimensional closed-domain adapted-entropy balance test.
  The maxima are taken over the recorded sampling times of each
  calculation in $0\leq t\leq1$, including both endpoints. The scaled
  residual is evaluated at each sample before taking its maximum.}
  \label{tab:closed-domain-adapted-entropy-balance}
  \begin{tabular}{@{}lccc@{}}
    \toprule
    Wall formulation & Recorded states &
    $\max_t|\mathcal R_{S_{T_w},h}^{\mathrm{cd}}|$ &
    $\max_t\widehat{\mathcal R}_{S_{T_w},h}^{\mathrm{cd}}$ \\
    \midrule
    CR--SAT & $1{,}378$ & $2.86\times10^{-14}$ & $8.85\times10^{-15}$ \\
    FEE--SAT & $1{,}379$ & $2.60\times10^{-14}$ & $1.03\times10^{-14}$ \\
    \bottomrule
  \end{tabular}
\end{table}

\FloatBarrier

\subsubsection{Curved-wall trace convergence in an annular pipe}
\label{subsubsec:three-dimensional-wall-trace-convergence}

The annular-pipe calculation tests whether the fluid-side numerical wall
traces produced by the CR--SAT and the FEE--SAT approach the prescribed data under
curved-grid refinement. Both wall formulations are used in their
entropy-stable settings for every solution degree and refinement level.
These are our default wall settings for robustness in three-dimensional
simulations. 

The annular pipe is aligned with the $z$ axis. The axial boundaries are
periodic, and both cylindrical walls are
stationary, no-slip, and isothermal with $T_w=1$. Axially homogeneous
momentum sources generate azimuthal and axial motion, while a total-energy
source raises the interior temperature above the wall value. Therefore, the resulting
steady state is nontrivial even though the exact wall data are
constant.

We use the nondimensionalization and reference-parameter definitions in
\ref{app:multidimensional-reference-scales}. For this test, $L_0$ is the
dimensional inner-cylinder diameter, $\rho_0$ is the uniform initial
dimensional density, and $T_0=T_w^{\mathrm{dim}}$ equals the uniform
initial dimensional temperature. The nondimensional inner and outer
radii are $1/2$ and $2$, respectively, and the axial length is $4$.
The reference Mach, Reynolds, and Prandtl numbers are
\begin{equation*}
  \mathrm{Ma}=0.1,
  \qquad
  \mathrm{Re}=1,
  \qquad
  \mathrm{Pr}=0.7.
\end{equation*}
The velocity scale $U_0=0.1a_0>0$ is distinct from the initial fluid
velocity, which is zero. 

The uniform initial state is
\begin{equation*}
  \rho(\bm x,0)=1,
  \qquad
  \bu(\bm x,0)=\bm0,
  \qquad
  T(\bm x,0)=1.
\end{equation*}
The prescribed conservative source is
\begin{equation*}
  \bm G_{\mathrm{ann}}
  =
  \begin{pmatrix}
    0\\
    -3y/(\mathrm{Re}\,r)\\
    3x/(\mathrm{Re}\,r)\\
    3/\mathrm{Re}\\
    1/\mathrm{Pr}
  \end{pmatrix},
  \qquad
  r=\sqrt{x^2+y^2}.
\end{equation*}
The transverse momentum source is azimuthal, the third momentum component
acts in the axial direction, and the last component is the total-energy
source.

The structured base mesh is generated with an in-house nonuniform rational
B-spline (NURBS) geometry tool. It contains three axial, sixteen
circumferential, and four radial cells, for a total of $192$ hexahedra.
The radial distribution is symmetric and nonuniform, with a wall-adjacent
cell size of $1/8$ on the base mesh.

The radial and circumferential partitions are refined uniformly from one
level to the next, while the three-cell periodic axial partition is fixed.
For $\ell=0,\ldots,4$, a representative transverse mesh scale satisfies
$h_\ell=2^{-\ell}h_0$, where $h$ refers to the refined cross-section,
not to a uniformly refined three-dimensional mesh.

The cell counts are
\begin{equation*}
  N_{\mathrm{cell},\ell}=192\,4^\ell,
  \qquad
  \ell=0,\ldots,4.
\end{equation*}
Steady states are obtained by implicit pseudo-transient continuation.
Both wall formulations use the entropy-scaled characteristic
inviscid-interface dissipation of
Subsection~\ref{subsec:interface-2015-coupling}, with the eigensystem
evaluated at the Roe-averaged state. The LDG trace-bias parameter is
$\alpha=0$. The values of $\beta_{\mathrm{int}}$ and $\beta_w$ for each
wall formulation are both $0.25$. These settings are fixed
across $\psol=1,\ldots,4$ and $\ell=0,\ldots,4$ within each wall
formulation. Steady-state convergence is declared when the discrete $L_2$ norm of the
residual of each of the five conservation equations is below $10^{-11}$.

Figure~\ref{fig:annular-pipe-mesh-and-profiles} shows the base
mesh and analytical radial profiles evaluated on the finest grid. Their reported peak
values in the nondimensional variables of
\ref{app:multidimensional-reference-scales} are
\begin{equation*}
  \max_{1/2\leq r\leq2}u_\theta=0.747,
  \qquad
  \max_{1/2\leq r\leq2}u_z=0.886,
  \qquad
  \max_{1/2\leq r\leq2}T=1.293.
\end{equation*}

\begin{figure}[!htb]
  \centering
  \begin{subfigure}[t]{0.38\linewidth}
    \centering
    \includegraphics[width=0.75\linewidth]
      {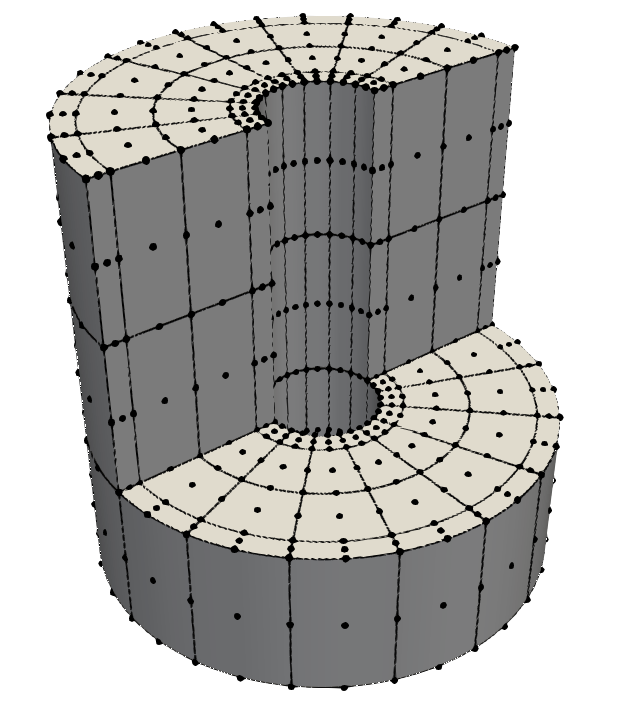}
    \caption{Base mesh and solution nodes, $\psol=2$.}
  \end{subfigure}
  \hspace{0.04\linewidth}
  \begin{subfigure}[t]{0.38\linewidth}
    \centering
    \includegraphics[width=0.75\linewidth]
      {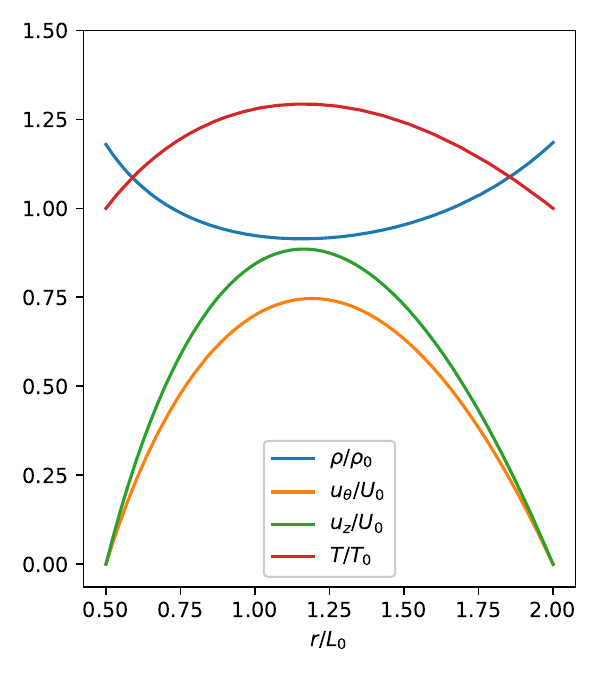}
    \caption{Radial profiles.}
  \end{subfigure}
  \caption{Three-dimensional annular-pipe mesh and representative steady radial profiles.
  Panel (a) shows the base mesh and degree-two solution nodes. Panel (b)
  shows density, azimuthal velocity $u_\theta$, axial velocity $u_z$,
  and temperature along $y=0$, $z=2$, and $1/2\leq r\leq2$, where
  $r=\sqrt{x^2+y^2}$ and the sampling coordinates are nondimensional.
  In the axis and legend labels of panel (b), the numerators denote
  dimensional quantities. Their ratios to $L_0$, $\rho_0$, $U_0$, and
  $T_0$ are the nondimensional variables defined in
  \ref{app:multidimensional-reference-scales}. These profiles show the steady 
analytical solution evaluated on the
finest grid.}
  \label{fig:annular-pipe-mesh-and-profiles}
\end{figure}

Let $\wall^{\mathrm{ann}}$ denote the union of the two cylindrical walls.
Using the physical SBP face inner product defined in
Subsection~\ref{subsec:tensor-product-sbp}, the discrete wall errors are
\begin{equation}
  E_u^w=
  \left[
    \sum_{f\subset\wall^{\mathrm{ann}}}
    \langle\lVert\bu_h-\bu_w\rVert_2^2,1\rangle_{f,h}
  \right]^{1/2},
  \qquad
  E_T^w=
  \left[
    \sum_{f\subset\wall^{\mathrm{ann}}}
    \langle(T_h-T_w)^2,1\rangle_{f,h}
  \right]^{1/2}.
  \label{eq:annular-pipe-wall-trace-errors}
\end{equation}
Here, $\bu_h$ and $T_h$ are the fluid-side numerical traces at the wall
quadrature nodes, $\bu_w=\bm0$ and $T_w =1$. The physical face-quadrature matrix
$\mathsf P_f$ includes the LGL face weights and the surface Jacobian
factors of the curved faces. These are unnormalized discrete wall $L_2$
errors: no division by wall area or node count is applied, and $E_u^w$
combines all three Cartesian velocity components. The superscript $w$
identifies wall errors. Each error is evaluated separately for each wall
formulation.
Section~2 of Supplementary Material~S1 gives the equivalent SBP matrix
expressions and the complete level-by-level errors and rates in
Tables~S10 and~S11 for velocity and temperature, respectively.

Figure~\ref{fig:annular-pipe-wall-convergence} reports both errors for
$\psol=1,\ldots,4$ and both wall formulations.
\begin{figure}[!htb]
  \centering
  \includegraphics[width=0.95\linewidth]
  {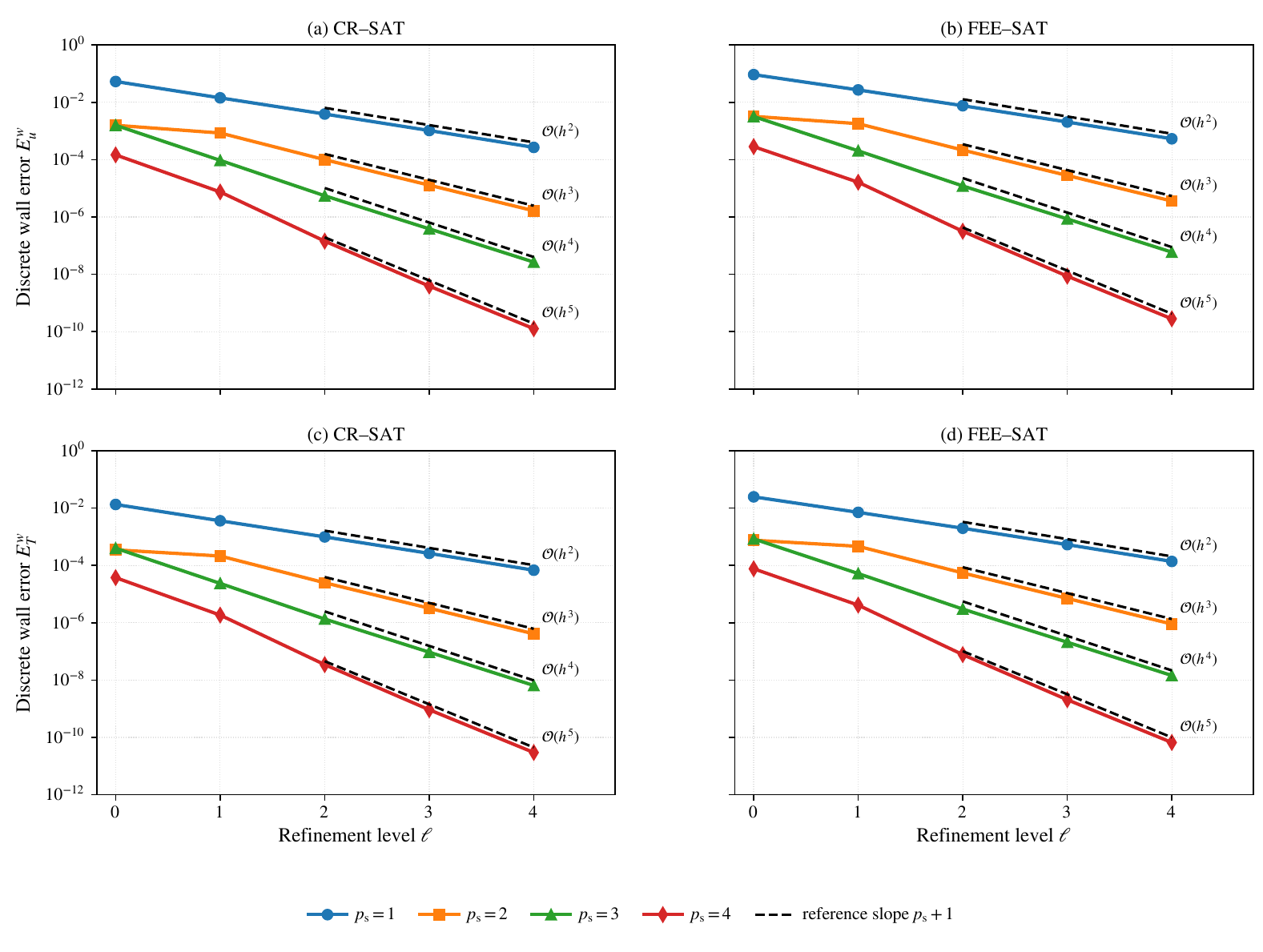}
  \caption{Three-dimensional annular-pipe wall-trace convergence for the entropy-stable
  settings of the CR--SAT (left) and the FEE--SAT (right), with
  $\psol=1,\ldots,4$. Panels (a) and (b) show the no-slip wall error
  $E_u^w$. Panels (c) and (d) show the isothermal-wall error $E_T^w$.
  Each refinement halves the radial and circumferential mesh scales,
  while the axial partition is fixed. The colors and markers identify
  the solution degree as in the one-dimensional convergence figures.
  Black dashed segments over levels $\ell=2,3,4$ show the 
  $\mathcal O(h^{\psol+1})$ reference slopes.}
  \label{fig:annular-pipe-wall-convergence}
\end{figure}
All reported errors decrease
under refinement. The first degree-two mesh pair is preasymptotic for
both formulations. On finer meshes, the trends are consistent with the
 $\mathcal O(h^{\psol+1})$ reference slopes.
Table~\ref{tab:annular-pipe-finest-rates} reports the rates on the finest
pair, computed as in Eq.~\eqref{eq:matched-mms-observed-rate} using the
respective wall errors on levels $\ell=3$ and $\ell=4$. Both formulations 
exhibit similar fine-grid rates, close to the
orders two through five. At every reported degree and refinement level,
the CR--SAT has smaller wall errors than the FEE--SAT in this test.

\begin{table}[!htb]
  \centering
  \caption{Three-dimensional annular-pipe wall-trace convergence rates for the
  entropy-stable wall settings. Observed rates between levels
  $\ell=3$ and $\ell=4$, computed using the errors in
  Eq.~\eqref{eq:annular-pipe-wall-trace-errors}, are compared with the
  nominal rate $\psol+1$. Observed rates are rounded to three decimal
  places.}
  \label{tab:annular-pipe-finest-rates}
  \begin{tabular}{@{}cccccc@{}}
    \toprule
    & Nominal rate & \multicolumn{2}{c}{CR--SAT}
    & \multicolumn{2}{c}{FEE--SAT} \\
    \cmidrule(lr){3-4}\cmidrule(l){5-6}
    $\psol$ & $\psol+1$
    & $E_u^w$ rate & $E_T^w$ rate
    & $E_u^w$ rate & $E_T^w$ rate \\
    \midrule
    $1$ & $2$ & $1.954$ & $1.952$ & $1.942$ & $1.947$ \\
    $2$ & $3$ & $2.969$ & $2.968$ & $2.963$ & $2.968$ \\
    $3$ & $4$ & $3.854$ & $3.857$ & $3.850$ & $3.856$ \\
    $4$ & $5$ & $4.933$ & $4.956$ & $4.945$ & $4.956$ \\
    \bottomrule
  \end{tabular}
\end{table}

\FloatBarrier

\subsubsection{Supersonic inclined-cube stress test}
\label{subsubsec:supersonic-inclined-cube}

The third multidimensional calculation is an operational stress test, not an
accuracy study. Figure~\ref{fig:inclined-cube-setup} shows the setup and
mesh for a stationary cube exposed to a uniform, undisturbed incoming
flow. The subscript
$\infty$ denotes this
upstream free-stream state, whereas $0$ denotes the reference scales, not a stagnation
state. With the nondimensionalization in
\ref{app:multidimensional-reference-scales}, the flow reference quantities
are selected from the free stream and the reference length is the dimensional
cube edge length:
\begin{equation}
  \begin{aligned}
    L_0&=l_e^{\mathrm{dim}},
    &\rho_0&=\rho_\infty^{\mathrm{dim}},
    \\
    U_0&=|\bu_\infty^{\mathrm{dim}}|,
    &T_0&=T_\infty^{\mathrm{dim}}.
  \end{aligned}
  \label{eq:inclined-cube-reference-scales}
\end{equation}

\begin{figure}[!htb]
  \centering
  \captionsetup[subfigure]{
    font=small,
    labelformat=parens,
    labelsep=space,
    justification=centering,
    singlelinecheck=false,
    skip=0.45em
  }

  \begin{subfigure}[c]{0.55\linewidth}
    \centering
    \includegraphics[width=\linewidth]
    {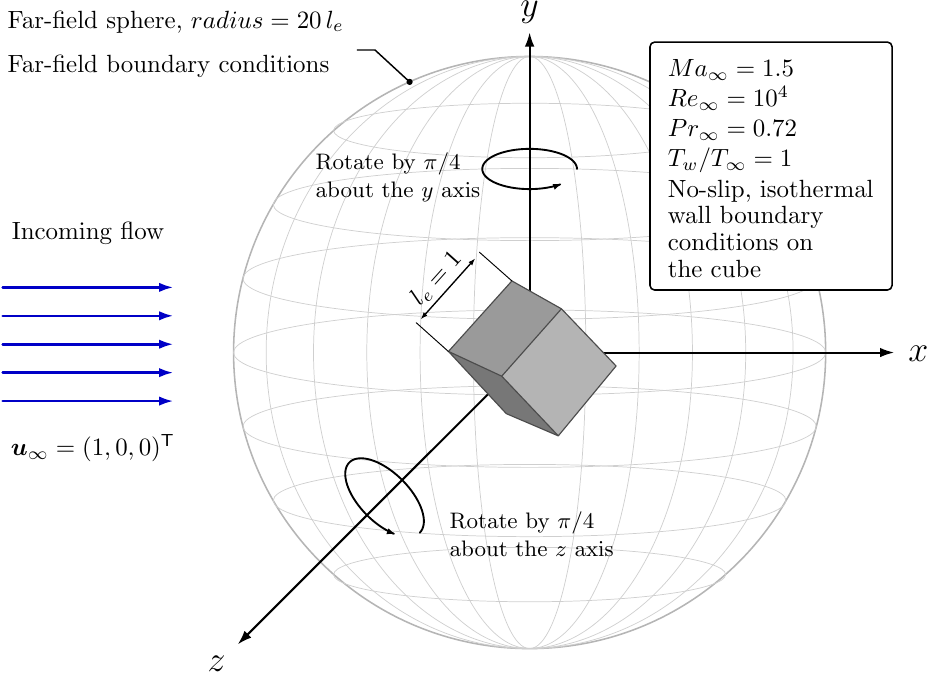}
    \caption{Schematic setup (not to scale)}
    \label{subfig:inclined-cube-setup}
  \end{subfigure}\hfill
  \begin{minipage}[c]{0.45\linewidth}
    \centering
    \begin{subfigure}[t]{\linewidth}
      \centering
      \includegraphics[width=0.50\linewidth]
        {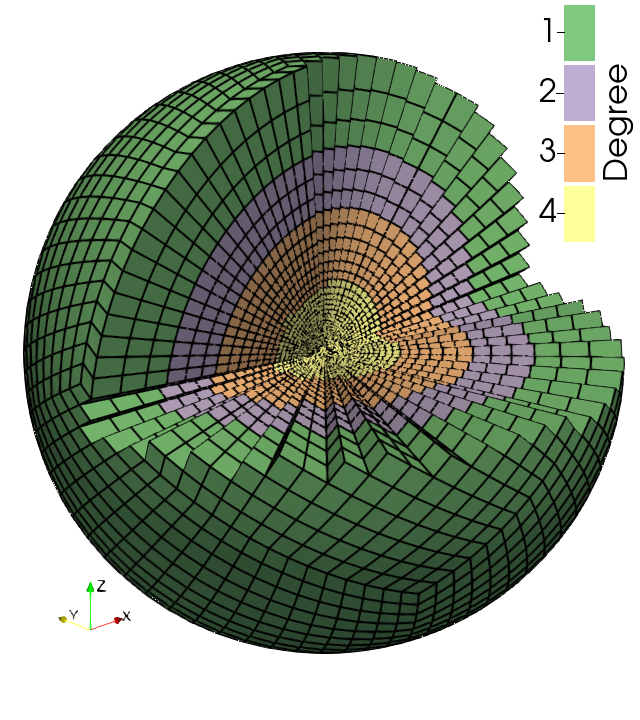}
      \caption{Domain mesh and solution-degree regions}
      \label{subfig:inclined-cube-mesh-full}
    \end{subfigure}

    \par\medskip

    \begin{subfigure}[t]{\linewidth}
      \centering
      \includegraphics[width=0.50\linewidth]
        {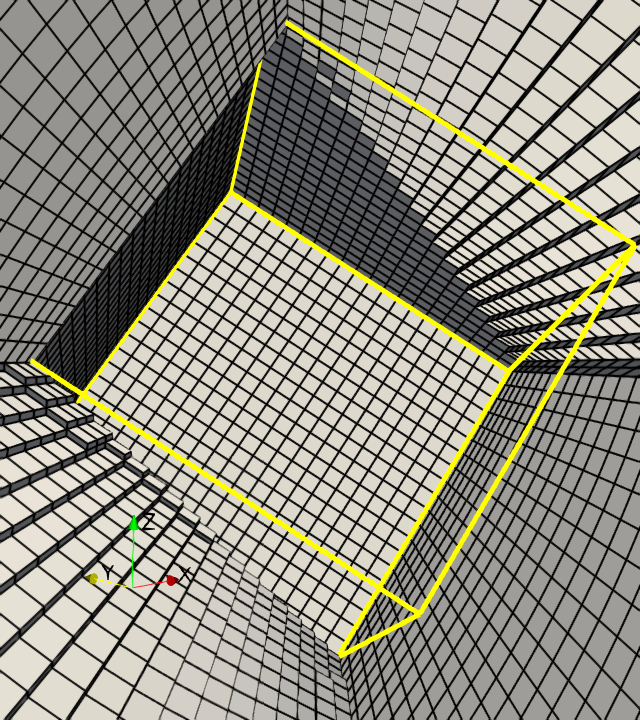}
      \caption{Near-body wall-normal grading}
      \label{subfig:inclined-cube-mesh-detail}
    \end{subfigure}
  \end{minipage}

  \caption{Setup and hexahedral mesh for the three-dimensional
  supersonic inclined-cube test. In panel (a), the incoming-flow arrows
  identify the uniform, undisturbed upstream state used for the reference
  quantities in Eq.~\eqref{eq:inclined-cube-reference-scales}.
  The stationary, no-slip cube is centered in the spherical far-field
  boundary and has prescribed temperature $T_w=T_\infty=1$.
  Panel (b) shows the complete spherical domain and the regions with
  $\psol=1,\ldots,4$. Panel (c) shows the near-body wall-normal grading.}
  \label{fig:inclined-cube-setup}
\end{figure}

The reference viscosity $\mu_0$ and thermal conductivity $\kappa_0$ are their
dimensional upstream values, and $a_0=\sqrt{\gamma R T_0}$ is the upstream
sound speed. Therefore, the reference parameters are
\begin{equation*}
  \mathrm{Ma}=\mathrm{Ma}_\infty=1.5,
  \qquad
  \mathrm{Re}=\mathrm{Re}_\infty=10^4,
  \qquad
  \mathrm{Pr}=\mathrm{Pr}_\infty=0.72.
\end{equation*}
These use the fixed upstream quantities, not the local disturbed state. The
nondimensional cube edge length is $l_e=1$, and the free-stream state is
\begin{equation*}
  \rho_\infty=1,
  \qquad
  \bu_\infty=(1,0,0)^{\transpose},
  \qquad
  T_\infty=1.
\end{equation*}
The unit nondimensional speed reflects scaling by $U_0$. However, the Mach number
remains $U_0/a_0=1.5$.

Starting with its edges aligned with the fixed global Cartesian axes, the
cube is rotated first by $\pi/4$ about the $z$ axis and then by $\pi/4$
about the $y$ axis. One body diagonal then makes the angle
\begin{equation*}
  \theta_{\mathrm{cube}}
  =
  \frac{1}{2}\arcsin\left(\frac{1}{3}\right)
  \approx 9.74^\circ
\end{equation*}
with the streamwise $x$ axis. The flow first impinges on an upstream vertex
and subsequently separates at the sharp edges. The stationary cube satisfies
$\bu_w=\bm0$ and $T_w=T_\infty=1$, so its temperature equals the undisturbed
upstream temperature. Far-field conditions are imposed on an outer sphere of
radius $20\,l_e$. No sponge layer, shock-capturing or shock-tracking
mechanism, artificial viscosity, or other shock-specific treatment is used.

The butterfly-type mesh is generated with Gmsh~\cite{Geuzaine2009:Gmsh}.
Each geometric line has twenty segments with progression factor $1.125$.
It contains $96{,}000$ hexahedra strongly graded toward the cube. The
solution degree varies by region from $\psol=1$ near the outer sphere to
$\psol=4$ near the body. Figure~\ref{fig:inclined-cube-setup} shows the
complete mesh, solution-degree distribution, and near-body wall-normal
grading.
This calculation uses CR--SAT with BS3 time integration. The values of
$\beta_{\mathrm{int}}$ and $\beta_w$ are both $0.25$.
The calculation starts from the uniform free stream at $t=0$ and advances
to nondimensional time $t=50$.

Panels (a)--(d) of Figure~\ref{fig:inclined-cube-flow} show the computed
fields at $t=50$.
Density, temperature, and velocity magnitude use the reference scales
$\rho_0$, $T_0$, and $U_0$, respectively. The density and temperature fields
show a detached bow shock and a compressed, heated windward region. A
separated low-density, low-speed wake develops downstream of the sharp edges.
The symmetry-plane speed shows the near-body wake, and the downstream plane
shows a three-lobed velocity-deficit pattern. These lobes cannot be identified
as vortices from velocity magnitude alone.

\begin{figure}[!htb]
  \centering
  \captionsetup[subfigure]{
    font=small,
    labelformat=parens,
    labelsep=space,
    justification=centering,
    singlelinecheck=false,
    skip=0.45em
  }

  \begin{subfigure}[t]{0.27\linewidth}
    \centering
    \vspace{0pt}
    \includegraphics[width=\linewidth]
      {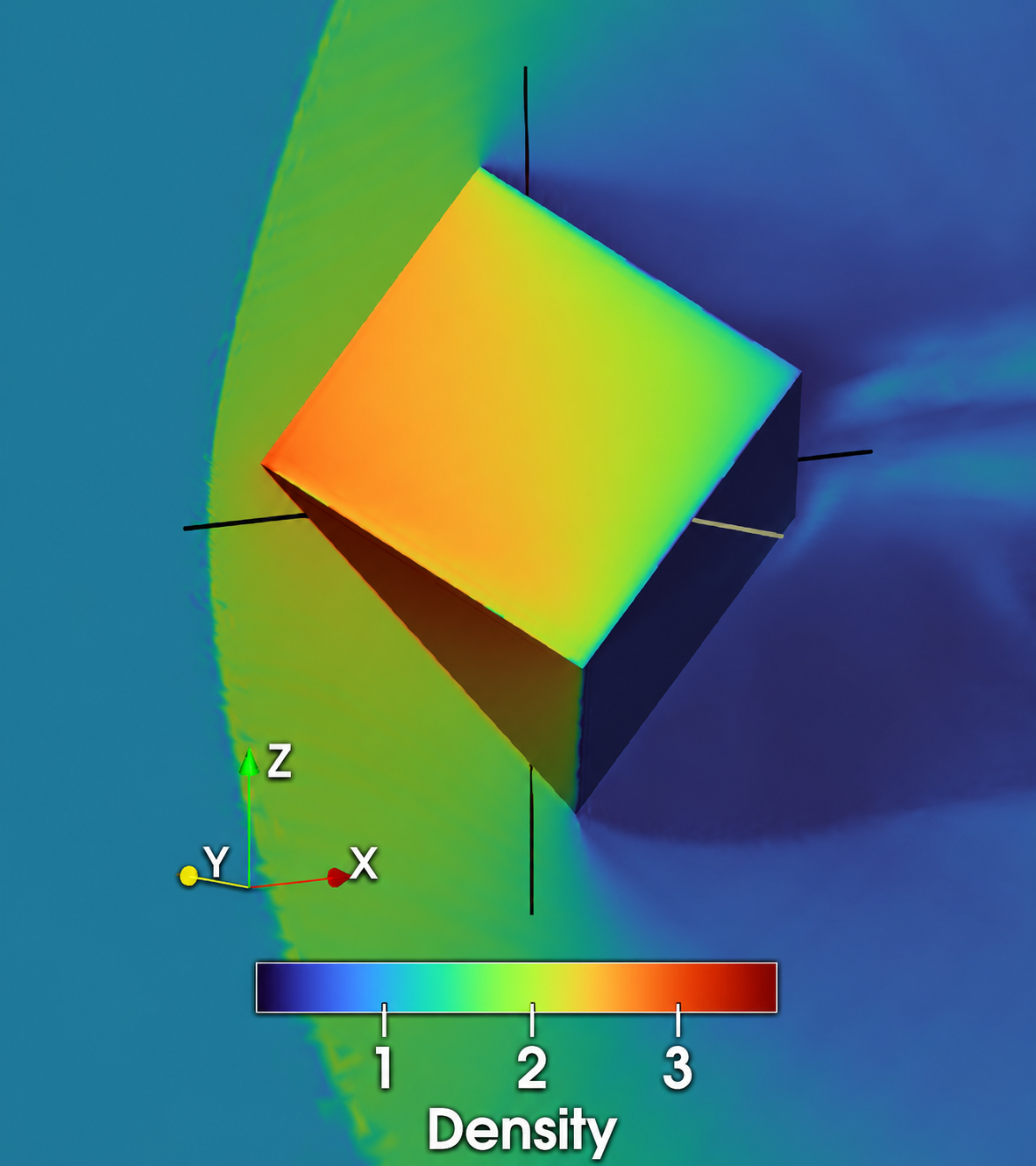}
    \caption{Density, $\rho$}
  \end{subfigure}\hfill
  \begin{subfigure}[t]{0.27\linewidth}
    \centering
    \vspace{0pt}
    \includegraphics[width=\linewidth]
      {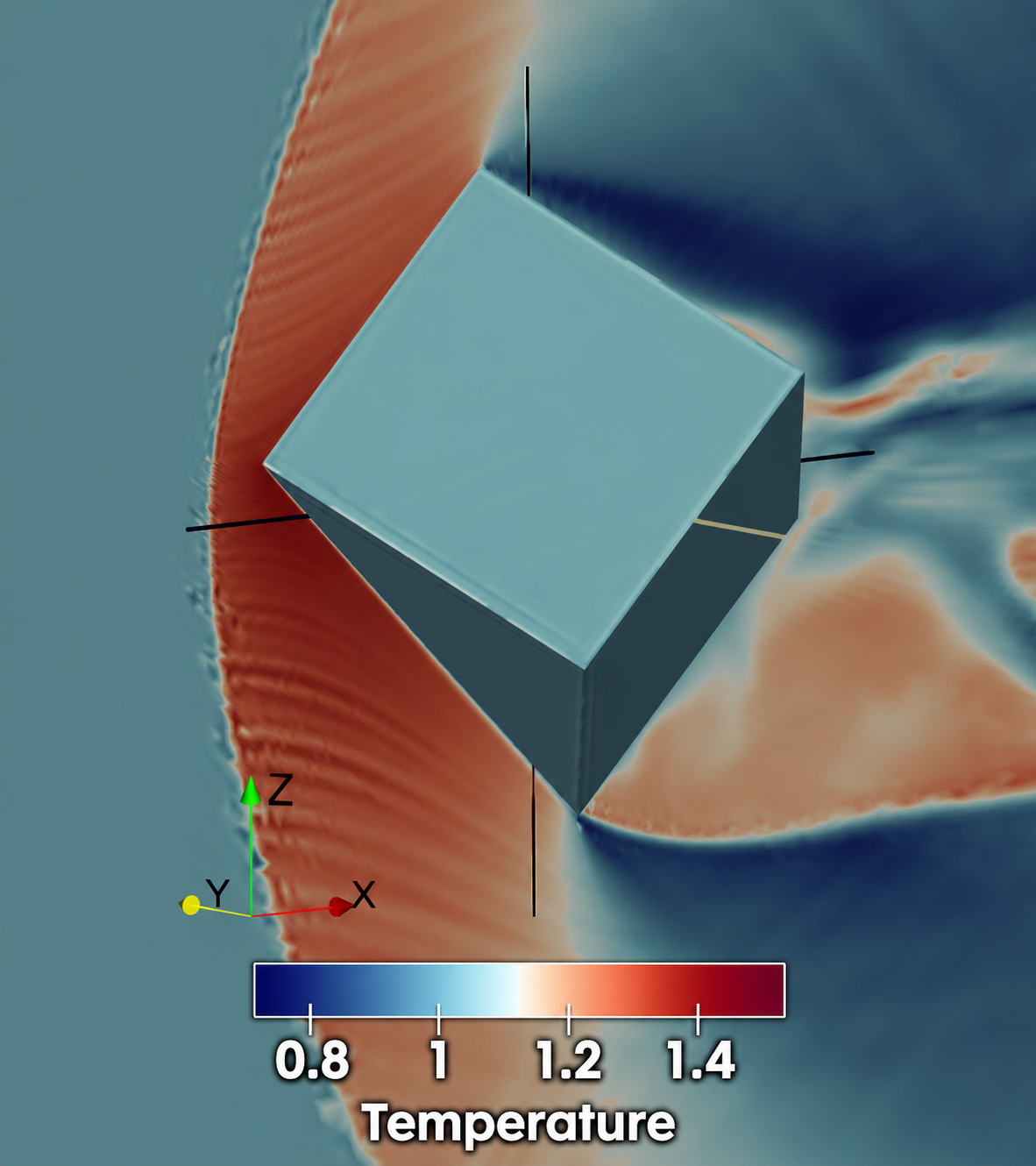}
    \caption{Temperature, $T$}
  \end{subfigure}\hfill
  \begin{subfigure}[t]{0.27\linewidth}
    \centering
    \vspace{0pt}
    \includegraphics[width=\linewidth]
      {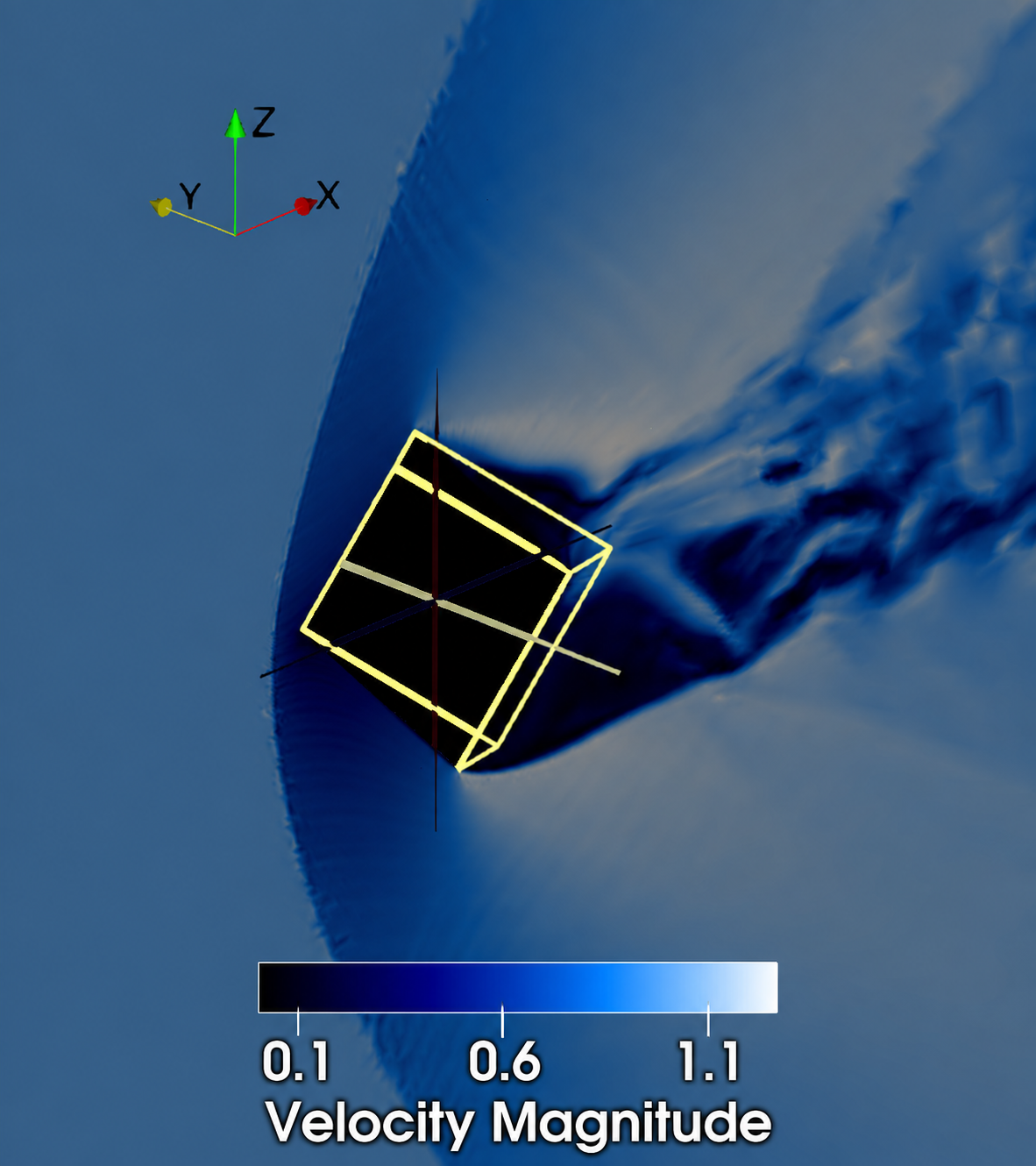}
    \caption{Velocity magnitude, $|\bu|$, on the symmetry plane}
  \end{subfigure}

  \par\medskip

  \newlength{\cubesecondrowheight}
  \newlength{\cubedownstreamheight}
  \settoheight{\cubesecondrowheight}{%
    \includegraphics[width=0.60\linewidth]
      {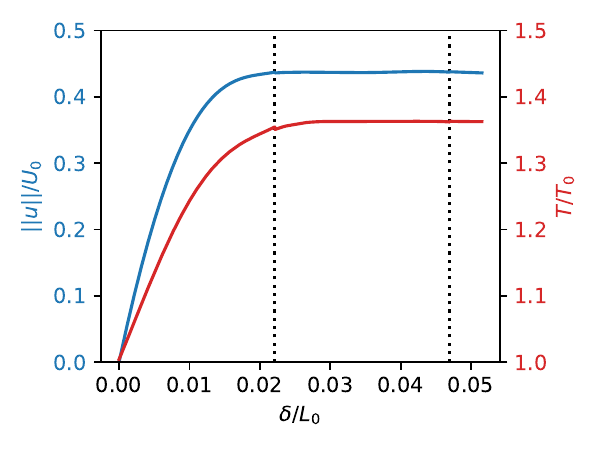}%
  }
  \settoheight{\cubedownstreamheight}{%
    \includegraphics[width=0.30\linewidth]
      {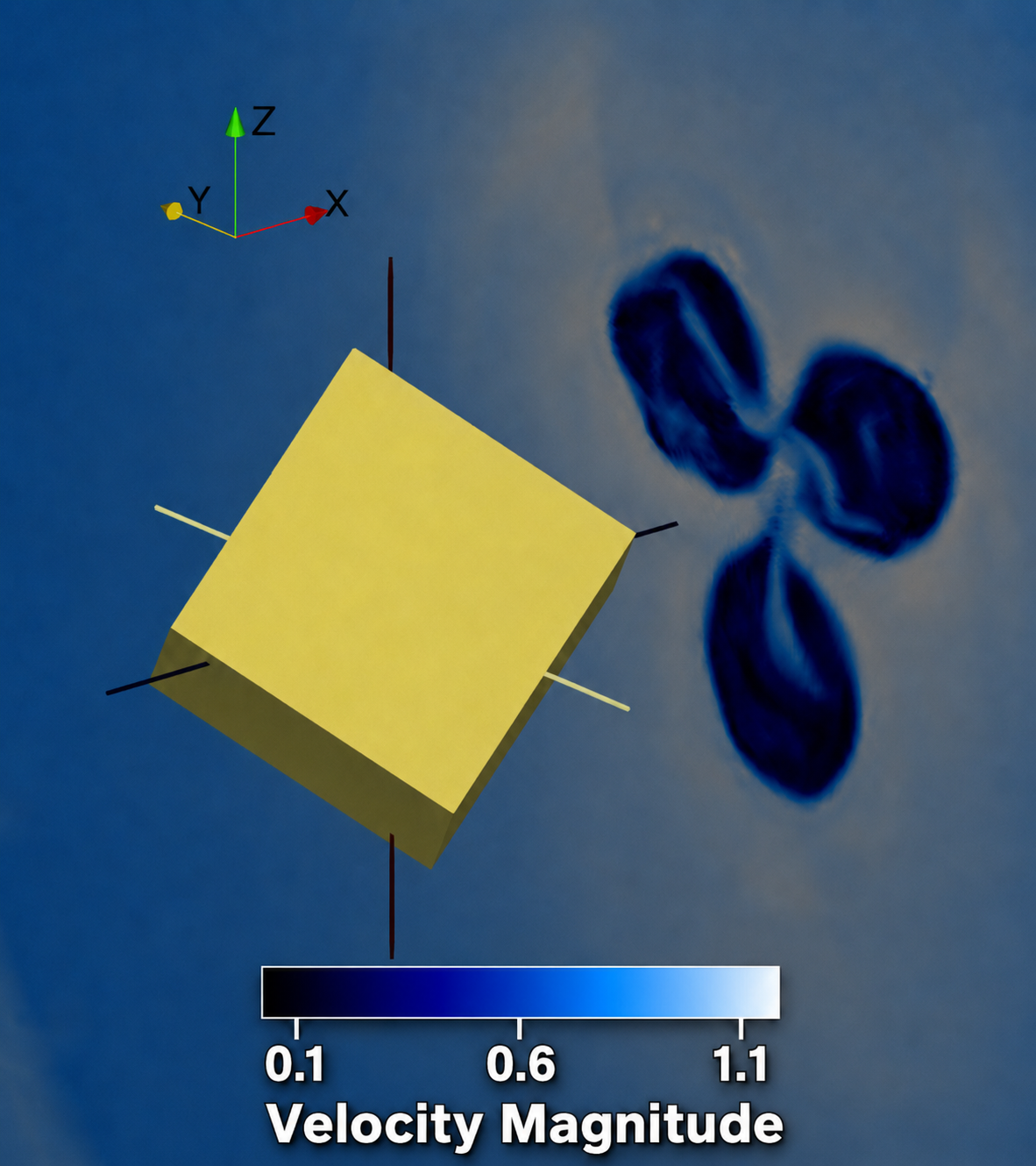}%
  }
  \ifdim\cubedownstreamheight>\cubesecondrowheight
    \setlength{\cubesecondrowheight}{\cubedownstreamheight}
  \fi

  \begin{subfigure}[t]{0.30\linewidth}
    \centering
    \vspace{0pt}
    \begin{minipage}[c][\cubesecondrowheight][c]{\linewidth}
      \centering
      \includegraphics[width=\linewidth]
        {figures/multidimensional/inclined_cube/cube_velocity_magnitude_down_plane_legend.png}
    \end{minipage}
    \caption{Velocity magnitude, $|\bu|$, on a downstream plane}
  \end{subfigure}\hfill
  \begin{subfigure}[t]{0.60\linewidth}
    \centering
    \vspace{0pt}
    \begin{minipage}[c][\cubesecondrowheight][c]{\linewidth}
      \centering
      \includegraphics[width=\linewidth]
        {figures/multidimensional/inclined_cube/cube-bndlayer.pdf}
    \end{minipage}
    \caption{Wall-normal velocity-magnitude and temperature profiles}
    \label{fig:inclined-cube-wall-profiles}
  \end{subfigure}

  \caption{Computed flow over the inclined three-dimensional cube at
  nondimensional time $t=50$ and the corresponding wall-normal profiles.
  Panels (a) and (b) show density and temperature. Panels (c) and (d)
  show velocity magnitude on the symmetry and downstream planes,
  respectively. Panel (e) shows velocity-magnitude and temperature
  profiles along the wall-normal probe from the center of the sampled
  windward face. All color bars and normalized profile coordinates use
  the reference scales in Eq.~\eqref{eq:inclined-cube-reference-scales}.
  In the plotted labels of panel (e), the numerators denote dimensional
  quantities: $\delta/L_0$, $|\bu|/U_0$, and $T/T_0$ represent the
  nondimensional distance, velocity magnitude, and temperature used for
  the description of the test case. The origin denotes the wall, and
  the vertical dotted lines mark mesh-element interfaces.}
  \label{fig:inclined-cube-flow}
\end{figure}

\begin{samepage}
Figure~\ref{fig:inclined-cube-flow}(\subref{fig:inclined-cube-wall-profiles})
shows velocity-magnitude and temperature profiles along a wall-normal
probe from the center of the sampled
windward face. In the nondimensional coordinates of
\ref{app:multidimensional-reference-scales}, let $\delta\geq0$ denote
wall-normal distance into the fluid. The probe starts at
\begin{equation*}
  \bm x_w
  =
  \left(
    -\frac{1}{2\sqrt2},
    0,
    -\frac{1}{2\sqrt2}
  \right)^{\transpose}
\end{equation*}
and follows
\begin{equation*}
  \bm x(\delta)
  =\bm x_w-\frac{\delta}{\sqrt2}(1,0,1)^{\transpose},
  \qquad
  0\leq\delta\leq0.051543.
\end{equation*}
\end{samepage}

The probe points from the cube into the fluid and its direction is opposite
to the outward unit normal of the fluid domain at the cube wall. The vertical
lines in panel~(\subref{fig:inclined-cube-wall-profiles}) mark mesh-element
interfaces. Most of the
velocity and temperature variation occurs before the first marked interface,
within the wall-adjacent region with $\psol=4$. The profiles show the local
near-wall variation and the pointwise discrepancies from the prescribed wall data reported below.

At the center of the sampled windward face, the wall-normal probe in
Figure~\ref{fig:inclined-cube-flow}(\subref{fig:inclined-cube-wall-profiles})
gives the nondimensional trace (i.e., boundary) discrepancies
\begin{equation*}
  |\bu_h|=4.46\times10^{-4},
  \qquad
  |T_h-T_w|=4.30\times10^{-3}.
\end{equation*}
These are absolute pointwise discrepancies, not relative errors with respect
to zero or wall-wide maxima, and do not establish convergence of wall stress
or heat flux.

Together, the three multidimensional cases provide complementary evidence
without duplicating the controlled one-dimensional tests: closure of an
assembled adapted-entropy balance, high-order convergence of the numerical
wall traces to the prescribed no-slip and isothermal data on smooth curved
walls, and successful ``operation'' in one supersonic, shock-dominated
separated-flow configuration.

\section{Conclusions}
\label{sec:conclusions}

We have established a homogeneous adapted-entropy wall estimate for a
stationary, impermeable, no-slip wall at constant $T_w>0$. The affine
representative $S_{T_w}=S+\rho E/T_w$ cancels the canonical exchange
$q_n/T_w$ through the total-energy balance while preserving the entropy
Hessian and relative entropy. Scaling by $T_w$ gives the total
ballistic-energy balance. Under the assumptions in
Subsection~\ref{subsec:entropy-l2-bound}, this estimate yields a conditional
$L_2$-type a priori bound based on the prescribed initial field
$\bQ_0(\bm x)$~\cite{Dafermos2016,Smoller1994}.

For DG discretizations with the SBP property, including DGSEM, we have
developed CR--SAT and FEE--SAT, which reproduce the continuous
adapted-entropy wall balance at the semi-discrete level under the stated
operator-compatibility assumptions. Both yield zero adapted-entropy wall
contribution when the wall penalty is disabled and a nonpositive quadratic
contribution when it is enabled, without imposing an artificial relation
between $T^-$ and $T_w$. For source-free problems with admissible states,
the discrete adapted entropy is nonincreasing when the remaining
boundaries are periodic or have zero adapted-entropy contribution.
CR--SAT penalizes the complete wall-data residual and balances fluid
energy through its complete numerical wall flux. FEE--SAT combines an
entropy-neutral skew correction with a momentum-projected penalty whose
mass and total-energy components vanish, making the numerical outward
total-energy flux equal to the one-sided numerical Fourier heat flux at
every wall node and spatial resolution.

Matched pointwise tests verify the local identities. Thermal-relaxation
calculations show roundoff-level closure of the semi-discrete total-energy
and adapted-entropy balances for all four wall configurations, using
independently evaluated contributions. In these tests, only FEE--SAT also
closes the Fourier-only energy balance at roundoff. The manufactured-solution
study shows comparable observed spatial convergence orders for both
formulations at fixed interior discretization. The symmetric-LDG
configuration without optional numerical entropy dissipation exhibits the
same parity-sensitive reduction for both formulations. With entropy-stable 
interior coupling, both wall-penalty settings yield
convergence rates consistent with the nominal order $\psol+1$ in the
discrete $L_2$ norm, where $\psol$ denotes the solution polynomial degree.

The three-dimensional tests demonstrate roundoff-level adapted-entropy
balance closure along a nontrivial closed-domain transient, curved-wall
trace convergence consistent with order $\psol+1$ in the discrete wall
$L_2$ norm under transverse annular-pipe refinement, and operation in one
supersonic, shock-dominated separated-flow configuration. The principal
distinction between the formulations is their finite-resolution energy
structure rather than their observed convergence order.

\appendix

\section{Canonical entropy identities}
\label{app:section2-details}

This appendix collects only the identities needed to verify the signs and
normalizations used in the main text. Standard thermodynamic derivations are
not repeated; see Refs.~\cite{Callen1985,DeGrootMazur1984,LandauLifshitz1987,
HughesFrancaMallet1986,Dafermos2016}.

\subsection{Thermodynamic reference state and admissibility}
\label{app:thermodynamics}
For a calorically perfect gas,
\begin{equation}
  T\,\mathrm ds
  =\mathrm de+p\,\mathrm d(1/\rho),
  \qquad
  s=s_{\mathrm{ref}}+c_v\log(T/T_{\mathrm{ref}})
     -R\log(\rho/\rho_{\mathrm{ref}}).
  \label{app:eq:gibbs-perfect-gas}
\end{equation}
Changing the positive reference state together with the corresponding
$s_{\mathrm{ref}}$ leaves $s(\rho,T)$ unchanged. Shifting the entropy
zero by $C_s\in\R$ replaces $s$ by $s+C_s$ and $S=-\rho s$ by
$S-C_s\rho$. Thus, $\mathrm ds$, the entropy Hessian,
entropy-variable gradients, relative entropy, and entropy production are
unchanged. The admissible set is characterized by $\rho>0$ and
$\rho e=\rho E-|\bm m|^2/(2\rho)>0$. The latter inequality is
equivalent to $T>0$ for this gas.

\subsection{Constitutive and entropy-production identities}
\label{app:constitutive-details}

The Newtonian stress may equivalently be written as
\begin{equation*}
  \btau
  =
  \mu
  \left(
    \nabla\bu+(\nabla\bu)^{\transpose}
  \right)
  +\lambda(\nabla\cdot\bu)\bm I,
  \qquad
  \lambda=\zeta-\frac{2\mu}{d}.
\end{equation*}
Here, $\lambda$ is the usual second coefficient of viscosity. In three
spatial dimensions, Stokes' hypothesis $\lambda=-2\mu/3$ is equivalent to
$\zeta=0$. With the outward heat flux
$q_n=-\kappa\partial_nT$, the local entropy production is
\begin{equation}
  \Dcal
  =
  \frac{2\mu}{T}\bm D^0(\bu):\bm D^0(\bu)
  +\frac{\zeta}{T}(\nabla\!\cdot\bu)^2
  +\frac{\kappa}{T^2}|\nabla T|^2
  \geq0.
  \label{app:eq:entropy-production-derivation}
\end{equation}
This follows directly from the Newtonian stress, Fourier's law, and the Gibbs
relation. This compact form is standard. Explicit component manipulations may
be found in Refs.~\cite{HughesFrancaMallet1986,Fisher2012PhD,FisherCarpenter2013,CarpenterEtAl2014}.

\subsection{Canonical entropy balance}
\label{app:canonical-balance-derivation}

On the admissible set, $S=-\rho s$ is strictly convex and
$\bW=\nabla_{\bQ}S$ is one-to-one
\cite{Harten1983,HughesFrancaMallet1986,Dafermos2016}. Contracting the
compressible Navier--Stokes equations with $\bW$ gives
\begin{equation}
  \partial_tS+\partial_i(Su_i)
  =
  \partial_i\!\left(\bW^{\transpose}\bFV{i}\right)
  -(\partial_i\bW)^{\transpose}\bm C_{ij}\partial_j\bW.
  \label{app:eq:canonical-local-balance}
\end{equation}
The quadratic form is nonnegative because
$[\bm C_{ij}]_{i,j=1}^{d}\succeq\bm0$. Hence, the
last term in Eq.~\eqref{app:eq:canonical-local-balance} is nonpositive.
This identity is the only canonical-balance result used later.

\section{Stationary isothermal-wall balance}
\label{app:section3-details}

Let $\bn$ be outward and let $q_n>0$ denote heat leaving the fluid.
No penetration, $u_n=0$, eliminates the inviscid entropy flux and the
advective part of the total-energy flux:
\begin{equation}
  F_n=Su_n=0,
  \qquad
  (\rho E+p)u_n=0.
  \label{app:eq:wall-inviscid-entropy-flux}
\end{equation}
The full no-slip condition, $\bu=\bm0$, additionally eliminates the
viscous mechanical-work term $\bu\cdot\btau\bn$. Independently, direct
contraction of the normal viscous flux in
Eq.~\eqref{eq:normal-viscous-flux} gives
\begin{equation}
  \bW^{\transpose}\bm F_n^{(V)}
  =\frac{q_n}{T}.
  \label{app:eq:wall-viscous-entropy-contraction}
\end{equation}
Consequently, imposing $T=T_w$ gives the canonical open-system contribution
\begin{equation}
  \left.\Bcal_S\right|_{\wall}
  =
  \frac{q_n}{T_w}.
  \label{app:eq:isothermal-canonical-wall-contribution}
\end{equation}
The net outward total-energy flux at the same wall is
\begin{equation}
  F_{E,n}=q_n.
  \label{app:eq:isothermal-wall-energy-flux}
\end{equation}
Therefore, $\Bcal_S-F_{E,n}/T_w=0$ pointwise. This cancellation is the
continuous mechanism reproduced by the discrete wall treatments presented
herein. A detailed derivation of the canonical wall-entropy contribution
in the SBP--SAT wall-boundary setting is provided in
Refs.~\cite{ParsaniCarpenterNielsen2014NASA,ParsaniCarpenterNielsen2015}.

\section{Affine entropy and ballistic-energy identities}
\label{app:adapted-entropy-details}

This appendix supplies the supporting derivations for
Subsections~\ref{subsec:affine-entropy}, \ref{subsec:entropy-l2-bound},
and~\ref{subsec:ballistic-free-energy}. The wall-adapted pair, inherited
structure, and local balance are defined in
Eqs.~\eqref{eq:adapted-entropy}--\eqref{eq:local-adapted-entropy}.

\subsection{Proof of affine invariance}
\label{app:affine-invariance-proof}

\begin{proof}[Proof of Lemma~\ref{lem:affine-entropy}]
Let
\begin{equation*}
  \bm A_i(\bQ)
  =
  \frac{\p\bFI{i}}{\p\bQ}.
\end{equation*}
Because $\bm a$ is constant, the entropy-pair compatibility relation gives
\begin{align*}
  \bm A_i^{\transpose}\bm W^{\bm a}
  &=
  \bm A_i^{\transpose}\bW
  +
  \bm A_i^{\transpose}\bm a
  \\
  &=
  \nabla_{\bQ}\FI{i}
  +
  \nabla_{\bQ}
  \left(
    \bm a^{\transpose}\bFI{i}
  \right)
  =
  \nabla_{\bQ}F_i^{\bm a}.
\end{align*}
Hence, $(S^{\bm a},F_i^{\bm a})$ is an entropy pair for the inviscid
equations. Direct substitution gives
\begin{equation*}
  \Phi^{\bm a}
  =
  (\bW+\bm a)^{\transpose}\bQ
  -
  \left(
    S+\bm a^{\transpose}\bQ
  \right)
  =
  \Phi,
\end{equation*}
and, similarly, $\Psi_i^{\bm a}=\Psi_i$. Finally, because
$\bm a$ is constant, we have $\p_i\bm W^{\bm a}=\p_i\bW$. Therefore, the viscous
coefficient matrices and the volume entropy-production density are unchanged.
\end{proof}

\subsection{Derivation of the conditional
  \texorpdfstring{$L_2$}{L2}-type bound}
\label{app:entropy-l2-derivation}

The conditional $L_2$-type bound follows from the classical convex-entropy
stability framework developed by Dafermos in
1979~\cite{Dafermos1979}; see also
Dafermos~\cite[Chapter~V]{Dafermos2016} and
Smoller~\cite{Smoller1994}. For completeness, we detail its application
to the wall-adapted entropy under the assumptions of
Subsection~\ref{subsec:entropy-l2-bound}.

Fix a time $t$ for which these assumptions hold and
$\underline{\lambda}_{T_w}(t)>0$.
Equation~\eqref{eq:global-adapted-dissipation-identity} implies
\begin{equation*}
  \int_{\Omega}
  \ST\bigl(\bQ(\bm x,t)\bigr)\,\dd\Omega
  \leq
  \int_{\Omega}
  \ST\bigl(\bQ_0(\bm x)\bigr)\,\dd\Omega.
\end{equation*}
The admissible set $\mathcal A$ is convex because
$(\rho,\bm m)\mapsto\lvert\bm m\rvert^2/(2\rho)$ is convex for
$\rho>0$. Hence, whenever $\bQ_0(\bm x)$ and $\bQ(\bm x,t)$ are
admissible, the entire state segment
\begin{equation*}
  (1-\theta)\bQ_0(\bm x)+\theta\bQ(\bm x,t),
  \qquad \theta\in[0,1],
\end{equation*}
remains in $\mathcal A$.
By the essential-infimum definition in
Eq.~\eqref{eq:adapted-entropy-minimum-hessian-eigenvalue}, the smallest
eigenvalue of the entropy Hessian along this segment is bounded below
by $\underline{\lambda}_{T_w}(t)$ for almost every $\bm x\in\Omega$
and all $\theta\in[0,1]$.
Here, ``almost every'' allows an exceptional spatial set of zero
Lebesgue measure, which does not affect the volume integrals.
Taylor's theorem along the segment, followed by integration over
$\Omega$, use of the entropy inequality above, and the
Cauchy--Schwarz inequality, yields
\begin{align*}
  \frac{\underline{\lambda}_{T_w}(t)}{2}
  \left\|
    \bQ(\cdot,t)-\bQ_0
  \right\|_{L_2(\Omega)}^2
  &\leq
  -
  \int_{\Omega}
  \bWT\bigl(\bQ_0(\bm x)\bigr)^{\transpose}
  \left[
    \bQ(\bm x,t)-\bQ_0(\bm x)
  \right]
  \dd\Omega
  \\
  &\leq
  \left\|
    \bWT(\bQ_0)
  \right\|_{L_2(\Omega)}
  \left\|
    \bQ(\cdot,t)-\bQ_0
  \right\|_{L_2(\Omega)}.
\end{align*}
If $\bQ(\cdot,t)=\bQ_0$ in $L_2(\Omega)$, the bound is immediate.
Otherwise, division by
$\lVert\bQ(\cdot,t)-\bQ_0\rVert_{L_2(\Omega)}$ gives
\begin{equation*}
  \left\|
    \bQ(\cdot,t)-\bQ_0
  \right\|_{L_2(\Omega)}
  \leq
  \frac{2}{\underline{\lambda}_{T_w}(t)}
  \left\|
    \bWT(\bQ_0)
  \right\|_{L_2(\Omega)}.
\end{equation*}
The triangle inequality then yields
Eq.~\eqref{eq:continuous-entropy-l2-bound}. A positive lower bound for
$\underline{\lambda}_{T_w}(t)$ that is uniform over the time interval
under consideration gives the uniform-in-time estimate.
Note that strict convexity alone does not guarantee such a uniform lower bound.

\subsection{Ballistic-energy identities}
\label{app:ballistic-energy-identities}

For constant $\Tw>0$, Eq.~\eqref{eq:total-ballistic-energy-density}
and the adapted-entropy Hessian give
\begin{equation}
  \nabla_{\bQ}^{2}
  \mathcal E_{\Tw}^{\mathrm{bal}}
  =
  \Tw\nabla_{\bQ}^{2}\ST
  =
  \Tw\bm H_S
  \succ\bm0
  \qquad
  \text{on }\mathcal A.
  \label{eq:ballistic-energy-hessian}
\end{equation}
Multiplying Eq.~\eqref{eq:global-adapted-wall-identity} by $\Tw$ yields
\begin{equation}
  \frac{\dd}{\dd t}
  \int_{\Omega}
  \mathcal E_{\Tw}^{\mathrm{bal}}(\bQ)
  \,\dd\Omega +
  \Tw
  \int_{\Omega}
  \Dcal\,\dd\Omega = \Tw
  \int_{\partial\Omega\setminus\wall}
  \Bcal_{\ST}\,\dd\Gamma.
  \label{eq:global-ballistic-energy-balance}
\end{equation}
If the remaining boundary portions have zero adapted-entropy
contribution, this reduces to
\begin{equation}
  \frac{\dd}{\dd t}
  \int_{\Omega}
  \mathcal E_{\Tw}^{\mathrm{bal}}(\bQ)
  \,\dd\Omega
  + \Tw \int_{\Omega} \Dcal\,\dd\Omega = 0.
  \label{eq:global-ballistic-dissipation-identity}
\end{equation}
At a stationary no-slip wall,
\begin{align}
  \left.
  \Tw\Bcal_{\ST}
  \right|_{\wall}
  &=
  \left.
  \left(
    \Tw\Bcal_S-F_{E,n}
  \right)
  \right|_{\wall}
  \notag\\
  &=
  \qnormal
  \left(
    \frac{\Tw}{T}-1
  \right).
  \label{eq:ballistic-wall-contribution}
\end{align}
The right-hand side vanishes for $T=\Tw$, consistently with
Theorem~\ref{thm:isothermal-wall-adapted}.

\subsection{Relative-entropy invariance under the affine modification}
\label{app:relative-entropy-invariance}

For a differentiable convex function $H$ and states
$\bQ,\overline{\bQ}\in\mathcal A$, define the Bregman divergence
\cite{Bregman1967}
\begin{equation}
  H(\bQ\mid\overline{\bQ})
  =
  H(\bQ)-H(\overline{\bQ})
  -\nabla_{\bQ}H(\overline{\bQ})^{\transpose}
   (\bQ-\overline{\bQ}).
  \label{app:eq:relative-entropy-definition}
\end{equation}
When $H$ is a mathematical entropy, this quantity is commonly called the
relative entropy. Substituting
$S_{T_w}=S+\be_E^{\transpose}\bQ/T_w$ and
$\bW_{T_w}=\bW+\be_E/T_w$ into
Eq.~\eqref{app:eq:relative-entropy-definition} gives
\begin{equation}
  S_{T_w}(\bQ\mid\overline{\bQ})
  =
  S(\bQ\mid\overline{\bQ}).
  \label{app:eq:relative-entropy-invariance}
\end{equation}
For every pair of admissible states, the affine contribution cancels
exactly between the entropy difference and the tangent-plane term. This
identity is not needed for the conditional $L_2$-type bound in
Subsection~\ref{subsec:entropy-l2-bound}.

\section{Discrete operators, viscous matrices, and implementation details}
\label{app:section5-support}

\subsection{Explicit Cartesian affine tensor-product SBP construction}
\label{app:sbp-operator-construction}

Let $\widehat K=[-1,1]^d$. For the LGL nodes
$\{\xi_a\}_{a=0}^{\psol}$, weights
$\{\omega_a\}_{a=0}^{\psol}$, and Lagrange basis functions $\ell_b$, 
define~\cite{CarpenterEtAl2014,Gassner2013,CarpenterParsaniFisherNielsen2015}
\begin{equation}
  \mathsf P_{\Nsol}
  =\operatorname{diag}(\omega_0,\ldots,\omega_{\psol}),
  \qquad
  (\mathsf D_{\Nsol})_{ab}=\ell_b'(\xi_a),
  \qquad
  \mathsf Q_{\Nsol}
  =\mathsf P_{\Nsol}\mathsf D_{\Nsol}.
  \label{app:eq:one-dimensional-lgl-operators}
\end{equation}
These matrices satisfy Eq.~\eqref{eq:one-dimensional-sbp-property}.

Consider
\begin{equation*}
  K=\prod_{i=1}^{d}[x_{i,L},x_{i,H}],
  \qquad
  h_{i,K}=x_{i,H}-x_{i,L},
  \qquad
  x_{i,c}=\frac{x_{i,L}+x_{i,H}}{2}.
\end{equation*}
Here, $x_{i,L}$ and $x_{i,H}$ are the lower and upper endpoints of
$K$ in Cartesian direction $i$.
Let $\delta_{ij}$ denote the Kronecker delta. For
$\boldsymbol{\xi}=(\xi_1,\ldots,\xi_d)^{\transpose}\in\widehat K$
and
$\boldsymbol{x}=(x_1,\ldots,x_d)^{\transpose}\in K$, the coordinate-wise
affine map and its constant metric factors are
\begin{align}
  x_i(\boldsymbol{\xi})
  &=x_{i,c}+\frac{h_{i,K}}{2}\xi_i,
  \label{app:eq:cartesian-affine-map}
  \\
  \frac{\partial x_i}{\partial\xi_j}
  &=\frac{h_{i,K}}{2}\delta_{ij},
  &
  \frac{\partial\xi_i}{\partial x_j}
  &=\frac{2}{h_{i,K}}\delta_{ij}.
  \label{app:eq:cartesian-affine-metrics}
\end{align}
The volume Jacobian determinant is
\begin{equation}
  J_K
  =\det\!\left(
    \frac{\partial\boldsymbol{x}}
         {\partial\boldsymbol{\xi}}
  \right)
  =\prod_{i=1}^{d}\frac{h_{i,K}}{2}>0.
  \label{app:eq:cartesian-volume-jacobian}
\end{equation}
Thus, $J_K$ converts reference-volume quadrature to physical-volume
quadrature.

Let $\mathsf I_{\Nsol}$ be the $\Nsol\times\Nsol$ identity matrix. The
physical scalar tensor-product operators are
\begin{align}
  \mathsf P_K
  &=J_K\mathsf P_{\Nsol}^{\otimes d},
  \\
  \mathsf D_{i,K}
  &=\frac{2}{h_{i,K}}
    \mathsf I_{\Nsol}^{\otimes(i-1)}
    \otimes\mathsf D_{\Nsol}
    \otimes\mathsf I_{\Nsol}^{\otimes(d-i)},
  \\
  \mathsf Q_{i,K}
  &=\mathsf P_K\mathsf D_{i,K}.
  \label{app:eq:tensor-product-sbp-operators}
\end{align}
Here,
$\mathsf P_K,\mathsf D_{i,K},\mathsf Q_{i,K}
\in\R^{N_K\times N_K}$ and $N_K=\Nsol^d$.

For flattened volume-node indices $a,b\in\{1,\ldots,N_K\}$ with
multi-indices $\boldsymbol{\alpha},\boldsymbol{\beta}
\in\{0,\ldots,\psol\}^d$, the derivative entries are
\begin{equation*}
  (\mathsf D_{i,K})_{ab}
  =
  \frac{2}{h_{i,K}}
  (\mathsf D_{\Nsol})_{\alpha_i\beta_i}
  \prod_{\substack{j=1\\j\neq i}}^d
  \delta_{\alpha_j\beta_j}.
\end{equation*}
Only nodes with identical coordinates in all directions $j\neq i$
contribute to the flux-differencing sum in
Eq.~\eqref{eq:entropy-conservative-volume-flux-differencing}. Thus, the
sum is evaluated along the coordinate line through node $a$ in direction
$i$, often called an $i$-direction pencil.

For $r\in\{1,\ldots,d\}$ and $\sigma\in\{-1,1\}$, let
$f=f_{r,\sigma}$ be the image of the reference face $\xi_r=\sigma$.
All face operators below refer to the current element $K$, whose index
is suppressed where unambiguous. Define
\begin{equation*}
  \bm s_{-1}=(1,0,\ldots,0)^{\transpose},
  \qquad
  \bm s_{1}=(0,\ldots,0,1)^{\transpose},
  \qquad
  \bm s_{\sigma}\in\R^{\Nsol}.
\end{equation*}
With the tensor-product ordering used above, the restriction matrix is
\begin{equation}
  \mathsf R_f
  =\mathsf I_{\Nsol}^{\otimes(r-1)}
   \otimes\bm s_{\sigma}^{\transpose}
   \otimes\mathsf I_{\Nsol}^{\otimes(d-r)}
  \in\R^{N_f\times N_K},
  \qquad
  N_f=\Nsol^{d-1}.
  \label{app:eq:cartesian-face-restriction}
\end{equation}
The surface Jacobian and outward unit normal are
\begin{equation}
  J_{f,K}
  =\prod_{\substack{j=1\\j\neq r}}^{d}
    \frac{h_{j,K}}{2}
  =\frac{2J_K}{h_{r,K}},
  \qquad
  \bn_f=\sigma\be_r.
  \label{app:eq:cartesian-face-geometry}
\end{equation}
Here, $\be_r$ is the $r$th Cartesian basis vector, and $\bn_f$ is outward
from the current element $K$. At an interior face, $\bn_f=\bn$ for
$K=K^-$ and $\bn_f=-\bn$ for $K=K^+$, with the common normal of
Subsection~\ref{subsec:interface-2015-coupling}.
The factor $J_{f,K}$ converts reference-face quadrature to physical-face
quadrature. Let $\bm1_f\in\R^{N_f}$ be the face vector of ones
and define
\begin{equation*}
  \bm n_{i,f}
  =\sigma\delta_{ir}\bm1_f.
\end{equation*}
The physical face quadrature and normal-component matrices are
\begin{equation}
  \mathsf P_f
  =J_{f,K}\mathsf P_{\Nsol}^{\otimes(d-1)},
  \qquad
  \mathsf N_{i,f}
  =\operatorname{diag}(\bm n_{i,f}),
  \label{app:eq:face-quadrature-and-normal}
\end{equation}
with
$\mathsf P_f,\mathsf N_{i,f}
\in\R^{N_f\times N_f}$. These definitions give
Eq.~\eqref{eq:multidimensional-sbp-property}.

The physical one-dimensional norm in direction $r$ is
$(h_{r,K}/2)\mathsf P_{\Nsol}$. Hence, its endpoint entry on either
face $f_{r,\sigma}$ is
\begin{equation}
  p_{n,f}
  =\frac{h_{r,K}}{2}\,\omega_{\mathrm{end}}>0,
  \qquad
  \omega_{\mathrm{end}}
  =\omega_0=\omega_{\psol}.
  \label{eq:normal-endpoint-mass-entry}
\end{equation}
The same formula defines $p_{n,w}$ on a physical wall face, using the local
wall-normal element length. The factor $1/2$ arises because the
reference interval $[-1,1]$ has length two.

This appendix gives the Cartesian affine representation used in the proof.
Subsection~\ref{subsec:tensor-product-sbp} points to the additional metric and
interface conditions required by curvilinear and nonconforming extensions.

\subsection{Viscous coefficient matrices}
\label{app:viscous-coefficient-matrices}

Let $m=d+2$ and let
$\bm z=(z_\rho,\bm z_u,z_E)^{\transpose}\in\R^m$ be an
entropy-variable direction. The corresponding primitive-variable variations
are
\begin{equation}
  \delta T=T^2z_E,
  \qquad
  \delta\bu=T\bm z_u+Tz_E\bu.
  \label{app:eq:entropy-to-primitive-direction}
\end{equation}
For $r,s\in\{1,\ldots,d\}$ and $\lambda=\zeta-2\mu/d$, define
\begin{equation}
  (\bm A_{ij})_{rs}
  =
  T\left[
    \mu(\delta_{ij}\delta_{rs}+\delta_{rj}\delta_{si})
    +\lambda\delta_{ri}\delta_{sj}
  \right].
  \label{app:eq:Aij-definition}
\end{equation}
The single-state viscous coefficient matrix is
\begin{equation}
  \bm C_{ij}(\bQ)
  =
  \begin{pmatrix}
    0 & \bm0^{\transpose} & 0\\
    \bm0 & \bm A_{ij} & \bm A_{ij}\bu\\
    0 & \bu^{\transpose}\bm A_{ij}
      & \bu^{\transpose}\bm A_{ij}\bu+\kappa T^2\delta_{ij}
  \end{pmatrix}.
  \label{app:eq:Cij-block-form}
\end{equation}
This dimension-independent representation satisfies
$\bm C_{ij}=\bm C_{ji}^{\transpose}$ and
$[\bm C_{ij}]_{i,j=1}^{d}\succeq\bm0$ under
Eq.~\eqref{eq:transport-assumptions}. It is equivalent to the explicit
three-dimensional matrices in
Refs.~\cite{Fisher2012PhD,ParsaniCarpenterNielsen2015}. The nodal matrices in
Eq.~\eqref{eq:block-diagonal-viscous-coefficients} are obtained by evaluating
Eq.~\eqref{app:eq:Cij-block-form} at every node of $K$ and assembling the
results block diagonally.

\subsection{Admissible wall numerical state}
\label{app:viscous-boundary-state-admissibility}

Using the wall notation of
Subsection~\ref{subsec:isothermal-wall-common-data}, let $\bQ^-$ be
the admissible interior trace and let
$\widehat{\bW}_w=(W_1^-,\bm0,-1/\Tw)^{\transpose}$ with $\Tw>0$.
The last $d+1$ components imply
$\widehat\bu_w=\bm0$ and $\widehat T_w=\Tw$. For a calorically perfect gas,
$h/T=c_p$. If $\widehat s_w=s(\widehat\rho_w,\Tw)$, equality of the first
entropy-variable components gives
\begin{equation*}
  \widehat s_w
  =s^-+\frac{\lvert\bu^-\rvert^2}{2T^-}.
\end{equation*}
Using Eq.~\eqref{eq:thermodynamic-entropy} for the interior and wall states
then yields
\begin{equation}
  \widehat\rho_w
  =
  \rho^-
  \left(
    \frac{\Tw}{T^-}
  \right)^{c_v/R}
  \exp\!\left(
    -\frac{\lvert\bu^-\rvert^2}{2RT^-}
  \right)
  >0.
  \label{app:eq:admissible-wall-density}
\end{equation}
Consequently,
\begin{equation*}
  \widehat{\bQ}_w
  =
  \begin{pmatrix}
    \widehat\rho_w\\
    \bm0\\
    \widehat\rho_w c_v\Tw
  \end{pmatrix}
  \in\mathcal A
\end{equation*}
for every admissible interior state and every $\Tw>0$. Thus, the inverse
entropy map is applied only to the admissible wall value
$\widehat{\bW}_w$. No reflected exterior entropy-variable state is used to
evaluate the viscous coefficient matrices.

\subsection{Fourier-energy-exact skew correction and projected penalty}
\label{app:fee-wall-identities}

We use the nodewise wall notation of
Subsection~\ref{subsec:isothermal-wall-common-data}.
The skew matrix in Eq.~\eqref{eq:fee-skew-matrix} satisfies
$\bm K_w^{\transpose}=-\bm K_w$. Transposing the scalar quadratic
form changes its sign, so
\begin{equation}
  \bm r_w^{\transpose}\bm K_w\bm r_w=0.
  \label{app:eq:fee-skew-quadratic}
\end{equation}
Its energy row acting on $\bm r_w$ is
$-\bu^-\cdot\bm t^-$, which cancels the mechanical-work component of the
one-sided viscous flux. The projected matrix satisfies
\begin{equation}
  \bm\Lambda_{w,u}^{\mathrm{FEE}}
  =
  \frac{\beta_w}{p_{n,w}}
  \bm\Pi_u
  \left(
    \bm C_{nn,w}^{-}+\widehat{\bm C}_{nn,w}
  \right)
  \bm\Pi_u
  =
  (\bm\Lambda_{w,u}^{\mathrm{FEE}})^{\transpose}
  \succeq\bm0.
  \label{app:eq:fee-projected-penalty}
\end{equation}
The property
$\bm C_{nn,w}^{-}+\widehat{\bm C}_{nn,w}\succeq\bm0$ and
$\bm\Pi_u^{\transpose}=\bm\Pi_u$ imply that the projected matrix is
symmetric. For any $\bm z\in\R^m$, its quadratic form is
$\beta_w/p_{n,w}$ times the quadratic form of
$\bm C_{nn,w}^{-}+\widehat{\bm C}_{nn,w}$ at $\bm\Pi_u\bm z$.
Both factors are nonnegative, so
$\bm\Lambda_{w,u}^{\mathrm{FEE}}\succeq\bm0$.
The projected matrix has zero mass and energy rows and columns.
Moreover, $\bW_{T_w}^-=\widehat{\bW}_{T_w,w}+\bm r_w$, and
$\widehat{\bW}_{T_w,w}$ has only its first component potentially nonzero.
Its contraction with the projected penalty therefore vanishes, giving
\begin{equation}
  \be_E^{\transpose}\bm M_{w,u}^{\mathrm{FEE}}=0,
  \qquad
  (\bW_{T_w}^-)^{\transpose}\bm M_{w,u}^{\mathrm{FEE}}
  =
  -\bm r_w^{\transpose}
  \bm\Lambda_{w,u}^{\mathrm{FEE}}\bm r_w
  \leq0.
  \label{app:eq:fee-penalty-identities}
\end{equation}
For any $\bm A\in\R^{m\times m}$ with $\bm A\succeq\bm0$, let
$\bm A^{1/2}\succeq\bm0$ denote its symmetric square root. Then
$\bm r_w^{\transpose}\bm A\bm r_w
=\lVert\bm A^{1/2}\bm r_w\rVert_2^2$, which vanishes exactly when
$\bm A\bm r_w=\bm0$.
Applying this observation to $\bm\Lambda_w^{\mathrm{CR}}$ and
$\bm\Lambda_{w,u}^{\mathrm{FEE}}$ gives the equality conditions stated
in Subsections~\ref{subsec:cr-sat} and~\ref{subsec:fee-sat}.
For a pure temperature mismatch, $\bm\Pi_u\bm r_w=\bm0$, so the
FEE--SAT projected penalty vanishes.
Combining Eqs.~\eqref{app:eq:fee-skew-quadratic} and
\eqref{app:eq:fee-penalty-identities} yields the local entropy identity
\eqref{eq:fee-wall-entropy-contribution}. The skew correction and zero
energy row of the projected penalty yield the local energy identity
\eqref{eq:fee-fourier-energy-identity} and, after summation over the wall,
the global balance \eqref{eq:fee-global-energy-balance}.

\subsection{Discrete entropy contraction}
\label{app:semidiscrete-entropy-identity}

For the entropy-scaled characteristic dissipation in
Subsection~\ref{subsec:interface-2015-coupling}, the normalization is~\cite{Merriam1989,CarpenterEtAl2014}
\begin{equation*}
  \frac{\p\bQ}{\p\bW}
  =\bm Y_f\bm Y_f^{\transpose},
  \qquad
  \frac{\p\bm F_n^{(I)}}{\p\bW}
  =\bm Y_f\bm\Lambda_f\bm Y_f^{\transpose}.
\end{equation*}
The derivatives and eigensystem are evaluated at the same consistent
interface state. Since $|\bm\Lambda_f|$ contains nonnegative absolute
characteristic speeds, $\bm D_f\succeq\bm0$.

Contract Eq.~\eqref{eq:semidiscrete-conservative-equation} with
$\bm w_{T_w,K}^{\transpose}\mathbb P_K$ and use the auxiliary-gradient
relation in Eq.~\eqref{eq:discrete-entropy-gradient}. The element
viscous-volume contribution is
\begin{equation}
  \Dcal_{K,h}^{(V)}
  =
  \bm\theta_{i,K}^{\transpose}
  \mathbb P_K[\bm C_{ij}]_K\bm\theta_{j,K}
  \geq0.
  \label{app:eq:discrete-viscous-dissipation}
\end{equation}
On each face $f\subset\partial K$, write
$\bW_{T_w,f,K}=\bW_{f,K}+\be_E/T_w$ for the adapted one-sided trace
from $K$ and $\widehat{\bW}_{T_w,f}=\widehat{\bW}_f+\be_E/T_w$ for
the adapted numerical trace. All normal fluxes below are oriented
outward from $K$. Applying the SBP identities with the numerical-flux
and auxiliary-trace SATs gives
\cite{CarpenterEtAl2014,ParsaniInterface2015}
\begin{align}
  \frac{\dd}{\dd t}\langle1,S_{T_w,K}\rangle_{K,h}
  +\Dcal_{K,h}^{(V)}
  &=-\sum_{f\subset\partial K}
  \left\langle
    (\bW_{T_w,f,K})^{\transpose}\widehat{\bm F}_{n,K}^{(I)}
    -\Psi_{n,K},1
  \right\rangle_{f,h}
  \notag\\
  &\quad+\sum_{f\subset\partial K}
  \left\langle
    (\bW_{T_w,f,K})^{\transpose}\widehat{\bm F}_{n,K}^{(V)},1
  \right\rangle_{f,h}
  \notag\\
  &\quad+\sum_{f\subset\partial K}
  \left\langle
    (\widehat{\bW}_{T_w,f}-\bW_{T_w,f,K})^{\transpose}
    \bm F_{n,K}^{(V)},1
  \right\rangle_{f,h}.
  \label{app:eq:element-adapted-entropy-contraction}
\end{align}
Here, all traces inside each sum are evaluated on $f$, and
$\Psi_{n,K}=n_{i,K}\Psi_i(\bQ_{f,K})$.
The first face sum combines the inviscid volume contraction and the
conservative SAT. The last sum comes from the auxiliary-gradient lifting.

At an interior face shared by $K^-$ and $K^+$, use the common normal
and trace conventions of Subsection~\ref{subsec:interface-2015-coupling}.
Adding the inviscid terms from the two elements and applying Tadmor's
identity gives Eq.~\eqref{eq:merriam-interface-entropy-identity}.
For the viscous terms, the constant affine shift cancels from trace
differences. The LDG trace satisfies
\begin{equation*}
  \widehat{\bW}_f-\bW^-
  =-\frac{1-\alpha}{2}\bm d_f,
  \qquad
  \widehat{\bW}_f-\bW^+
  =\frac{1+\alpha}{2}\bm d_f.
\end{equation*}
The physical outward viscous flux from $K^+$ is $-\bm F_n^{(V),+}$.
Thus, summing the viscous terms gives
\begin{equation*}
  \begin{aligned}
    &\bm d_f^{\transpose}\widehat{\bm F}_{n,f}^{(V)}
    +(\widehat{\bW}_f-\bW^-)^{\transpose}\bm F_n^{(V),-}
    -(\widehat{\bW}_f-\bW^+)^{\transpose}\bm F_n^{(V),+}
    \\
    &\qquad=\frac12\bm d_f^{\transpose}\bm L_f\bm d_f.
  \end{aligned}
\end{equation*}
Indeed, the trace terms have coefficients $-(1-\alpha)\bm d_f/2$
and $-(1+\alpha)\bm d_f/2$ against the common-normal fluxes.
They cancel the complementary LDG weights in
Eq.~\eqref{eq:ldg-ip-viscous-trace} for every $\alpha\in[-1,1]$,
leaving only the IP term \cite{ParsaniInterface2015}.

The same calculation applies to a conforming periodic face pair with
matched quadrature and opposite outward orientations. Under the convention of
Subsection~\ref{subsec:two-sat-global-balances}, each such pair is
counted once in $\mathcal F_h^{\mathrm{int}}$, with its inviscid and IP
dissipations included in $\Dcal_h^X$.

At an isothermal wall, use the fluid-side trace notation
$\bW_{T_w,f,K}=\bW_{T_w}^-$ and
$\widehat{\bW}_{T_w,f}=\widehat{\bW}_{T_w,w}
=(W_1^-,\bm0,0)^{\transpose}$.
The viscous fluxes have zero mass component, so the last two face sums
combine into the contraction of $\bm r_w$ with the viscous-flux correction
in Eq.~\eqref{eq:generic-discrete-wall-entropy-contribution}.
The CR--SAT and FEE--SAT choices then give
Eqs.~\eqref{eq:cr-wall-entropy-contribution}
and~\eqref{eq:fee-wall-entropy-contribution}, respectively.

\subsection{Nodewise wall implementation sequences}
\label{app:wall-implementation-sequences}

The following sequence describes the nodewise construction of the wall SAT
data. Operations common to both wall formulations are evaluated first.

\paragraph{Operations common to both wall formulations}
\begin{enumerate}
  \item Given the outward unit normal $\bn$, prescribed temperature $T_w$,
  and interior conservative trace $\bQ^-$, recover the primitive and entropy
  variables and check that $\rho^->0$ and $T^->0$.

  \item Construct the inviscid mirror state from
  Eq.~\eqref{eq:inviscid-mirror-state} and the common numerical inviscid
  wall flux from
  Eq.~\eqref{eq:inviscid-entropy-conservative-wall-flux}.

  \item Form $\widehat{\bW}_w$, invert the entropy map to obtain the
  admissible wall state $\widehat{\bQ}_w$, and compute
  $\bm r_w=\bW^--\widehat{\bW}_w$. Determine $p_{n,w}$ from the
  wall-adjacent element and evaluate
  $\bm C_{nn,w}^{-}$ and $\widehat{\bm C}_{nn,w}$.

  \item Form the nodewise auxiliary-gradient SAT data
  $\bm g_{i,w}^{\theta}=-n_i\bm r_w$ and assemble the corresponding
  face vectors as described in
  Subsection~\ref{subsec:isothermal-wall-common-data}.
  After completing the auxiliary-gradient equations, use their
  one-sided traces $\bm\theta_j^-=\bm\theta_{j,f,K}$ to evaluate the
  viscous traction $\bm t^-$, Fourier heat flux $q_n^-$, and physical
  normal viscous flux $\bm F_{n,w}^{(V),-}$.

  \item Choose $\beta_w=0$ for the boundary-entropy-conservative variant or
  $\beta_w>0$ for the boundary-entropy-stable variant.
\end{enumerate}

\paragraph{CR--SAT}
\begin{enumerate}
  \item Form $\bm\Lambda_w^{\mathrm{CR}}\succeq\bm0$ and the
  complete-residual correction
  $\bm M_w^{\mathrm{CR}}
  =-\bm\Lambda_w^{\mathrm{CR}}\bm r_w$.

  \item Construct $\widehat{\bm F}_{n,w}^{(V),\mathrm{CR}}$ from
  Eq.~\eqref{eq:cr-wall-viscous-flux}, form the nodewise conservative
  SAT data $\bm g_w^{q,\mathrm{CR}}$ in
  Eq.~\eqref{eq:cr-complete-wall-sat}, and stack these data into
  $\bm g_{f,K}^{q}$ on each wall face. Apply
  $\mathbb P_K^{-1}\mathbb E_{f,K}\bm g_{f,K}^{q}$ in
  Eq.~\eqref{eq:semidiscrete-conservative-equation}.

  \item Monitor the numerical outward total-energy flux
  $F_{E,n,h}^{\mathrm{CR}}$ and the pointwise $S_{T_w}$ wall contribution
  $-\bm r_w^{\transpose}\bm\Lambda_w^{\mathrm{CR}}\bm r_w$.
\end{enumerate}

\paragraph{FEE--SAT}
\begin{enumerate}
  \item Form the skew matrix $\bm K_w$ from $T^-$ and $\bm t^-$ and
  evaluate the entropy-neutral correction $\bm K_w\bm r_w$.

  \item Using $\bm\Pi_u$, form the momentum-projected matrix
  $\bm\Lambda_{w,u}^{\mathrm{FEE}}\succeq\bm0$ and correction
  $\bm M_{w,u}^{\mathrm{FEE}}
  =-\bm\Lambda_{w,u}^{\mathrm{FEE}}\bm r_w$.

  \item Construct $\widehat{\bm F}_{n,w}^{(V),\mathrm{FEE}}$ from
  Eq.~\eqref{eq:fee-wall-viscous-flux}, form the nodewise conservative
  SAT data $\bm g_w^{q,\mathrm{FEE}}$ in
  Eq.~\eqref{eq:fee-complete-wall-sat}, and stack these data into
  $\bm g_{f,K}^{q}$ on each wall face. Apply
  $\mathbb P_K^{-1}\mathbb E_{f,K}\bm g_{f,K}^{q}$ in
  Eq.~\eqref{eq:semidiscrete-conservative-equation}.

  \item Monitor the pointwise identity residual
  $F_{E,n,h}^{\mathrm{FEE}}-q_n^-$ and the pointwise $S_{T_w}$ wall
  contribution
  $-\bm r_w^{\transpose}
  \bm\Lambda_{w,u}^{\mathrm{FEE}}\bm r_w$.
\end{enumerate}

\section{Closed-form fluxes and detailed pointwise verification}
\label{app:closed-form-and-pointwise-details}

This appendix presents the closed-form entropy-conservative fluxes and the
complete randomized verification procedures summarized in
Section~\ref{sec:results}. The closed-form flux tests do not involve a
tensor-product cell, an LGL grid, or the matrices $\mathsf P_K$ and
$\mathsf D_{i,K}$. Consequently, no solution degree $\psol$ is associated
with those results. Separately, the one-dimensional SBP identities in
Eq.~\eqref{eq:one-dimensional-sbp-property} were checked for
$\psol=1,\ldots,12$.

The tolerances below are acceptance thresholds for scaled equality
residuals, which vanish in exact arithmetic. Each residual is scaled
using the magnitudes of the terms being combined, with a unit floor
in the denominator. These thresholds concern the numerical verification
of algebraic identities.

\subsection{Closed-form entropy-conservative fluxes}
\label{subsec:algebraic-flux-verification}

For positive scalars $a_L$ and $a_R$, define the arithmetic and
logarithmic means by
\begin{equation}
  \overline a
  =
  \frac{a_L+a_R}{2},
  \qquad
  a^{\ln}
  =
  \frac{a_R-a_L}{\log a_R-\log a_L}.
  \label{eq:closed-form-means}
\end{equation}
The continuous extension is $a^{\ln}=a_L$ when $a_L=a_R$. Vector
arithmetic means are taken componentwise.

For the Ismail--Roe (IR) flux~\cite{IsmailRoe2009}, introduce
\begin{equation}
  z_1=\sqrt{\rho/p},
  \qquad
  \bm z_2=z_1\bu,
  \qquad
  z_3=\sqrt{\rho p}.
  \label{eq:ismail-roe-z-variables}
\end{equation}
Using arithmetic means denoted by an overbar and logarithmic means denoted by
$(\cdot)^{\ln}$, define
\begin{align}
  \widehat\bu
  &=\frac{\overline{\bm z}_2}{\overline z_1},
  &
  \widehat\rho
  &=\overline z_1 z_3^{\ln},
  \\
  \widehat p_1
  &=\frac{\overline z_3}{\overline z_1},
  &
  \widehat p_2
  &=\frac{\gamma+1}{2\gamma}\frac{z_3^{\ln}}{z_1^{\ln}}
    +\frac{\gamma-1}{2\gamma}\frac{\overline z_3}{\overline z_1},
  \\
  \widehat a^2
  &=\gamma\frac{\widehat p_2}{\widehat\rho},
  &
  \widehat H
  &=\frac{\widehat a^2}{\gamma-1}
    +\frac12|\widehat\bu|^2.
  \label{eq:ismail-roe-averages}
\end{align}
The hatted quantities are two-state flux auxiliaries. In particular,
$\widehat p_1$ and $\widehat p_2$ are pressure averages, while
$\widehat a^2$ and $\widehat H$ are squared-sound-speed and
specific-total-enthalpy auxiliaries. They need not be the properties
of a single common thermodynamic state.
The normal flux is
\begin{equation}
  \bm F_{n}^{\mathrm{IR}}
  =
  \begin{pmatrix}
    \widehat\rho\,\widehat u_n\\
    \widehat\rho\,\widehat u_n\widehat\bu+\widehat p_1\bn\\
    \widehat\rho\,\widehat u_n\widehat H
  \end{pmatrix},
  \qquad
  \widehat u_n=\widehat\bu\cdot\bn.
  \label{eq:ismail-roe-normal-flux}
\end{equation}

For the Chandrashekar (C) flux~\cite{Chandrashekar2013}, set
\begin{equation}
  \beta=\frac{1}{2RT}=\frac{\rho}{2p},
  \qquad
  \widetilde p=\frac{\overline\rho}{2\overline\beta},
  \qquad
  \overline{|\bu|^2}=\frac{|\bu_L|^2+|\bu_R|^2}{2}.
  \label{eq:chandrashekar-means}
\end{equation}
Here, $\beta$ is the inverse-temperature variable of this flux,
not either penalty parameter $\beta_w$ or $\beta_{\mathrm{int}}$.
The mass, momentum, and total-energy flux components are
\begin{align}
  F_{\rho,n}^{\mathrm C}
  &=\rho^{\ln}\,\overline\bu\cdot\bn,
  \\
  \bm F_{m,n}^{\mathrm C}
  &=F_{\rho,n}^{\mathrm C}\overline\bu+\widetilde p\bn,
  \\
  F_{E,n}^{\mathrm C}
  &=\left[
      \frac{1}{2(\gamma-1)\beta^{\ln}}
      -\frac12\overline{|\bu|^2}
    \right]F_{\rho,n}^{\mathrm C}
    +\overline\bu\cdot\bm F_{m,n}^{\mathrm C}.
  \label{eq:chandrashekar-normal-flux}
\end{align}
The subscript $m$ in $\bm F_{m,n}^{\mathrm C}$ labels momentum,
whereas the scalar $m=d+2$ denotes the state-vector dimension.
Both fluxes are symmetric and consistent and satisfy the Tadmor identity~\cite{Tadmor2003} for
the canonical entropy pair. By
Eq.~\eqref{eq:adapted-differences-and-potentials}, the same identity holds for
the affine representative $S_{T_w}$.

\subsection{Closed-form flux test}
\label{subsec:algebraic-test-procedure}

Admissible left and right primitive states are sampled independently from
continuous uniform distributions on the following intervals, with fixed
$\gamma$:
\begin{equation}
  \rho\in[0.2,5],
  \qquad
  u_i\in[-2,2],
  \qquad
  T\in[0.2,5],
  \qquad
  \gamma=1.4.
  \label{eq:flux-test-state-ranges}
\end{equation}
The positivity of $\rho$ and $T$ ensures that the corresponding
conservative states belong to the admissible set defined in
Eq.~\eqref{eq:admissible-set}. For each vectorized batch, one
standard-normal vector is normalized and used as the common unit normal
for every state pair in that batch. The cases are $(d,R)=(2,1)$, $(3,1)$,
and $(3,287)$. For each case and each flux, $10{,}000$ state pairs are
evaluated in batches of $2,048$ using double-precision arithmetic. The
deterministic seeds are $20260717$ for Ismail--Roe and $20260718$ for
Chandrashekar. All tests are performed using a code written in
Fortran 2023.

For either flux, the raw and scaled Tadmor residuals are
\begin{align}
  \mathcal R_{\mathrm T}
  &=
  (\bW_R-\bW_L)^{\transpose}
  \bm F_n^{\mathrm{ec}}(\bQ_L,\bQ_R)
  -\left[\Psi_n(\bQ_R)-\Psi_n(\bQ_L)\right],
  \label{eq:tadmor-test-residual}
  \\
  \widehat{\mathcal R}_{\mathrm T}
  &=
  \frac{|\mathcal R_{\mathrm T}|}
  {1+
   \displaystyle\sum_{\ell=1}^{m}
   |(W_{R,\ell}-W_{L,\ell})(F_n^{\mathrm{ec}})_\ell|
   +|\Psi_n(\bQ_R)|+|\Psi_n(\bQ_L)|}.
  \label{eq:scaled-tadmor-test-residual}
\end{align}
Here, $m=d+2$. The equal-state consistency, left/right symmetry, and reflected-wall errors
are
\begin{align}
  \mathcal E_{\mathrm{cons}}
  &=\|\bm F_n^{\mathrm{ec}}(\bQ,\bQ)-\bm F_n^{(I)}(\bQ)\|_\infty,
  \label{eq:flux-consistency-test-error}
  \\
  \mathcal E_{\mathrm{sym}}
  &=\|\bm F_n^{\mathrm{ec}}(\bQ_L,\bQ_R)
       -\bm F_n^{\mathrm{ec}}(\bQ_R,\bQ_L)\|_\infty,
  \\
  \mathcal E_{\mathrm{wall}}
  &=\|\bm F_n^{\mathrm{ec}}(\bQ^-,\bQ^{+,(I)})
       -\widehat{\bm F}_{n,w}^{(I)}\|_\infty.
  \label{eq:reflected-wall-test-error}
\end{align}
The reflected-wall error uses the state pair
$(\bQ^-,\bQ^{+,(I)})$ defined by
Eq.~\eqref{eq:inviscid-mirror-state}.
The scaled identity tolerance is $2\times10^{-12}$. No unscaled tolerance is
imposed because the dimensional magnitudes vary substantially between the
$R=1$ and $R=287$ cases.

Table~\ref{tab:closed-form-fluxes} summarizes the maximum errors. The
\emph{Case} column identifies the spatial dimension and the specific gas
constant $R$, and the \emph{Flux} column identifies the IR or C flux. The
remaining columns report, respectively, the
maximum absolute raw Tadmor residual $|\mathcal R_{\mathrm T}|$, the maximum
scaled residual $\widehat{\mathcal R}_{\mathrm T}$, the maximum equal-state
consistency error $\mathcal E_{\mathrm{cons}}$, and the maximum reflected-wall
flux error $\mathcal E_{\mathrm{wall}}$.

\begin{table}[h]
\centering
\large
\setlength{\tabcolsep}{8pt}
\renewcommand{\arraystretch}{1.22}
\caption{Closed-form flux verification over $10{,}000$ admissible
left/right state pairs per case for the Ismail--Roe (IR) and Chandrashekar (C) fluxes defined in
\ref{subsec:algebraic-flux-verification}. The four error columns are
specified in \ref{subsec:algebraic-test-procedure}.}
\label{tab:closed-form-fluxes}
\begin{tabular}{llcccc}
\toprule
Case & Flux & \shortstack{Max.\\$|\mathcal R_{\mathrm T}|$} & \shortstack{Max.\\$\widehat{\mathcal R}_{\mathrm T}$} & \shortstack{Max.\\$\mathcal E_{\mathrm{cons}}$} & \shortstack{Max.\\$\mathcal E_{\mathrm{wall}}$} \\
\midrule
2D, $R=1$ & IR & $4.09\times10^{-14}$ & $4.97\times10^{-16}$ & $8.53\times10^{-14}$ & $7.02\times10^{-14}$ \\
2D, $R=1$ & C & $2.40\times10^{-14}$ & $4.64\times10^{-16}$ & $5.68\times10^{-14}$ & $2.99\times10^{-14}$ \\
3D, $R=1$ & IR & $8.70\times10^{-14}$ & $4.14\times10^{-16}$ & $1.14\times10^{-13}$ & $6.92\times10^{-14}$ \\
3D, $R=1$ & C & $4.97\times10^{-14}$ & $4.86\times10^{-16}$ & $8.53\times10^{-14}$ & $4.08\times10^{-14}$ \\
3D, $R=287$ & IR & $1.00\times10^{-11}$ & $1.17\times10^{-15}$ & $2.91\times10^{-11}$ & $1.78\times10^{-11}$ \\
3D, $R=287$ & C & $8.19\times10^{-12}$ & $9.51\times10^{-16}$ & $1.46\times10^{-11}$ & $1.37\times10^{-11}$ \\
\bottomrule
\end{tabular}
\end{table}

The minimum scaled residual is zero in all six cases, while the maximum lies
between $4.14\times10^{-16}$ and $1.17\times10^{-15}$. These values are
consistent with roundoff in double-precision arithmetic. The symmetry error is exactly zero
in all six cases.
For $R=1$, the maximum raw, consistency, and reflected-wall errors are of
order $10^{-14}$--$10^{-13}$. For $R=287$, these errors increase to order $10^{-11}$
because the dimensional fluxes and potentials are larger. The scaled
residual remains at roundoff, so this increase does not indicate a loss of
two-point entropy conservation.

\subsection{Detailed balanced wall campaign}
\label{app:balanced-pointwise-details}

The balanced wall campaign checks the local CR--SAT and FEE--SAT identities
without a mesh or time integrator, using the same generated inputs for both
formulations. For each $d\in\{2,3\}$, $5,000$ independent primitive states are
drawn from continuous uniform distributions on
\begin{equation}
  \rho\in[0.25,4],
  \qquad
  u_i\in[-2,2],
  \qquad
  T\in[0.2,5].
  \label{eq:pointwise-random-state-distributions}
\end{equation}
The entries of the full entropy-variable gradient are independent standard
normal random variables. A standard normal vector is normalized to obtain the
unit wall normal. Independently,
\begin{equation}
  T_w\in[0.2,5],
  \qquad
  p_{n,w}\in[0.05,2]
  \label{eq:pointwise-random-wall-distributions}
\end{equation}
are drawn uniformly. The common parameters are $\gamma=1.4$, $R=1$,
$\mu=0.7$, $\kappa=0.9$, and $\zeta=0.2$. Each test uses
$\beta_w=0$ and $\beta_w=0.25$, giving $40{,}000$ wall evaluations in total.
The seed is $20260728$, and all calculations use double-precision
arithmetic. 
All tests are performed using a code written in
Fortran 2023.

For $X\in\{\mathrm{CR},\mathrm{FEE}\}$, set
\begin{equation*}
  \bm\Lambda_w^X
  =
  \begin{cases}
    \bm\Lambda_w^{\mathrm{CR}},
    & X=\mathrm{CR},\\
    \bm\Lambda_{w,u}^{\mathrm{FEE}},
    & X=\mathrm{FEE},
  \end{cases}
\end{equation*}
and let $\widehat{\bm F}_{n,w}^{(V),X}$ denote the corresponding
numerical viscous wall flux. Define the canonical and adapted viscous wall
contributions as
\begin{align}
  \mathcal B_{S,w}^{X}
  &=
  (\widehat{\bW}_w-\bW^-)^{\transpose}
  \bm F_{n,w}^{(V),-}
  +(\bW^-)^{\transpose}\widehat{\bm F}_{n,w}^{(V),X},
  \\
  \mathcal B_w^{(V),X}
  &=
  (\widehat{\bW}_{T_w,w}-\bW_{T_w}^-)^{\transpose}
  \bm F_{n,w}^{(V),-}
  +(\bW_{T_w}^-)^{\transpose}
  \widehat{\bm F}_{n,w}^{(V),X}.
  \label{eq:pointwise-viscous-wall-contribution}
\end{align}
The inviscid, adapted-entropy, and affine-balance residuals are
\begin{align}
  \mathcal R_{I,w}
  &=
  (\bW_{T_w}^-)^{\transpose}\widehat{\bm F}_{n,w}^{(I)}
  -\Psi_n^-,
  \label{eq:pointwise-inviscid-wall-residual}
  \\
  \mathcal R_{A,w}^{X}
  &=
  \mathcal B_w^{(V),X}
  +\bm r_w^{\transpose}\bm\Lambda_w^X\bm r_w,
  \label{eq:pointwise-adapted-wall-residual}
  \\
  \mathcal R_{\mathrm{aff},w}^{X}
  &=
  \mathcal B_w^{(V),X}
  -\mathcal B_{S,w}^{X}
  +\frac{F_{E,n,h}^{X}}{T_w}.
  \label{eq:pointwise-affine-balance-identity}
\end{align}
For the independently evaluated energy identity, define
\begin{align*}
  F_{E,n,h}^{*,X}
  &=
  \begin{cases}
    q_n^- -\bu^-\!\cdot\bm t^-
    -\be_E^{\transpose}\bm M_w^{\mathrm{CR}},
    & X=\mathrm{CR},\\
    q_n^-,
    & X=\mathrm{FEE},
  \end{cases}
  \\
  \mathcal R_{E,w}^{X}
  &=
  F_{E,n,h}^{X}-F_{E,n,h}^{*,X}.
\end{align*}
For any equality residual $\mathcal R=\sum_k a_k$, define
\begin{equation*}
  \widehat{\mathcal R}
  =
  \frac{|\mathcal R|}{1+\sum_k|a_k|}.
\end{equation*}
The summands $a_k$ are the complete signed scalar contributions in
each residual definition, grouped as follows:
\begin{equation*}
  \begin{array}{c|l}
    \text{Residual} & \text{Scalar summands} \\
    \hline
    \mathcal R_{I,w}
    &
    (\bW_{T_w}^-)^{\transpose}\widehat{\bm F}_{n,w}^{(I)},
    \quad -\Psi_n^-
    \\[0.8ex]
    \mathcal R_{A,w}^{X}
    &
    \mathcal B_w^{(V),X},
    \quad \bm r_w^{\transpose}\bm\Lambda_w^X\bm r_w
    \\[0.8ex]
    \mathcal R_{\mathrm{aff},w}^{X}
    &
    \mathcal B_w^{(V),X},
    \quad -\mathcal B_{S,w}^{X},
    \quad F_{E,n,h}^{X}/T_w
    \\[0.8ex]
    \mathcal R_{E,w}^{X}
    &
    F_{E,n,h}^{X},
    \quad -F_{E,n,h}^{*,X}
  \end{array}
\end{equation*}
Each scalar contribution is evaluated before taking its absolute
value in the denominator. Dot products, quadratic forms, and the
previously defined quantities $\mathcal B_w^{(V),X}$,
$\mathcal B_{S,w}^{X}$, and $F_{E,n,h}^{*,X}$ are not expanded
further for this scaling.
An equality test passes when
$\widehat{\mathcal R}\leq 10^{-11}$.
Entropy-stable cases are also checked separately for
$\mathcal B_w^{(V),X}\leq0$.

The additional scaled Fourier-only mismatch is
\begin{equation}
  \widehat{\mathcal M}_{F,w}^{X}
  =
  \frac{|F_{E,n,h}^{X}-q_n^-|}
  {1+|F_{E,n,h}^{X}|+|q_n^-|},
  \qquad
  X\in\{\mathrm{CR},\mathrm{FEE}\}.
  \label{eq:scaled-fourier-energy-difference}
\end{equation}
It measures the finite-resolution difference between the complete numerical
energy flux and the one-sided numerical Fourier heat flux. It is an identity
residual for the FEE--SAT and a mismatch, rather than a vanishing
residual, for the CR--SAT.

The physical viscous flux, inviscid SAT, viscous SAT, one-sided mechanical
work, wall-penalty energy component, Fourier heat flux, minimum penalty
eigenvalue, and entropy-dissipation quadratic are stored separately. Thus, no
expected identity value is inserted into the flux routine before the residual
is formed.

\begin{table}[H]
\centering
\normalsize
\caption{Balanced pointwise verification of CR--SAT and FEE--SAT.
For each $d\in\{2,3\}$, both formulations and both wall
settings use the same $5,000$ random inputs. The first four rows in each
block report maxima of scaled equality residuals. The final row reports the
maximum scaled Fourier-only mismatch
$\widehat{\mathcal M}_{F,w}^{X}$ from
Eq.~\eqref{eq:scaled-fourier-energy-difference}. This quantity measures the
intended finite-resolution CR--SAT energy-flux difference and is an identity
residual for FEE--SAT.}
\label{tab:balanced-pointwise-wall-identities}
\setlength{\tabcolsep}{8pt}
\renewcommand{\arraystretch}{1.10}
\begin{tabularx}{0.98\linewidth}{>{\raggedright\arraybackslash}Xcc}
\toprule
Diagnostic & CR--SAT & FEE--SAT \\
\midrule
\multicolumn{3}{l}{\textbf{Boundary-entropy-conservative wall}} \\
Maximum scaled inviscid residual & $3.36\times10^{-14}$ & $3.36\times10^{-14}$ \\
Maximum scaled adapted-entropy identity residual & $0$ & $5.68\times10^{-14}$ \\
Maximum scaled affine-balance residual & $7.75\times10^{-15}$ & $1.31\times10^{-14}$ \\
Maximum scaled numerical-energy identity residual & $1.79\times10^{-14}$ & $1.17\times10^{-14}$ \\
Maximum scaled Fourier-only mismatch
$\widehat{\mathcal M}_{F,w}^{X}$
& $9.65\times10^{-1}$ & $1.17\times10^{-14}$ \\
\addlinespace[0.35em]
\multicolumn{3}{l}{\textbf{Boundary-entropy-stable wall}} \\
Maximum scaled inviscid residual & $3.36\times10^{-14}$ & $3.36\times10^{-14}$ \\
Maximum scaled adapted-entropy identity residual & $1.84\times10^{-15}$ & $2.46\times10^{-14}$ \\
Maximum scaled affine-balance residual & $3.35\times10^{-15}$ & $1.27\times10^{-14}$ \\
Maximum scaled numerical-energy identity residual & $2.58\times10^{-14}$ & $1.17\times10^{-14}$ \\
Maximum scaled Fourier-only mismatch
$\widehat{\mathcal M}_{F,w}^{X}$
& $9.97\times10^{-1}$ & $1.17\times10^{-14}$ \\
\bottomrule
\end{tabularx}
\end{table}

Table~\ref{tab:balanced-pointwise-wall-identities} compares the two
formulations. All equality residuals satisfy the stated tolerance, and every
entropy-stable test has a nonpositive wall contribution. The maximum
value of $\widehat{\mathcal M}_{F,w}^{X}$ is approximately one for the CR--SAT
over the deliberately broad random-state range and
$1.17\times10^{-14}$ for the FEE--SAT.

\begin{table}[H]
\centering
\normalsize
\setlength{\tabcolsep}{3.8pt}
\renewcommand{\arraystretch}{1.14}
\caption{Dimension-resolved statistics for the balanced pointwise wall verification. Each formulation and wall setting uses the same $5,000$ admissible inputs in each dimension. The column $\beta_w$ distinguishes the boundary-entropy-conservative setting ($0$) from the boundary-entropy-stable setting ($0.25$). Residuals are listed by row. The numerical columns give their maximum (Max.) and average (Avg.) in two and three dimensions. The residual symbols are defined in the text.}
\label{tab:balanced-pointwise-dimensionwise}
\begin{tabular}{lllcccc}
\toprule
\shortstack{Wall\\formulation} & $\beta_w$ & Residual & \shortstack{$d=2$\\Max.} & \shortstack{$d=2$\\Avg.} & \shortstack{$d=3$\\Max.} & \shortstack{$d=3$\\Avg.} \\
\midrule
CR--SAT & $0$ & $\widehat{\mathcal R}_{I,w}$ & $1.55\times10^{-14}$ & $9.62\times10^{-17}$ & $3.36\times10^{-14}$ & $1.31\times10^{-16}$ \\
 &  & $\widehat{\mathcal R}_{A,w}^{X}$ & $0$ & $0$ & $0$ & $0$ \\
 &  & $\widehat{\mathcal R}_{\mathrm{aff},w}^{X}$ & $6.40\times10^{-15}$ & $8.18\times10^{-17}$ & $7.75\times10^{-15}$ & $8.41\times10^{-17}$ \\
 &  & $\widehat{\mathcal R}_{E,w}^{X}$ & $8.52\times10^{-15}$ & $8.10\times10^{-17}$ & $1.79\times10^{-14}$ & $8.27\times10^{-17}$ \\
\addlinespace[0.25em]
CR--SAT & $0.25$ & $\widehat{\mathcal R}_{I,w}$ & $1.55\times10^{-14}$ & $9.62\times10^{-17}$ & $3.36\times10^{-14}$ & $1.31\times10^{-16}$ \\
 &  & $\widehat{\mathcal R}_{A,w}^{X}$ & $1.52\times10^{-15}$ & $6.09\times10^{-17}$ & $1.84\times10^{-15}$ & $6.72\times10^{-17}$ \\
 &  & $\widehat{\mathcal R}_{\mathrm{aff},w}^{X}$ & $3.18\times10^{-15}$ & $5.98\times10^{-17}$ & $3.35\times10^{-15}$ & $5.96\times10^{-17}$ \\
 &  & $\widehat{\mathcal R}_{E,w}^{X}$ & $8.43\times10^{-15}$ & $7.00\times10^{-17}$ & $2.58\times10^{-14}$ & $7.30\times10^{-17}$ \\
\addlinespace[0.25em]
FEE--SAT & $0$ & $\widehat{\mathcal R}_{I,w}$ & $1.55\times10^{-14}$ & $9.62\times10^{-17}$ & $3.36\times10^{-14}$ & $1.31\times10^{-16}$ \\
 &  & $\widehat{\mathcal R}_{A,w}^{X}$ & $2.84\times10^{-14}$ & $4.16\times10^{-17}$ & $5.68\times10^{-14}$ & $1.11\times10^{-16}$ \\
 &  & $\widehat{\mathcal R}_{\mathrm{aff},w}^{X}$ & $1.31\times10^{-14}$ & $1.14\times10^{-16}$ & $1.08\times10^{-14}$ & $1.49\times10^{-16}$ \\
 &  & $\widehat{\mathcal R}_{E,w}^{X}$ & $1.17\times10^{-14}$ & $1.68\times10^{-16}$ & $1.17\times10^{-14}$ & $2.17\times10^{-16}$ \\
\addlinespace[0.25em]
FEE--SAT & $0.25$ & $\widehat{\mathcal R}_{I,w}$ & $1.55\times10^{-14}$ & $9.62\times10^{-17}$ & $3.36\times10^{-14}$ & $1.31\times10^{-16}$ \\
 &  & $\widehat{\mathcal R}_{A,w}^{X}$ & $2.05\times10^{-14}$ & $7.62\times10^{-17}$ & $2.46\times10^{-14}$ & $8.75\times10^{-17}$ \\
 &  & $\widehat{\mathcal R}_{\mathrm{aff},w}^{X}$ & $1.27\times10^{-14}$ & $7.96\times10^{-17}$ & $9.07\times10^{-15}$ & $8.50\times10^{-17}$ \\
 &  & $\widehat{\mathcal R}_{E,w}^{X}$ & $1.17\times10^{-14}$ & $1.68\times10^{-16}$ & $1.17\times10^{-14}$ & $2.17\times10^{-16}$ \\
\addlinespace[0.25em]
\bottomrule
\end{tabular}
\end{table}

Table~\ref{tab:balanced-pointwise-dimensionwise} resolves the results by
spatial dimension, wall formulation, and wall setting. The residual
symbols $\widehat{\mathcal R}_{I,w}$,
$\widehat{\mathcal R}_{A,w}^{X}$,
$\widehat{\mathcal R}_{\mathrm{aff},w}^{X}$, and
$\widehat{\mathcal R}_{E,w}^{X}$ identify the rows. The numerical
columns give the maximum and arithmetic average of the scaled residuals
in each dimension.
Here, $X$ denotes the formulation identified in the first column.
The wall-penalty column gives $\beta_w=0$ for the
boundary-entropy-conservative setting and $\beta_w=0.25$ for the
boundary-entropy-stable setting.

The nonzero average scaled equality residuals are of order
$10^{-17}$--$10^{-16}$, and the maxima remain below $5.7\times10^{-14}$.

\section{Thermal-relaxation supplementary tests}
\label{app:thermal-relaxation-details}

This appendix gives the stage-integrated energy check and the complete
four-case results for the thermal-relaxation problem specified in
Subsection~\ref{subsec:thermal-relaxation-comparison}.

\paragraph{Stage-integrated energy values}
Let $t_n$ be the start of accepted time step $n$, and set
$\Delta t_n=t_{n+1}-t_n$. Let $\mathcal E_h^n$ denote the discrete
fluid energy evaluated at the numerical state at $t_n$. The superscript
$(n,s)$ denotes evaluation at stage $s$ of that step. The RK4
weights are
\begin{equation*}
  b_1=b_4=\frac16,
  \qquad
  b_2=b_3=\frac13.
\end{equation*}
The superscript $\mathrm{RK}$ identifies Runge--Kutta stage integration.
Using the wall-flux sums defined in
Eq.~\eqref{eq:thermal-relaxation-energy-quantities}, the complete
energy residual and Fourier-only mismatch over one step are
\begin{align}
  \mathcal R_{E,n}^{\mathrm{RK},X}
  &=
  \mathcal E_h^{n+1}-\mathcal E_h^n
  +\Delta t_n\sum_{s=1}^{4}b_s
  \mathcal F_{E,h}^{X,(n,s)},
  \\
  \mathcal M_{F,n}^{\mathrm{RK},X}
  &=
  \mathcal E_h^{n+1}-\mathcal E_h^n
  +\Delta t_n\sum_{s=1}^{4}b_s
  \mathcal Q_{F,h}^{(n,s)}.
  \label{eq:thermal-relaxation-stage-diagnostics}
\end{align}
These quantities are energy increments, not rates. In dimensional
variables, their units are those of energy.

The functional $\mathcal E_h$ is linear in the conservative nodal state.
Consequently, applying it to the Runge--Kutta update gives, in exact
arithmetic,
\begin{equation}
  \mathcal R_{E,n}^{\mathrm{RK},X}
  =
  \Delta t_n\sum_{s=1}^{4}b_s
  \mathcal R_{E,h}^{X,(n,s)}.
  \label{eq:thermal-relaxation-rk-energy-identity}
\end{equation}
Thus, the complete stage-integrated energy balance closes when the
spatial energy balance closes at every stage. This property uses the
same stage fluxes and weights as the solution update and it does not
require exact time integration of the continuous wall flux. The
Fourier-only quantity is a balance residual for the FEE--SAT and a
finite-resolution mismatch for the CR--SAT. 

Unlike total energy, the canonical and adapted entropies are nonlinear
functions of the conservative variables, so the RK4 update used here
does not automatically satisfy an analogous stage-integrated entropy
identity. Relaxation Runge--Kutta methods can enforce a fully discrete
balance for a selected convex entropy by modifying the solution update,
under the conditions of Ref.~\cite{RanochaEtAl2020}. No relaxation
correction is used in the present thermal-relaxation calculations.

\paragraph{Four-case summary and full histories}
Table~\ref{tab:thermal-relaxation-balance-summary} reports maximum
absolute instantaneous values over the recorded states
and maximum absolute stage-integrated values over accepted steps.
It also gives the final fluid-energy change
$\Delta\mathcal E_h=\mathcal E_h(t_f)-\mathcal E_h(0)$.
The instantaneous quantities use the independent evaluations described
in Subsection~\ref{subsec:balance-diagnostics}.

\begin{table}[H]
\centering
\normalsize
\caption{Thermal-relaxation balance tests for the four wall
settings in Subsection~\ref{subsec:thermal-relaxation-comparison}.
Maximum instantaneous entries are maxima of absolute values over the
recorded states. Maximum stage-integrated entries are
maxima of absolute values over accepted time steps and use the
Runge--Kutta stage weights of the solution update. The former are rate
diagnostics, whereas the latter are energy increments. The final
fluid-energy change $\Delta\mathcal E_h$ is signed.}
\label{tab:thermal-relaxation-balance-summary}
\setlength{\tabcolsep}{8pt}
\renewcommand{\arraystretch}{1.10}
\begin{tabularx}{0.98\linewidth}{>{\raggedright\arraybackslash}Xcc}
\toprule
Diagnostic & CR--SAT & FEE--SAT \\
\midrule
\multicolumn{3}{l}{\textbf{Boundary-entropy-conservative wall}} \\
Maximum instantaneous complete energy residual
& $6.94\times10^{-18}$ & $6.94\times10^{-18}$ \\
Maximum instantaneous Fourier-only mismatch
& $3.76\times10^{-9}$ & $6.94\times10^{-18}$ \\
Maximum stage-integrated complete energy residual
& $9.85\times10^{-16}$ & $1.02\times10^{-15}$ \\
Maximum stage-integrated Fourier-only mismatch
& $1.53\times10^{-13}$ & $1.02\times10^{-15}$ \\
Maximum instantaneous adapted-entropy residual
& $6.82\times10^{-18}$ & $6.06\times10^{-18}$ \\
Final fluid-energy change $\Delta\mathcal E_h$
& $-9.26\times10^{-4}$ & $-9.26\times10^{-4}$ \\
\addlinespace[0.35em]
\multicolumn{3}{l}{\textbf{Boundary-entropy-stable wall}} \\
Maximum instantaneous complete energy residual
& $6.94\times10^{-18}$ & $6.94\times10^{-18}$ \\
Maximum instantaneous Fourier-only mismatch
& $1.78\times10^{-4}$ & $6.94\times10^{-18}$ \\
Maximum stage-integrated complete energy residual
& $1.04\times10^{-15}$ & $1.14\times10^{-15}$ \\
Maximum stage-integrated Fourier-only mismatch
& $7.23\times10^{-9}$ & $1.14\times10^{-15}$ \\
Maximum instantaneous adapted-entropy residual
& $7.98\times10^{-18}$ & $5.33\times10^{-18}$ \\
Final fluid-energy change $\Delta\mathcal E_h$
& $-9.24\times10^{-4}$ & $-9.26\times10^{-4}$ \\
\bottomrule
\end{tabularx}
\end{table}

Figures~\ref{fig:thermal-relaxation-energy-terms-full}--
\ref{fig:thermal-relaxation-adapted-entropy-full} give the complete
energy-rate, energy-residual, Fourier-only mismatch, and adapted-entropy
histories for both wall formulations and both wall-penalty settings.

\begin{figure}[!htbp]
  \centering
  \includegraphics[width=\linewidth]
    {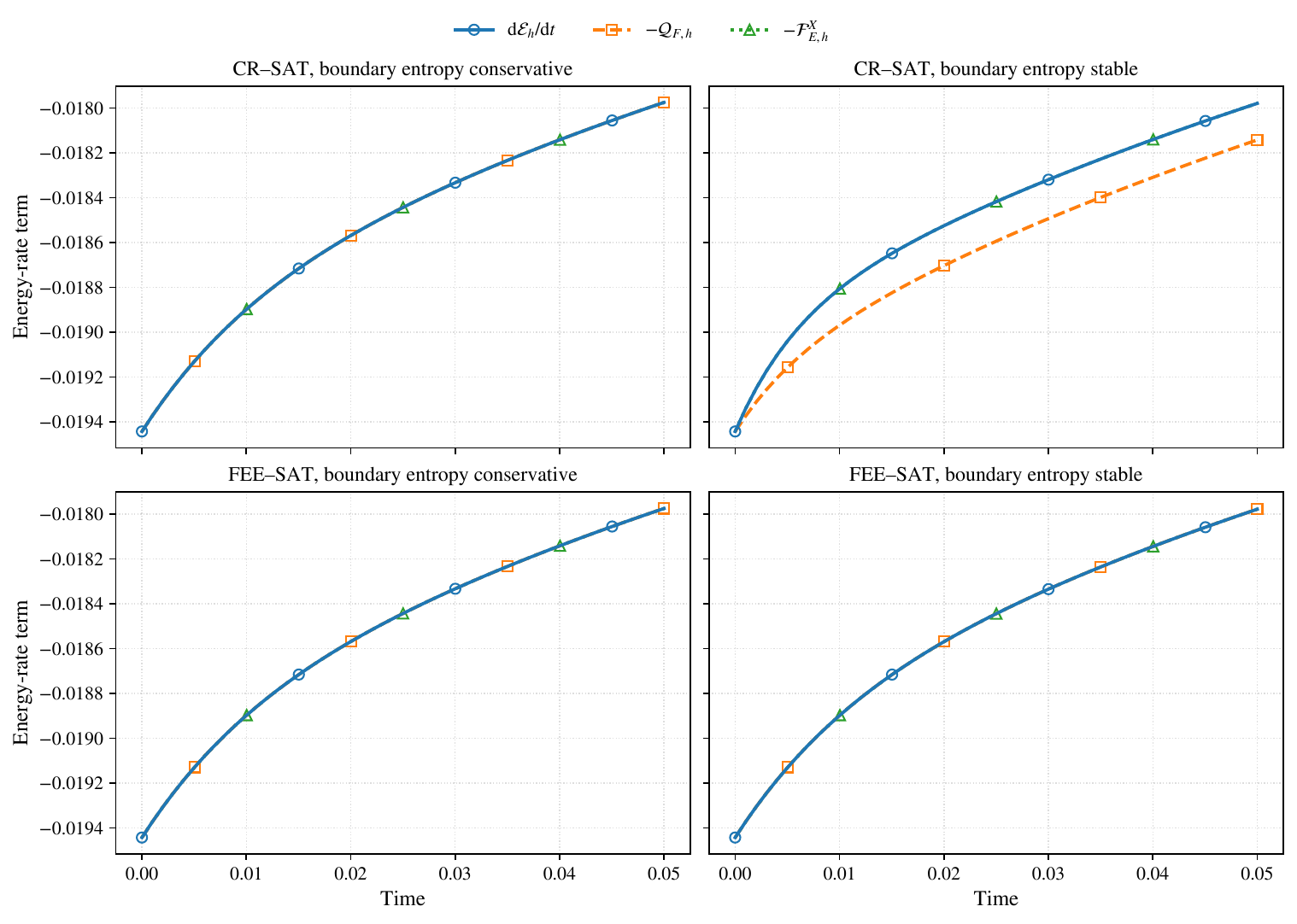}
  \caption{Four-case thermal-relaxation energy-rate histories. The
  upper and lower rows correspond to the CR--SAT and the FEE--SAT,
  respectively. The left and right columns use $\beta_w=0$ and
  $\beta_w=0.25$. These settings distinguish the wall treatment.
  All four calculations retain the same entropy-stable interior
  discretization. The curves compare
  $\mathrm d\mathcal E_h/\mathrm dt$, $-\mathcal Q_{F,h}$, and
  $-\mathcal F_{E,h}^{X}$, evaluated as described in
  Subsection~\ref{subsec:balance-diagnostics}. The flux
  signs are reversed only for comparison with the energy rate.
  Overlap at this scale does not resolve the balance residual or
  Fourier-only mismatch. These are shown in
  Figure~\ref{fig:thermal-relaxation-energy-defects-full} and
  summarized in Table~\ref{tab:thermal-relaxation-balance-summary}.}
  \label{fig:thermal-relaxation-energy-terms-full}
\end{figure}

\begin{figure}[!htbp]
  \centering
  \includegraphics[width=\linewidth]
    {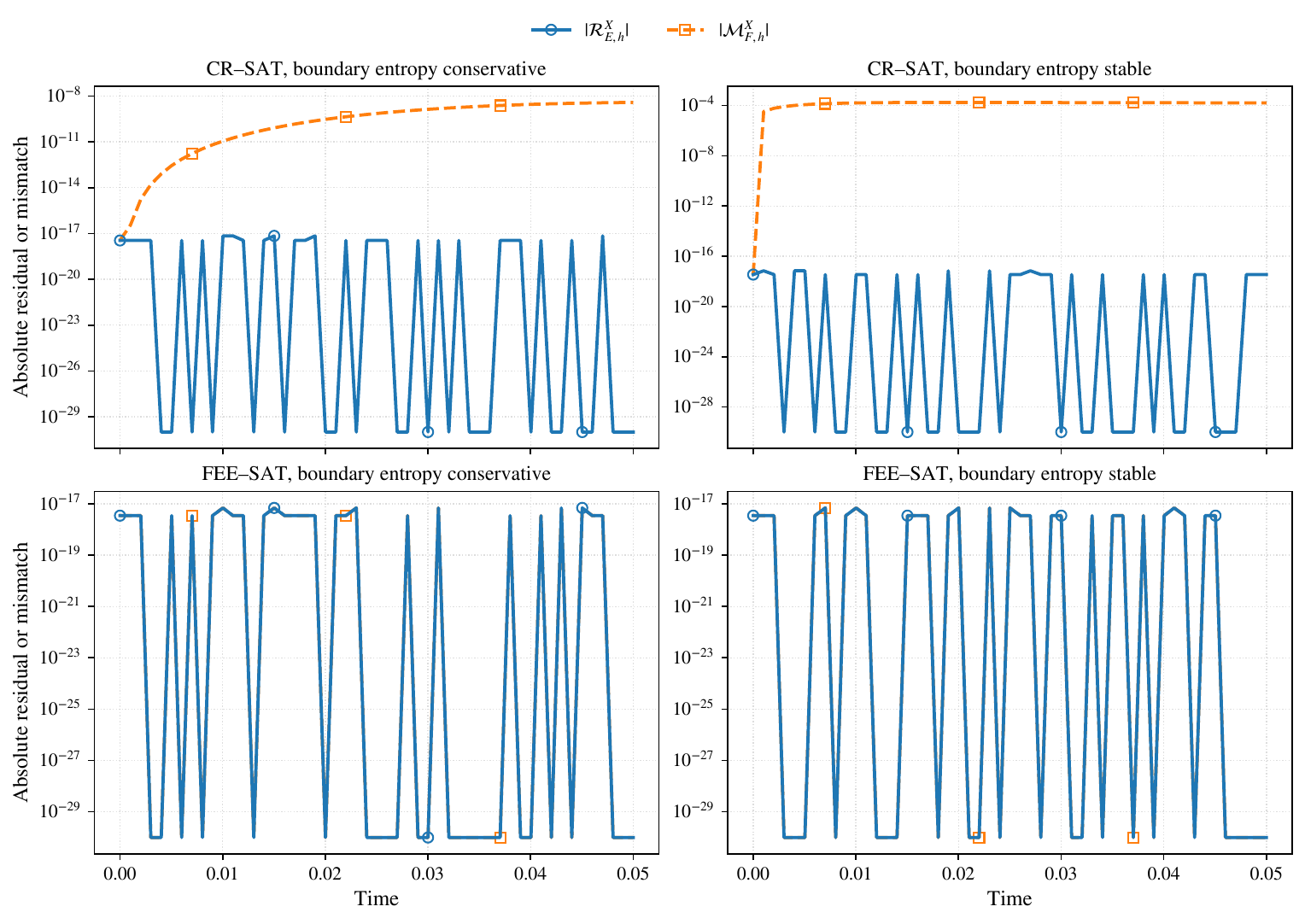}
  \caption{Four-case thermal-relaxation instantaneous energy
  tests. The upper and lower rows correspond to the CR--SAT and
  the FEE--SAT, respectively. The left and right columns use $\beta_w=0$
  and $\beta_w=0.25$. The curves show the unscaled absolute complete
  energy residual $|\mathcal R_{E,h}^{X}|$ and Fourier-only mismatch
  $|\mathcal M_{F,h}^{X}|$ on logarithmic axes. The complete balance
  closes at roundoff scale for all four calculations. The two
  values coincide for the FEE--SAT, whereas a nonzero Fourier-only
  mismatch is permitted for the CR--SAT. These are instantaneous rate
  values, not the stage-integrated energy increments in
  Eq.~\eqref{eq:thermal-relaxation-stage-diagnostics}.}
  \label{fig:thermal-relaxation-energy-defects-full}
\end{figure}

\begin{figure}[!htbp]
  \centering
  \includegraphics[width=\linewidth]
    {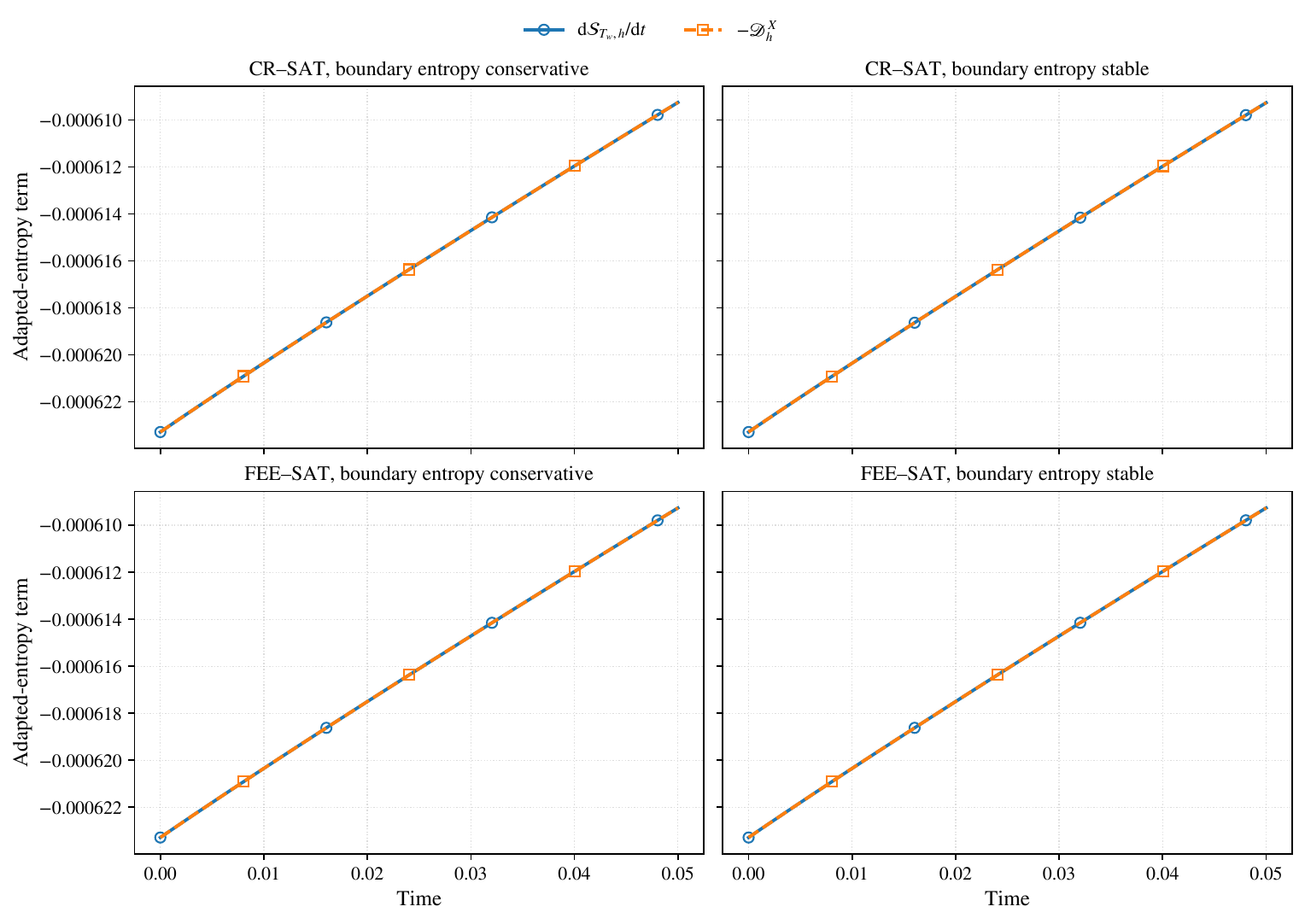}
  \caption{Four-case thermal-relaxation adapted-entropy-rate
  histories. The upper and lower rows correspond to the CR--SAT and the
  FEE--SAT, respectively. The left and right columns use $\beta_w=0$
  and $\beta_w=0.25$. The curves compare
  $\mathrm d\mathcal S_{T_w,h}/\mathrm dt$, evaluated using
  Eq.~\eqref{eq:balance-adapted-entropy-rate}, with the separately
  evaluated $-\Dcal_h^X$. Here, $\Dcal_h^X$ is the total dissipation,
  including the viscous-volume and all active numerical
  contributions, not only $\Dcal_h^{(V)}$. The sign reversal makes
  closure of the adapted-entropy balance appear as overlapping
  curves. The first plotted quantity minus the second is
  $\mathcal R_{S_{T_w},h}^{X}$. This residual is not plotted here.
  Its maximum absolute value is reported in
  Table~\ref{tab:thermal-relaxation-balance-summary}.}
  \label{fig:thermal-relaxation-adapted-entropy-full}
\end{figure}

\section{Discrete error measures and supplementary convergence data}
\label{app:constant-Tw-convergence-data}

\subsection{SBP-quadrature error norms and observed rates}
\label{app:constant-Tw-error-norms}

Let \(\varphi\in\{\rho,u,T\}\), and let
\begin{equation}
  \bm e_{\varphi,K}(t_f)
  =
  \bm\varphi_{h,K}(t_f)
  -
  \bm\varphi_{\mathrm{ex},K}(t_f)
  \label{app:eq:constant-Tw-nodal-error}
\end{equation}
be the vector of errors at the LGL solution nodes of cell \(K\). The
reported errors are evaluated at the final time \(t_f=0.5\). They are not
maxima over \([0,t_f]\). Define the discrete domain measure
\begin{equation}
  |\Omega|_h
  =
  \sum_{K\in\mathcal T_h}
  \bm1_K^{\transpose}\mathsf P_K\bm1_K.
  \label{app:eq:constant-Tw-discrete-domain-measure}
\end{equation}
For the affine meshes used here, \(|\Omega|_h=|\Omega|=1\) to roundoff. The
normalized SBP-quadrature errors are
\begin{align}
  \lVert e_\varphi\rVert_{1,h}
  &=
  \frac{1}{|\Omega|_h}
  \sum_{K\in\mathcal T_h}
  \bm1_K^{\transpose}
  \mathsf P_K
  \left|\bm e_{\varphi,K}\right|,
  \label{app:eq:constant-Tw-discrete-L1}
  \\
  \lVert e_\varphi\rVert_{2,h}
  &=
  \left[
    \frac{1}{|\Omega|_h}
    \sum_{K\in\mathcal T_h}
    \bm e_{\varphi,K}^{\transpose}
    \mathsf P_K
    \bm e_{\varphi,K}
  \right]^{1/2},
  \label{app:eq:constant-Tw-discrete-L2}
  \\
  \lVert e_\varphi\rVert_{\infty,h}
  &=
  \max_{K\in\mathcal T_h}
  \max_{1\leq a\leq\Nsol}
  \left|e_{\varphi,K,a}\right|.
  \label{app:eq:constant-Tw-discrete-Linf}
\end{align}
Equations~\eqref{app:eq:constant-Tw-discrete-L1}--
\eqref{app:eq:constant-Tw-discrete-Linf} follow the discrete SBP definitions
used in Ref.~\cite[Eq.~(4.6)]{RanochaEtAl2020}. The absolute value in
\eqref{app:eq:constant-Tw-discrete-L1} is applied componentwise. The
\(L_\infty\) value is the maximum absolute error over all element-local LGL
nodes. Note that duplicated interface nodes are included as distinct DG degrees of
freedom.

For each fixed variable, norm, solution degree, and wall formulation,
observed rates are computed from successive errors using
Eq.~\eqref{eq:matched-mms-observed-rate}. The rates are assigned to the
finer level and no rate is reported at $\ell=0$.

\subsection{Supplementary convergence data}
\label{app:constant-Tw-detailed-tables}

Supplementary Material~S1 collects the complete paired convergence
results: Section~1 gives the $L_1$, $L_2$, and nodal $L_\infty$ volume
errors and rates for the three configurations in
Eq.~\eqref{eq:numerical-configuration-definitions} (Tables~S1--S9)
and the density and velocity $L_2$ plots (Figures~S1 and~S2).
Section~2 gives the SBP matrix definitions, values, and rates of the
annular-pipe wall errors (Tables~S10--S11).

\section{Reference conventions for multidimensional tests}
\label{app:multidimensional-configurations}

The reference conventions below apply to the three-dimensional tests in
Subsection~\ref{subsec:multidimensional-evidence}. The configurations,
run details, and figures are presented with the corresponding results
in that subsection.

\subsection{Reference scales}
\label{app:multidimensional-reference-scales}

Let $L_0$, $\rho_0$, $U_0$, and $T_0$ denote the reference length, density,
velocity, and temperature. Dimensional variables, marked by the superscript
$\mathrm{dim}$, are scaled according to
\begin{equation*}
  \bm x
  =
  \frac{\bm x^{\mathrm{dim}}}{L_0},
  \qquad
  t
  =
  \frac{t^{\mathrm{dim}}U_0}{L_0},
  \qquad
  \rho
  =
  \frac{\rho^{\mathrm{dim}}}{\rho_0},
  \qquad
  \bu
  =
  \frac{\bu^{\mathrm{dim}}}{U_0},
  \qquad
  T
  =
  \frac{T^{\mathrm{dim}}}{T_0}.
\end{equation*}
In the following reference-parameter definitions, $R$ and $c_p$ are the
dimensional specific gas constant and specific heat capacity, respectively.
With $a_0=\sqrt{\gamma R T_0}$, the reference parameters are
\begin{equation*}
  \mathrm{Ma}
  =
  \frac{U_0}{a_0},
  \qquad
  \mathrm{Re}
  =
  \frac{\rho_0U_0L_0}{\mu_0},
  \qquad
  \mathrm{Pr}
  =
  \frac{\mu_0c_p}{\kappa_0},
\end{equation*}
where $\mu_0$ and $\kappa_0$ are the reference viscosity and thermal
conductivity. The solution variables and spatial and temporal
coordinates in the multidimensional test descriptions are nondimensional
unless explicitly identified as dimensional. Reference quantities retain
their dimensional units.

\section*{Acknowledgments}

This work was funded by King Abdullah University of Science and Technology
(KAUST) under grant number BAS/1/1663-01-01. The authors gratefully
acknowledge the KAUST Supercomputing Laboratory for providing computational
resources on the central processing unit (CPU) partition of the Shaheen III
supercomputer.

The authors also acknowledge the use of Mathematica and SymPy. Grammarly
was used for three rounds of language and grammar checks before
submission. Its generative-AI features were not used. The \LaTeX{} source
files, including the equation environments, were organized and formatted
using \texttt{vim-easy-align} and \texttt{VimTeX}. 

\section*{Data availability}

The numerical data, verification and post-processing tools, figure and
table generators, and reproducibility instructions supporting this study
will be made available upon request. The one-dimensional simulations were
performed using an in-house code that conforms to the Fortran 2023
standard and whose development began in 2014. The three-dimensional
simulations were performed using SSDC, an in-house, mixed-language code
written in Fortran, C, and C++, with input files in the YAML format.

\section*{Declaration of competing interest}
The authors declare no known competing financial interests or personal
relationships that could have appeared to influence the work reported in this
paper.

\section*{A decade of work and a changing landscape of discovery}

The foundations of this work go back to 2014 and to the formative
mentorship of Mark H. Carpenter, an intellectual heir to David I.
Gottlieb in the tradition of the Institute for Computer Applications in
Science and Engineering (ICASE). Mark taught us entropy theory, entropy
stability, and boundary conditions, insisting that every step be
understood and every claim justified. His mentorship shaped not only
our mathematical understanding but also the standards we would hold
ourselves to as scientists.

This work began to close a circle opened in our student years, among
ICASE reports, foundational papers, and monographs on CFD. The writings
of Heinz-Otto Kreiss, Bram van Leer, Amiram (Ami) Harten, Stanley J. Osher,
Antony Jameson, Philip Lawrence Roe, David I. Gottlieb, Eitan Tadmor,
Alfio M. Quarteroni, Claudio Canuto, Mark H. Carpenter, Jan Nordstr\"{o}m,
Chi-Wang Shu, George Em Karniadakis, Randall J. LeVeque, and their
collaborators had become patient but exacting companions: texts
revisited, pencil in hand, until an argument began to yield its meaning.
These connections took shape at the huge green chalkboards in the
aerosciences branch at NASA Langley. Side by side with Mark, we derived
and rederived, questioned, corrected, and began again. Mistakes and steps
we could not explain sent us back to the board; working through them,
again and again, was how understanding deepened into mastery. By day's
end, chalk dust covered our shoes and clothes. Like athletes refining
fundamental movements, we practiced, practiced, and practiced until the
basics became dependable under pressure. The aim was not to memorize a
derivation, but to understand how to arrive at it, reconstruct its steps,
explain its assumptions, and recognize its limits: to build foundations
strong enough for independent thought. \textit{What would later feel
like intuition was taking shape there, day after day, through mistakes,
repetition, and chalk.}

Looking back, we wonder how differently we would approach this problem
if we were starting today. The work presented here took shape and
reached completion through discussions with colleagues, partial
differential equation theory,
calculus, linear algebra, and numerical experimentation. From the
initial ideas through the completed mathematical analysis, numerical
software, and their verification, we did not consult a generative
artificial intelligence (AI) system. Grammarly was used to check the text in three rounds,
interspersed with human reading and editing. Its generative-AI features
were not used. This is not a claim that slower research is better
research. Yet the time spent struggling, returning to an argument, and
discovering why an approach failed did more than produce a result.
It helped develop the understanding needed to recognize why the
result works.

We should not romanticize the difficulties faced by
earlier generations. Young scientists should not have to repeat every
derivation, inefficient calculation, frustrating implementation, or avoidable
mistake simply because their predecessors did. Removing unnecessary
obstacles is progress. The challenge is to distinguish those obstacles
from the intellectual effort through which judgment develops.

Publication practices have to evolve accordingly. A convincing account
of a discovery should make its assumptions, evidence, computational
procedures, and contributions explicit. Accessible data, checkable
arguments, and transparent descriptions of AI use should become more
central, not less. We should resist measuring progress simply by the
number of papers that increasingly powerful systems make it possible
to produce. A paper's value is not easily
quantified. It lies in the reliable, useful human understanding it
contributes and in what it enables others to understand and discover
long after publication.

Completing this work renews our gratitude to our human mentors, collaborators,
and students and reminds us how much we still have to learn. What we
carry from that apprenticeship into the age of AI is the willingness to
stay with a question after an answer has appeared. We want to bring that
habit to problems we once could not have attempted. At a time of
uncertainty, we find ourselves thinking about the world our students,
young people everywhere, and generations yet to come will inherit. We
hope to pass on what our mentors gave us: the patience and trust to find
our own way, and the conviction that careful, honest work remains worth
doing, even when we cannot yet see where it will lead. That inheritance
also carries a duty we take personally: to pursue science for the common
good of all.

\textit{The chalk and pencil marks fade; the responsibility to
understand deeply and question remains.}

\setlength{\bibsep}{7pt plus 1pt minus 1pt}
\bibliographystyle{elsarticle-num}
\bibliography{references}

\end{document}